\documentclass[12pt]{article}

\usepackage{amssymb}
\usepackage{amsmath}
\usepackage{amsthm}
\usepackage{txfonts}
\usepackage{enumerate}
\usepackage[dvipdfmx]{color}
\usepackage[top=30truemm, bottom=30truemm, left=20truemm, right=20truemm]{geometry}

\newtheorem{thm}{Theorem}[section]
\newtheorem{lem}{Lemma}[section]
\newtheorem{cor}{Corollary}[section]
\newtheorem{rem}{\rm REMARK}[section]
\newtheorem{prop}{Proposition}[section]

\title{{\huge Construction and sample path properties of Brownian house-moving between two curves}}
\author{Kensuke Ishitani
\thanks{Department of Mathematical Sciences, Tokyo Metropolitan University, Hachioji, Tokyo 192-0397, Japan} 
\thanks{E-mail: k-ishitani@tmu.ac.jp
\qquad https://orcid.org/0000-0001-8400-0654}
\and 
Daisuke Hatakenaka 
\thanks{JAPAN PROCESS DEVELOPMENT CO., LTD., Shinagawa, Tokyo 141-0032, Japan}
\and 
Keisuke Suzuki
\thanks{NEC Corporation, Kawasaki, Kanagawa 211-8666, Japan}}
\date{}
\begin{document}
\maketitle

\begin{abstract}
This study aims to construct a stochastic process called 
``Brownian house-moving,'' which is a Brownian bridge 
conditioned to stay between two curves. 
To construct this process, statements are prepared 
on the weak convergence of conditioned Brownian motions, 
conditioned Brownian bridges, and 
conditioned three-dimensional Bessel bridges.
Moreover, the sample path properties of Brownian house-moving are studied as well.   

\end{abstract}

\section{Introduction}

Recently, \cite{bib_Ishitani} developed 
a chain rule for Wiener path integrals between two curves 
that arise in the computation of first-order Greeks for barrier options, 
and demonstrated the effectiveness of this chain rule through numerical examples. 
In this chain rule, Brownian meander and BES($3$)-bridge conditioned to stay between two curves 
played an important role. 
Furthermore, we are currently investigating higher-order chain rules 
for computing higher-order Greeks for barrier options, 
and we expect a stochastic process called ``Brownian house-moving''  
to play an important role in their computation. 
A Brownian house-moving is defined as a Brownian bridge 
conditioned to stay between two curves. 
The purpose of this study is to construct these stochastic processes.

The remainder of this paper is organized as follows. 
In Section~\ref{sec_notation}, we present the notation used in this study. 
Section~\ref{section_Mainresults} states the main results of this study. 
In Subsection~\ref{subsection_Moving_between_2curves}, 
we construct the Brownian house-moving (Theorem~\ref{Thm_Def_and_Decomp_curved_Moving}).
In addition, the sample path properties of 
Brownian house-moving (Corollaries~\ref{Cor_30_curve_max_0t}, \ref{Cor_30_curve_min_t1} and Theorem~\ref{Thm_abs_conti}) 
are provided in this subsection. 
In Subsection~\ref{subsection_Meander_between_2curves}, 
we construct the Brownian meander between two curves. 
In Subsection~\ref{subsection_BES3bridge_between_2curves}, 
we construct the BES($3$)-bridge between two curves. 
Sections~\ref{section_proof_main_Lemma}, 
\ref{section_proof_Preparation_for_Proofs}, 
\ref{section_proof_Moving_between_2curves}, 
\ref{section_proof_Meander_between_2curves}, and
\ref{section_proof_3dBesselbridge_between_2curves}
are devoted to proving the main results in Section~\ref{section_Mainresults}.
Section~\ref{Sec_conclusion_future_work} presents an application of Brownian house-moving and concludes.

\section{Notation}\label{sec_notation}

For $0\leq s < t\leq 1$, let $C([s,t], \mathbb{R})$ 
be the class of $\mathbb{R}$-valued continuous functions 
defined on $[s,t]$, and let 
$$d_{\infty}(w,w') = \sup_{u\in [s,t]} \left| w(u) - w'(u) \right|
\quad (w, w' \in C([s,t], \mathbb{R})). $$ 
$\mathcal{B}(C([s,t], \mathbb{R}))$ denotes the Borel $\sigma$-algebra 
with respect to the topology generated by the metric $d_{\infty}$. 
In addition, for $0 \leq s < t \leq 1$, 
$\pi_{[s, t]} : C([0,1], \mathbb{R}) \to C([s,t], \mathbb{R})$ 
denotes the restriction map.

Assume that 
$Y : (\Omega, \mathcal{F}, P) \to 
(C([0,1], \mathbb{R}), \mathcal{B}(C([0,1], \mathbb{R})))$ 
is a random variable and 
$\Lambda \in \mathcal{B}(C([0,1], \mathbb{R}))$ 
satisfies $P(Y \in \Lambda) > 0$. 
Then, we define the probability measure 
$P_{Y^{-1}(\Lambda)}$ on 
$(Y^{-1}(\Lambda), Y^{-1}(\Lambda) \cap \mathcal{F})$ as
\begin{align*}
P_{Y^{-1}(\Lambda)}(A):= \frac{P(A)}{P(Y \in \Lambda)}, 
\qquad A \in Y^{-1}(\Lambda) \cap \mathcal{F}
:= \left\{ Y^{-1}(\Lambda) \cap F~|~F \in \mathcal{F} \right\}.
\end{align*}
Let $Y|_{\Lambda}$ denote the restriction $Y$ to 
$(Y^{-1}(\Lambda), Y^{-1}(\Lambda) \cap \mathcal{F}, P_{Y^{-1}(\Lambda)})$. 
Then, 
\begin{align}\label{Def_conditional_StochProc}
Y|_{\Lambda} : (Y^{-1}(\Lambda), Y^{-1}(\Lambda) \cap \mathcal{F}, P_{Y^{-1}(\Lambda)}) \to 
(\Lambda, \mathcal{B}(\Lambda))
\end{align}
is a random variable.
Throughout this study, 
$P_{Y^{-1}(\Lambda)}(Y|_{\Lambda} \in \Gamma)$ is often written as 
$P(Y|_{\Lambda} \in \Gamma)$, and
$E^{P_{Y^{-1}(\Lambda)}}[f(Y|_{\Lambda})]$ is often written as $E[f(Y|_{\Lambda})]$.

For $s>0$, we define
\begin{align*}
n_s(x) := \frac{1}{\sqrt{2 \pi s}} \exp \left( -\frac{x^2}{2s} \right) \qquad (x\in \mathbb{R}).
\end{align*}

$X_n \overset{\mathcal{D}}{\to} X$ denotes the convergence in distribution 
of the sequence of random variables $\{ X_n \}_{n=1}^{\infty}$ to the random variable $X$.  
In addition, we write $X \overset{\mathcal{D}}{=} Y$ 
for random variables $X, Y$ that follow the same distribution.

Let $0 \leq t_1 < t_2 \leq 1$. 
Throughout this study, we use the following notation. 

For $f, g \in C([0, 1], \mathbb{R})$, we define
\begin{align*}
&K_{[t_1,t_2]}(f,g) := 
\{ w = \{ w(t) \}_{t \in [t_1,t_2]} \in C([t_1,t_2], \mathbb{R} )
~|~f(t) \leq w(t) \leq g(t),\ t_1 \leq t \leq t_2 \},\\
&K_{[t_1,t_2]}^+(f) := \bigcup_{n=1}^{\infty} K_{[t_1,t_2]}(f, n),
\qquad K_{[t_1,t_2]}^-(g) := \bigcup_{n=1}^{\infty} K_{[t_1,t_2]}(-n, g), 
\end{align*}
and
\begin{align*}
&K(f,g) := K_{[0,1]}(f,g), \qquad K^+(f) := K_{[0,1]}^+(f), 
\qquad K^-(g) := K_{[0,1]}^-(g), \\
&K_t(f,g) := K_{[0,t]}(f,g), \qquad K_t^+(f) := K_{[0,t]}^+(f),
\qquad K_t^-(g) := K_{[0,t]}^-(g).
\end{align*}

For an $\mathbb{R}$-valued continuous process $X=\{ X(t) \}_{t \in [0,1]}$, 
we write its maximal and minimal values as
\begin{align*}
&M_{[t_1,t_2]}(X) = \max_{t_1 \leq u \leq t_2}X(u), \qquad
M_t(X) = M_{[0,t]}(X) ,\qquad M(X) = M_{[0,1]}(X), \\
&m_{[t_1,t_2]}(X) = \min_{t_1 \leq u \leq t_2}X(u), \qquad
m_t(X) = m_{[0,t]}(X), \qquad m(X) = m_{[0,1]}(X).
\end{align*}
Moreover, the natural filtration 
$\sigma(X(s) \mid 0 \leq s \leq t)$ of $X$ 
is denoted by $\mathcal{F}_t^X$.

$W = \{ W(t) \}_{t \geq 0}$, 
$B^{a \to b} = \{ B^{a \to b}(t) \}_{t \in [0,1]}$ $(a, b \in \mathbb{R})$, 
$W^+ = \{ W^+(t) \}_{t \in [0,1]}$, 
and $r^{c \to d} = \{ r^{c \to d}(t) \}_{t \in [0,1]}$ $(c, d \geq 0)$ denote 
standard one-dimensional Brownian motion, 
one-dimensional Brownian bridge from $a$ to $b$ on the time interval $[0,1]$, 
Brownian meander on the time interval $[0,1]$, and 
BES($3$)-bridge from $c$ to $d$ on the time interval $[0,1]$ 
defined on some probability space, respectively.  
For $a, b \in \mathbb{R}$ and $c, d \geq 0$, 
$W_{[t_1,t_2]}$, $B_{[t_1,t_2]}^{a \to b}$, $W_{[t_1,t_2]}^+$ and $r_{[t_1,t_2]}^{c \to d}$ denote 
one-dimensional Brownian motion, 
one-dimensional Brownian bridge from $a$ to $b$, 
Brownian meander, and 
BES($3$)-bridge from $c$ to $d$ defined on $[t_1,t_2]$, respectively. 
Laws of $W_{[t_1,t_2]}$, $B_{[t_1,t_2]}^{a \to b}$, $W_{[t_1,t_2]}^+$, and $r_{[t_1,t_2]}^{c \to d}$ are given by
\begin{align*}
&\{ W_{[t_1,t_2]}(u) \}_{u \in [t_1, t_2]}
\overset{\mathcal{D}}{=} 
\{ W(u-t_1) \}_{u \in [t_1, t_2]}, \\
&\left\{ B_{[t_1,t_2]}^{a \to b}(u) \right\}_{u \in [t_1, t_2]}
\overset{\mathcal{D}}{=}
\left\{ \sqrt{t_2-t_1}B^{\frac{a}{\sqrt{t_2-t_1}} \to \frac{b}{\sqrt{t_2-t_1}}}\left(\frac{u - t_1}{t_2-t_1}\right) \right\}_{u \in [t_1, t_2]}, \\
&\left\{ W_{[t_1,t_2]}^+(u) \right\}_{u \in [t_1, t_2]}
\overset{\mathcal{D}}{=}
\left\{ \sqrt{t_2-t_1}W^+\left(\frac{u - t_1}{t_2-t_1}\right) \right\}_{u \in [t_1, t_2]}, \\
&\left\{ r_{[t_1,t_2]}^{c \to d}(u) \right\}_{u \in [t_1, t_2]}
\overset{\mathcal{D}}{=}
\left\{ \sqrt{t_2-t_1}r^{\frac{c}{\sqrt{t_2-t_1}} \to \frac{d}{\sqrt{t_2-t_1}}}\left(\frac{u - t_1}{t_2-t_1}\right) \right\}_{u \in [t_1, t_2]}.
\end{align*}

\section{Main results}\label{section_Mainresults}

Let $g^-$ and $g^+$ be 
$\mathbb{R}$-valued $C^2$-functions defined on $[0, 1]$ that satisfy
\begin{align*}
\min_{0 \leq t \leq 1}(g^+(t) - g^-(t)) > 0.
\end{align*} 

Let $0 \leq t_1 < t_2 \leq 1$. 
According to the values $g^-(t_1) \leq a \leq g^+(t_1)$ 
and $g^-(t_2) \leq b \leq g^+(t_2)$, 
the continuous process $X_{[t_1,t_2]}^{a, b,(g^-, g^+)}$ on $[t_1,t_2]$ 
is defined as follows (see also \eqref{Def_conditional_StochProc} and Lemma~\ref{Lem_for_H_gm_to_gp} below):
\begin{itemize}
\item in the case $a=g^-(t_1), b < g^+(t_2)$, the weak limit of 
$B_{[t_1,t_2]}^{a \to b} |_{K_{[t_1,t_2]}(g^--\varepsilon, g^+)}$ 
as $\varepsilon \downarrow 0$;
\item in the case $a>g^-(t_1), b = g^+(t_2)$, the weak limit of 
$B_{[t_1,t_2]}^{a \to b} |_{K_{[t_1,t_2]}(g^-, g^++\varepsilon)}$ 
as $\varepsilon \downarrow 0$;
\item in the case $g^-(t_1) < a < g^+(t_1), g^-(t_2) < b < g^+(t_2)$, 
the conditioned process $B_{[t_1,t_2]}^{a \to b} |_{K_{[t_1,t_2]}(g^-, g^+)}$.
\end{itemize}
In addition, according to the value $g^-(t_1) \leq a < g^+(t_1)$, the continuous process 
$X_{[t_1,t_2]}^{a, (g^-, g^+)}$ on $[t_1,t_2]$ is defined as follows 
(see also \eqref{Def_conditional_StochProc} and Lemma~\ref{Lem_WeakConv_Meander_btwn_2crvs} below):
\begin{itemize}
\item in the case $g^-(t_1) = a$, the weak limit of 
$\left( a+W_{[t_1,t_2]} \right)|_{K_{[t_1,t_2]}(g^--\varepsilon, g^+)}$ 
as $\varepsilon \downarrow 0$; 
\item in the case $g^-(t_1) < a$, the conditioned process 
$\left( a+W_{[t_1,t_2]} \right)|_{K_{[t_1,t_2]}(g^-, g^+)}$. 
\end{itemize}
For an $\mathbb{R}$-valued continuous process $X$ on $[t_1,t_2]$ and 
$\mathbb{R}$-valued $C^2$-function $g$ defined on $[t_1,t_2]$, we define
\begin{align*}
Z_{[t_1,t_2]}^g (X) := 
\exp \left\{ g'(t_2)X(t_2) - g'(t_1)X(t_1) 
-\int_{t_1}^{t_2} X(u)g^{\prime \prime}(u) du 
- \frac{1}{2} \int_{t_1}^{t_2} g'(u)^2 du \right\}.
\end{align*}
Therefore, if $X$ is $W_{[t_1,t_2]}$, then it follows from It\^{o}'s formula that
\begin{align*}
Z_{[t_1,t_2]}^g (W_{[t_1,t_2]}) = 
\exp \left\{ \int_{t_1}^{t_2} g'(u) dW_{[t_1,t_2]}(u) 
- \frac{1}{2} \int_{t_1}^{t_2} g'(u)^2 du \right\}.
\end{align*}
For ease of later computations, we define 
$\widetilde{Z}_{[t_1,t_2]}^g (X) := Z_{[t_1,t_2]}^g (X+g)$. 

For $f \in C([t_1,t_2], \mathbb{R})$, 
we define $\overset{\leftarrow}{f}\in C([t_1,t_2], \mathbb{R})$ as
\begin{align*}
\overset{\leftarrow}{f}(t):=f(t_1+t_2-t), \quad t_1 \leq t \leq t_2.
\end{align*}

The following two lemmas are prepared to state our main results 
(Theorems \ref{Thm_Def_and_Decomp_curved_Moving}, 
\ref{Thm_Def_and_Decomp_curved_Meander_two_curve} and 
\ref{Thm_Def_and_Decomp_curved_BESbridge_two_curve}). 
The processes $X_{[t_1,t_2]}^{a, b,(g^-, g^+)}$ and $X_{[t_1,t_2]}^{a, (g^-, g^+)}$ 
in Lemmas \ref{Lem_for_H_gm_to_gp} and \ref{Lem_WeakConv_Meander_btwn_2crvs} 
are not our main target of this study, but just auxiliary ones.
In Lemmas~\ref{Lem_for_H_gm_to_gp} and 
\ref{Lem_WeakConv_Meander_btwn_2crvs} below, 
we assume that $\{ \eta(\varepsilon) \}_{\varepsilon>0}$ satisfies
\begin{align*}
\eta (\varepsilon)\geq 0\quad (\varepsilon>0)
\quad \mbox{and} \quad \eta(\varepsilon) \downarrow 0\quad (\varepsilon \downarrow 0).
\end{align*}
Two curves $g^--\eta^-(\varepsilon)$, $g^++\eta^+(\varepsilon)$ 
in Theorem~\ref{Thm_Def_and_Decomp_curved_Moving} 
move simultaneously as $\varepsilon \downarrow 0$. 
Therefore, to prove Theorem~\ref{Thm_Def_and_Decomp_curved_Moving}, 
$\{ \eta(\varepsilon) \}_{\varepsilon>0}$ 
in Lemmas~\ref{Lem_for_H_gm_to_gp} and \ref{Lem_WeakConv_Meander_btwn_2crvs} 
play an important role.

\begin{lem}\label{Lem_for_H_gm_to_gp}
Let $0 \leq t_1 < t_2 \leq 1$. 
$X_{[t_1,t_2]}^{a, b,(g^-, g^+)}$ exists and its distribution is given as follows. 
For every $\mathbb{R}$-valued bounded continuous function $F$ on $C([t_1,t_2], \mathbb{R})$,
\begin{itemize}
\item[{\rm ({1})}] if $a=g^-(t_1)$, $g^-(t_2)\leq b < g^+(t_2)$, then
\begin{align}
&E\left[ F(X_{[t_1,t_2]}^{a, b,(g^-, g^+)}) \right] 
= \lim_{\varepsilon \downarrow 0}
E\left[ F\left( 
B_{[t_1,t_2]}^{a \to b} |_{K_{[t_1,t_2]}(g^--\varepsilon, g^++\eta(\varepsilon))}
\right)\right] \nonumber \\
&\quad = \frac{ E\left[ F\Big(r_{[t_1,t_2]}^{0 \to b-g^-(t_2)} \big|_{K_{[t_1,t_2]}^-(g^+-g^-)} + g^-\Big) 
\widetilde{Z}_{[t_1,t_2]}^{g^--a} \Big(r_{[t_1,t_2]}^{0 \to b-g^-(t_2)} \big|_{K_{[t_1,t_2]}^-(g^+-g^-)}\Big)^{-1} \right]}
{E\left[\widetilde{Z}_{[t_1,t_2]}^{g^--a} \Big(r_{[t_1,t_2]}^{0 \to b-g^-(t_2)} \big|_{K_{[t_1,t_2]}^-(g^+-g^-)} \Big)^{-1} \right]}, 
\label{G_bridge_on_t1t2_touch_at_t1}
\end{align}
\item[{\rm ({2})}] if $g^-(t_1)<a\leq g^+(t_1)$, $b = g^+(t_2)$, then
\begin{align}
& E\left[F(X_{[t_1,t_2]}^{a, b,(g^-, g^+)}) \right] 
 = \lim_{\varepsilon \downarrow 0}
E\left[ F\left( 
B_{[t_1,t_2]}^{a \to b} |_{K_{[t_1,t_2]}(g^--\eta(\varepsilon), g^++\varepsilon)}
\right)\right] \nonumber \\
&\quad  = 
\frac{ E\left[F\Big(g^+-\overset{\leftarrow}{r}_{[t_1,t_2]}^{0 \to g^+(t_1)-a} \big|_{K_{[t_1,t_2]}^-(g^+-g^-)}\Big) 
\widetilde{Z}_{[t_1,t_2]}^{b-\overset{\leftarrow}{g}^+} 
\Big(r_{[t_1,t_2]}^{0 \to g^+(t_1)-a} \big|_{K_{[t_1,t_2]}^-(\overset{\leftarrow}{g}^+-\overset{\leftarrow}{g}^-)}\Big)^{-1}
\right]}
{E\left[\widetilde{Z}_{[t_1,t_2]}^{b-\overset{\leftarrow}{g}^+} 
\Big(r_{[t_1,t_2]}^{0 \to g^+(t_1)-a} \big|_{K_{[t_1,t_2]}^-(\overset{\leftarrow}{g}^+-\overset{\leftarrow}{g}^-)}\Big)^{-1}
\right]}, 
\label{G_bridge_on_t1t2_touch_at_t2}
\end{align}
where $\overset{\leftarrow}{r}_{[t_1,t_2]}^{0 \to g^+(t_1)-a}$
denotes the continuous process 
$\big\{ r_{[t_1,t_2]}^{0 \to g^+(t_1)-a}(t_1+t_2-t)\big\}_{t\in [t_1, t_2]}$, 
and $\overset{\leftarrow}{g}^{\pm}$ denote continuous functions 
$g^{\pm}(t_1+t_2-\cdot) \in C([t_1,t_2], \mathbb{R})$, respectively.
\end{itemize}
\end{lem}

Here, note that Lemma~\ref{Lem_for_H_gm_to_gp} is necessary 
to state Theorem~\ref{Thm_Def_and_Decomp_curved_Moving}.

\begin{lem}\label{Lem_WeakConv_Meander_btwn_2crvs}
Let $0 \leq t_1 < t_2 \leq 1$ and $a=g^-(t_1)$. 
$X_{[t_1,t_2]}^{a, (g^-, g^+)}$ 
exists and its distribution is given as follows. 
For every $\mathbb{R}$-valued bounded continuous function $F$ on $C([t_1,t_2], \mathbb{R})$,
\begin{align}
& E\left[ F(X_{[t_1,t_2]}^{a, (g^-, g^+)}) \right] 
= \lim_{\varepsilon \downarrow 0}
E\left[ F\left( 
\left( a+W_{[t_1,t_2]} \right)|_{K_{[t_1,t_2]}(g^--\varepsilon, g^++\eta(\varepsilon))}
\right)\right] \nonumber \\
& \quad = \frac{E\left[ F\big(W_{[t_1,t_2]}^+ |_{K_{[t_1,t_2]}^-(g^+-g^-)} +g^-\big) 
\widetilde{Z}_{[t_1,t_2]}^{g^--a} \big(W_{[t_1,t_2]}^+|_{K_{[t_1,t_2]}^-(g^+-g^-)}\big)^{-1}\right]}
{E\left[ \widetilde{Z}_{[t_1,t_2]}^{g^--a} \big(W_{[t_1,t_2]}^+|_{K_{[t_1,t_2]}^-(g^+-g^-)}\big)^{-1}\right] }. 
\label{G_meander_on_t1t2}
\end{align}
\end{lem}

\begin{rem}\label{Rem_F_canbetakento_Indicator_MeanderBESbridge}
Let $A$ be a closed subset of $C([t_1, t_2], \mathbb{R})$, and let
\begin{align*}
&d_{\infty}(w, A):=\inf\{d_{\infty}(w,v) \mid v\in A\} \quad (w\in C([t_1, t_2], \mathbb{R})), \\
&\varphi(x):=1-\int_0^1 1_{(-\infty, x]}(u)du \quad (x\in \mathbb{R}), \\
&F_n(w):=\varphi(nd_{\infty}(w, A))\quad (w\in C([t_1, t_2], \mathbb{R})).
\end{align*}
Then, $F_n$ is a bounded continuous function on $C([t_1, t_2], \mathbb{R})$  
and satisfies 
\[F_n(w)\downarrow 1_A(w), \qquad n \to \infty \]
for $w\in C([t_1, t_2], \mathbb{R})$. 
Thus, the dominated convergence theorem 
implies that Lemmas~\ref{Lem_for_H_gm_to_gp} and \ref{Lem_WeakConv_Meander_btwn_2crvs} hold true for $F=1_A$. 
Let $B\in \mathcal{B}(C([t_1,t_2], \mathbb{R}))$. 
Then, it follows from Dynkin's $\pi$-$\lambda$ theorem that 
Lemmas~\ref{Lem_for_H_gm_to_gp} and \ref{Lem_WeakConv_Meander_btwn_2crvs} hold true for $F=1_B$. 
\end{rem}

Further, we present the notation used in 
Subsections~\ref{subsection_Moving_between_2curves}, 
\ref{subsection_Meander_between_2curves} and
\ref{subsection_BES3bridge_between_2curves}.

Let $t_0 \in (t_1, t_2)$. 
For $w_1 \in C([t_1, t_0], \mathbb{R})$ and $w_2 \in C([t_0, t_2], \mathbb{R})$ 
that satisfy $w_1(t_0) = w_2(t_0)$, 
we define $w_1 \oplus_{t_0} w_2 \in C([t_1,t_2], \mathbb{R})$ as
\begin{align*}
(w_1 \oplus_{t_0} w_2)(t) :=
\begin{cases}
 w_1(t), \qquad t_1 \leq t \leq t_0, \\
 w_2(t), \qquad t_0 \leq t \leq t_2.
\end{cases}
\end{align*}

For $0<t<1$, $0 \leq t_1 < t_2 \leq 1$ and 
$y \in (g^-(t), g^+(t))$, $y_i \in (g^-(t_i), g^+(t_i))$ $(i=1,2)$, we define
\begin{align*}
q^{(g^-,g^+), (\uparrow)}_{[0,t]}(y)
&=E\Big[\widetilde{Z}_{[0, t]}^{g^--g^-(0)} \big(r_{[0, t]}^{0 \to y-g^-(t)} |_{K_{[0, t]}^-(g^+-g^-)}\big)^{-1}\Big] \\
&\quad \times
P\Big(r_{[0, t]}^{0 \to y-g^-(t)} \in K_{[0, t]}^-(g^+-g^-)\Big)
\frac{P(W_{[0, t]}^+(t) \in dy-g^-(t))}{dy}, \\
q^{(g^-,g^+), (\downarrow)}_{[t,1]}(y)
&=E\Big[\widetilde{Z}_{[t, 1]}^{g^+(1)-\overset{\leftarrow}{g}^+} 
\Big(r_{[t, 1]}^{0 \to g^+(t)-y} |_{K_{[t, 1]}^-(\overset{\leftarrow}{g}^+-\overset{\leftarrow}{g}^-)}\Big)^{-1}\Big] \\
&\quad \times
P\Big(r_{[t, 1]}^{0 \to g^+(t)-y} \in K_{[t, 1]}^-(\overset{\leftarrow}{g}^+-\overset{\leftarrow}{g}^-)\Big)
\frac{P(W_{[t, 1]}^+(1) \in g^+(t)-dy)}{dy}, \\
p^{(g^-,g^+)}_{[t_1,t_2]}(y_1)&=P(y_1 + W_{[t_1,t_2]} \in K_{[t_1,t_2]}(g^-, g^+)), \\
p^{(g^-,g^+)}_{[t_1,t_2]}(y_1, y_2)&=P(y_1 + W_{[t_1,t_2]} \in K_{[t_1,t_2]}(g^-, g^+), y_1 + W_{[t_1,t_2]}(t_2) \in dy_2)/dy_2.
\end{align*}
Here, we have
\begin{align*}
\frac{P(W_{[0, t]}^+(t) \in dy-g^-(t))}{dy}&=\sqrt{2\pi}\cdot \frac{y-g^-(t)}{\sqrt{t}}n_t(y-g^-(t)), \\
\frac{P(W_{[t, 1]}^+(1) \in g^+(t)-dy)}{dy}&=\sqrt{2\pi}\cdot \frac{g^+(t)-y}{\sqrt{1-t}}n_{1-t}(g^+(t)-y).
\end{align*}

\subsection{Construction and sample path properties 
of Brownian house-moving}
\label{subsection_Moving_between_2curves}

In this subsection, we define $b:=g^+(1)$ and assume that $g^-(0)=0$. 

Assume that $\{ \eta^-(\varepsilon) \}_{\varepsilon>0}$ and $\{ \eta^+(\varepsilon) \}_{\varepsilon>0}$ satisfy
\begin{align*}
\eta^{\pm}(\varepsilon)>0\quad (\varepsilon>0)
\quad \mbox{and} \quad \eta^{\pm}(\varepsilon) \downarrow 0\quad (\varepsilon \downarrow 0).
\end{align*}
For $0<t<1$, $0<t_1 < t_2<1$ and 
$y \in (g^-(t), g^+(t))$, $y_i \in (g^-(t_i), g^+(t_i))$ $(i=1,2)$, we define
\begin{align*}
&h(t,y)=(C_{g^-,g^+})^{-1}\frac{1}{\sqrt{t}}q^{(g^-,g^+), (\uparrow)}_{[0,t]}(y)\frac{1}{\sqrt{1-t}}q^{(g^-,g^+), (\downarrow)}_{[t,1]}(y),\\
&h(t_1, y_1, t_2, y_2) 
=\dfrac{p^{(g^-,g^+)}_{[t_1,t_2]}(y_1, y_2) \frac{1}{\sqrt{1-t_2}}q^{(g^-,g^+), (\downarrow)}_{[t_2,1]}(y_2)}
{\frac{1}{\sqrt{1-t_1}}q^{(g^-,g^+), (\downarrow)}_{[t_1,1]}(y_1) } ,
\end{align*}
where
\begin{align}\label{Def_Cgmgp}
C_{g^-,g^+}:=\frac{\pi n_1(b) }{2}
\lim_{\varepsilon \downarrow 0} 
\frac{P(B_{[0, 1]}^{0 \to b} \in K_{[0, 1]}(g^--\eta^-(\varepsilon), g^++\eta^+(\varepsilon)))}
{\eta^-(\varepsilon)\eta^+(\varepsilon)}.
\end{align}
Here, note that $h(t. \cdot)$ becomes a density function on the interval $(g^-(t), g^+(t))$. 

Our aim in this subsection is to prove the existence of the weak limit of 
$B_{[0,1]}^{0 \to b} |_{K_{[0, 1]}(g^--\eta^-(\varepsilon), g^++\eta^+(\varepsilon))}$ 
as $\varepsilon \downarrow 0$. 
$H^{g^-\to g^+}$ denotes this weak limit. 
Sample paths of $H^{0\to b}$ in \cite{bib_Ishitani_moving} 
appear as if they are randomly moving from old house $0$ to new house $b$. 
Therefore, in this study, we call $H^{g^-\to g^+}$ ``Brownian house-moving.''

\begin{thm}\label{Thm_Def_and_Decomp_curved_Moving}
There exists an $\mathbb{R}$-valued continuous Markov process 
$H^{g^-\to g^+}=\{H^{g^-\to g^+}(t)\}_{t\in [0,1]}$ that satisfies
\begin{align}
& E\left[ F(H^{g^-\to g^+}) \right]
=\lim_{\varepsilon \downarrow 0} 
E[F(B_{[0,1]}^{0 \to b} |_{K_{[0, 1]}(g^--\eta^-(\varepsilon), g^++\eta^+(\varepsilon))})]
\label{curvedMoving_def_eq1}\\
&\quad = \int_{g^-(t)}^{g^+(t)} 
E\left[ F(X_{[0,t]}^{0,y,(g^-, g^+)} \oplus_{t} X_{[t,1]}^{y,b,(g^-, g^+)}) \right] h(t,y) dy \label{6.3} \\
&\quad =\int_{g^-(t_1)}^{g^+(t_1)} \int_{g^-(t_2)}^{g^+(t_2)} 
E[F(X_{[0,t_1]}^{0,y_1,(g^-, g^+)} \oplus_{t_1} 
X_{[t_1,t_2]}^{y_1, y_2,(g^-, g^+)}\oplus_{t_2} 
X_{[t_2, 1]}^{y_2, b,(g^-, g^+)})] \label{6.4} \\
&\quad \qquad \qquad \qquad \times 
h(t_1, y_1)h(t_1, y_1, t_2, y_2) dy_1 dy_2 \nonumber
\end{align}
for every $\mathbb{R}$-valued bounded continuous function $F$ on 
$C([0, 1], \mathbb{R})$, $0<t<1$ and $0<t_1<t_2<1$,  
where the respective processes 
that appear in {\rm (\ref{6.3})} and {\rm (\ref{6.4})} are independent of each other. 
Moreover, for $0 < t < 1$, $0<t_1<t_2<1$ and $y\in (g^-(t), g^+(t))$, $y_i \in (g^-(t_i), g^+(t_i))$ $(i=1,2)$, 
the transition densities for $H^{g^-\to g^+}$ are given by
\begin{align*}
&P(H^{g^-\to g^+}(t) \in dy) = h(t, y)dy, \\
&P(H^{g^-\to g^+}(t_2) \in dy_2 \ \vert \ H^{g^-\to g^+}(t_1)=y_1) 
=h(t_1, y_1, t_2, y_2) dy_2.
\end{align*}
\end{thm}

For $0<t<1$, 
$y, y_1, y_2 \in \mathbb{R}$ and $\eta>0$, we define
\begin{align}
&J^{(\eta)}(t,y) := 
\sum_{k=-\infty}^{\infty} \frac{2(y+2k\eta)}{t} n_t(y+2k\eta) , 
\nonumber \\
&\overline{J}^{(\eta)}(t,y):=\frac{\partial}{\partial \eta}J^{(\eta)}(t,y)=
4\sum_{k=-\infty}^{\infty} k 
\left( \frac{1}{t}-\frac{(y+2k\eta)^2}{t^2}\right)n_t(y+2k\eta), 
\label{Def_of_functions_J}\\
&J^{(\eta)}(t, y_1, y_2) := 
\sum_{k=-\infty}^{\infty} (n_{t}(y_2-y_1+2k\eta) - n_{t}(y_2+y_1+2k\eta)).
\nonumber
\end{align}
Applying Theorem~\ref{Thm_Def_and_Decomp_curved_Moving} (\ref{6.3}) 
for $g^- \equiv 0$ and $g^+ \equiv b$, we obtain the next corollary.
\begin{cor}\label{Cor_Decomp_flat_Moving}
Let $b>0$. It holds for every $\mathbb{R}$-valued bounded continuous function $F$ on $C([0,1], \mathbb{R})$ that
\begin{align*}
& E\left[ F(H^{0 \to b}) \right]
=\lim_{\varepsilon \downarrow 0} 
E[F(B_{[0,1]}^{0 \to b} |_{K_{[0, 1]}(-\eta^-(\varepsilon), b+\eta^+(\varepsilon))})]\\
&\quad = \int_0^b 
E\left[ F\left( r_{[0,t]}^{0 \to y}|_{K_{[0,t]}^-(b)} 
\oplus_t \Big(b-\overset{\leftarrow}{r}_{[t,1]}^{0 \to b-y}|_{K_{[t,1]}^-(b)}\Big) \right) \right] 
P\left( H^{0 \to b}(t) \in dy \right), \quad 0<t<1, 
\end{align*}
where $r_{[0,t]}^{0 \to y}|_{K_{[0,t]}^-(b)}$ and 
$\overset{\leftarrow}{r}_{[t,1]}^{0 \to b-y}|_{K_{[t,1]}^-(b)}$
are chosen to be independent. 
Moreover, for $0 < s < t < 1$ and $x, y \in (0,b)$, 
the transition densities for $H^{0 \to b}$ are given by
\begin{align*}
&P\left( H^{0 \to b}(t) \in dy \right) 
=\dfrac{J^{(b)}(t,y)\ J^{(b)}(1-t, b-y)}{\overline{J}^{(b)}(1,b)}dy, \\
&P\left( H^{0 \to b}(t) \in dy~|~H^{0 \to b}(s)=x \right) 
=\frac{J^{(b)}(t-s,x,y)\ J^{(b)}(1-t,b-y)}{J^{(b)}(1-s,b-x)}dy.
\end{align*}
\end{cor}

\begin{rem}\label{Rem_F_canbetakento_Indicator_Moving}
Let $B\in \mathcal{B}(C([0,1], \mathbb{R}))$ be a measurable subset of $C([0,1], \mathbb{R})$. 
Then, it follows 
from the same argument in Remark~\ref{Rem_F_canbetakento_Indicator_MeanderBESbridge} 
that Theorem~\ref{Thm_Def_and_Decomp_curved_Moving} 
and Corollary~\ref{Cor_Decomp_flat_Moving} hold true for $F=1_B$.
\end{rem}

\begin{cor}\label{Cor_30_curve_max_0t}
Let $g$ be an $\mathbb{R}$-valued $C^1$-function defined on $[0, 1]$ that satisfies 
\begin{align*}
g^-(t)<g(t)\leq g^+(t), \quad 0\leq t \leq 1.
\end{align*} 
Then, for $t \in (0,1)$ and $g^-(t)\leq z\leq g(t)$, we have
\begin{align}
&P\left( \min_{u\in [0,t]} \left\{ g(u)-H^{g^-\to g^+}(u) \right\} = 0 \right) = 0, 
\label{eq_curve_max_0t_bdry} \\
&P\left( \min_{u\in [0,t]}\left\{ g(u)-H^{g^-\to g^+}(u) \right\} \geq 0, H^{g^-\to g^+}(t) \leq z \right)
\nonumber \\
&\quad =\int_{g^-(t)}^z(C_{g^-,g^+})^{-1}\frac{1}{\sqrt{t}}q^{(g^-,g), (\uparrow)}_{[0,t]}(y)\frac{1}{\sqrt{1-t}}q^{(g^-,g^+), (\downarrow)}_{[t,1]}(y)dy.
\label{eq_curve_max_0t_jointdist}
\end{align}
\end{cor}

\begin{cor}\label{Cor_30_curve_min_t1}
Let $g$ be an $\mathbb{R}$-valued $C^1$-function defined on $[0, 1]$ that satisfies 
\begin{align*}
g^-(t)\leq g(t)< g^+(t), \quad 0\leq t \leq 1.
\end{align*} 
Then, for $t \in (0,1)$ and $g(t)\leq z \leq g^+(t)$, we have
\begin{align}
&P\left( \min_{u\in [t,1]} \left\{ H^{g^-\to g^+}(u)-g(u) \right\} = 0 \right) = 0, 
\label{eq_curve_min_t1_bdry}\\
&P\left( \min_{u\in [t,1]}\left\{ H^{g^-\to g^+}(u)-g(u) \right\} \geq 0, H^{g^-\to g^+}(t) \leq z \right)
\nonumber \\
&\quad =\int_{g(t)}^z(C_{g^-,g^+})^{-1}
\frac{1}{\sqrt{t}}q^{(g^-,g^+), (\uparrow)}_{[0,t]}(y)\frac{1}{\sqrt{1-t}}q^{(g,g^+), (\downarrow)}_{[t,1]}(y)dy.
\label{eq_curve_min_t1_jointdist}
\end{align}
\end{cor}

\begin{rem}
Let $t\in (0,1)$. Applying 
Corollary~\ref{Cor_30_curve_max_0t} 
(resp., Corollary~\ref{Cor_30_curve_min_t1}) 
for $g=g^+$ (resp., $g=g^-$), we have
\begin{align*}
&P\left( m_{[0,t]}(g^+-H^{g^-\to g^+}) = 0 \right) = 0, 
\quad P\left( m_{[t,1]}(H^{g^-\to g^+}-g^-) = 0 \right) = 0, \\
&P\left( m_{[0,t]}(g^+-H^{g^-\to g^+}) \geq 0 \right) 
= P\left( m_{[0,t]}(g^+-H^{g^-\to g^+}) \geq 0, H^{g^-\to g^+}(t) \leq g^+(t) \right) 
= \int_{g^-(t)}^{g^+(t)} h(t,y) dy = 1, \\
&P\left( m_{[t,1]}(H^{g^-\to g^+}-g^-) \geq 0 \right) 
= P\left( m_{[t,1]}(H^{g^-\to g^+}-g^-) \geq 0, H^{g^-\to g^+}(t) \leq g^+(t) \right) 
= \int_{g^-(t)}^{g^+(t)} h(t,y) dy = 1, \\
&P\left( m_{[0,t]}(g^+-H^{g^-\to g^+}) > 0 \right) 
= P\left( m_{[0,t]}(g^+-H^{g^-\to g^+}) \geq 0 \right) - P\left( m_{[0,t]}(g^+-H^{g^-\to g^+}) = 0 \right) =1, \\
&P\left( m_{[t,1]}(H^{g^-\to g^+}-g^-) > 0 \right)  
= P\left( m_{[t,1]}(H^{g^-\to g^+}-g^-) \geq 0 \right) - P\left( m_{[t,1]}(H^{g^-\to g^+}-g^-) = 0 \right) =1.
\end{align*}
Therefore, Brownian house-moving $H^{g^-\to g^+}$ satisfies
\begin{align*}
P\left( \bigcap_{n\geq 2} \left\{ \min_{0\leq u \leq 1-1/n}(g^+(u)-H^{g^-\to g^+}(u)) > 0, 
\ \min_{1/n\leq u\leq 1}(H^{g^-\to g^+}(u)-g^-(u)) > 0\right\} \right)=1.
\end{align*}
\end{rem}

Let $t\in (0,1)$. 
Applying Theorem~\ref{Thm_Def_and_Decomp_curved_Moving} \eqref{curvedMoving_def_eq1} 
and a change of measure formula between Brownian meander and BES$(3)$-process (\cite{bib_Imhof}), 
we obtain the Radon-Nikodym derivative of $\pi_{[0, t]}\circ H^{g^-\to g^+}$ with respect to $R_{[0,t]} +g^-$.

\begin{thm}\label{Thm_abs_conti}
Let $t\in (0, 1)$, and let $R_{[0,t]} = \{ R_{[0,t]}(u)\}_{u\in [0,t]}$ 
be the BES($3$)-process starting at $0$ on $[0,t]$. 
Then, it holds that
\begin{align*}
&\frac{d\left( P\circ (\pi_{[0, t]}\circ H^{g^-\to g^+})^{-1} \right)}{d\left( P\circ (R_{[0,t]} +g^-)^{-1}\right)}(w) \\
&\quad = \sqrt{\frac{\pi}{2}}\cdot 
\dfrac{ q^{(g^-,g^+), (\downarrow)}_{[t,1]}\left(w(t)\right) }
{C_{g^-,g^+} \sqrt{1-t} \cdot (w(t)-g^-(t))\cdot Z_{[0,t]}^{g^-}(w)}
\cdot 1_{K_{[0,t]}^-(g^+)}\left( w\right), 
\quad w\in C([0, t],\mathbb{R}).
\end{align*}
\end{thm}

\begin{rem}\label{Remark_Approx_ExpectedVal_HhouseMoving}
In \cite{bib_Ishitani_moving}, 
using Corollary~\ref{Cor_Decomp_flat_Moving} and a Monte Carlo sampling technique for BES$(3)$-bridges, 
we numerically generated Brownian house-moving $H^{0 \to b}$ at discrete times. 
On the other hand, this sampling method does not work effectively for general Brownian house-moving $H^{g^- \to g^+}$. 
However, combining Theorem~\ref{Thm_abs_conti} 
and a Monte Carlo sampling technique for the BES$(3)$-process, 
we can approximate the expected values of the functional of $H^{g^- \to g^+}$.
\end{rem}

\subsection{Construction of Brownian meander between two curves}
\label{subsection_Meander_between_2curves}

In this subsection, we assume that $g^-(0)=0$. 

For $0<t<1$, $0<t_1 < t_2<1$ and 
$y \in (g^-(t), g^+(t))$, $y_i \in (g^-(t_i), g^+(t_i))$ $(i=1,2)$, we define
\begin{align*}
&k(t,y)=(\widetilde{C}_{g^-,g^+})^{-1}\frac{1}{\sqrt{t}}q^{(g^-,g^+), (\uparrow)}_{[0,t]}(y)p^{(g^-,g^+)}_{[t,1]}(y), \\
&k(t_1, y_1, t_2, y_2) 
=\dfrac{p^{(g^-,g^+)}_{[t_1,t_2]}(y_1, y_2) p^{(g^-,g^+)}_{[t_2,1]}(y_2)}{p^{(g^-,g^+)}_{[t_1,1]}(y_1) } ,
\end{align*}
where
\begin{align*}
\widetilde{C}_{g^-,g^+}:=\sqrt{\frac{\pi}{2}}
\lim_{\varepsilon \downarrow 0} 
\frac{P(W_{[0, 1]} \in K_{[0, 1]}(g^--\varepsilon, g^+))}
{\varepsilon}.
\end{align*}

$W^{+, (g^-,g^+)}$ denotes $X_{[0,1]}^{0,(g^-, g^+)}$, 
which is the weak limit of 
$W_{[0,1]} |_{K_{[0, 1]}(g^--\varepsilon, g^+)}$ 
as $\varepsilon \downarrow 0$. 
The well-known Brownian meander $W^+$ is a non-negative process and can be formally denoted by $W^{+, (0,\infty)}$ in our notations. 
Therefore, in this study, we call $W^{+, (g^-,g^+)}$ ``Brownian meander between two curves.'' 
$W^{+, (g^-,g^+)}$ played an important role in \cite{bib_Ishitani}. 
Our aim in this subsection is to prove that 
$W^{+, (g^-,g^+)}$ is an $\mathbb{R}$-valued continuous Markov process on $[0,1]$.

\begin{thm}\label{Thm_Def_and_Decomp_curved_Meander_two_curve}
There exists an $\mathbb{R}$-valued continuous Markov process 
$W^{+, (g^-,g^+)}=\{W^{+, (g^-,g^+)}(t)\}_{t\in [0,1]}$ that satisfies
\begin{align}
&E\left[ F(W^{+, (g^-,g^+)}) \right]
=\lim_{\varepsilon \downarrow 0} 
E[F(W_{[0,1]} |_{K_{[0, 1]}(g^--\varepsilon, g^+)})]
\label{curvedMeander_twocurve_def_eq1}\\
&\quad = \int_{g^-(t)}^{g^+(t)} 
E\left[ F(X_{[0,t]}^{0,y,(g^-, g^+)} \oplus_{t} X_{[t,1]}^{y,(g^-, g^+)}) \right] k(t,y) dy 
\label{twocurve_meander_decomp_1} \\
&\quad =\int_{g^-(t_1)}^{g^+(t_1)} \int_{g^-(t_2)}^{g^+(t_2)} 
E[F(X_{[0,t_1]}^{0,y_1,(g^-, g^+)} \oplus_{t_1} 
X_{[t_1,t_2]}^{y_1, y_2,(g^-, g^+)}\oplus_{t_2} 
X_{[t_2, 1]}^{y_2,(g^-, g^+)})] 
\label{twocurve_meander_decomp_2} \\
&\quad \qquad \qquad \qquad \times
k(t_1, y_1)k(t_1, y_1, t_2, y_2) dy_1 dy_2
\nonumber
\end{align}
for every $\mathbb{R}$-valued  bounded continuous function $F$ on 
$C([0, 1], \mathbb{R})$, $0<t<1$ and $0<t_1<t_2<1$,  
where the respective processes 
that appear in \eqref{twocurve_meander_decomp_1} and \eqref{twocurve_meander_decomp_2} are independent of each other. 
Moreover, for $0 < t < 1$, $0<t_1<t_2<1$ and $y\in (g^-(t), g^+(t))$, $y_i \in (g^-(t_i), g^+(t_i))$ $(i=1,2)$, 
the transition densities for $W^{+, (g^-,g^+)}$ are given by
\begin{align*}
&P(W^{+, (g^-,g^+)}(t) \in dy) = k(t, y)dy, \quad 
P(W^{+, (g^-,g^+)}(t_2) \in dy_2 \ \vert \ W^{+, (g^-,g^+)}(t_1)=y_1) =k(t_1, y_1, t_2, y_2) dy_2.
\end{align*}
\end{thm}

\begin{rem}\label{Rem_F_canbetakento_Indicator_Meander}
Let $B\in \mathcal{B}(C([0,1], \mathbb{R}))$ be a measurable subset of $C([0,1], \mathbb{R})$. 
Then, it follows from the same argument in Remark~\ref{Rem_F_canbetakento_Indicator_MeanderBESbridge} 
that Theorem~\ref{Thm_Def_and_Decomp_curved_Meander_two_curve} holds true for $F=1_B$.
\end{rem}

\subsection{Construction of BES($3$)-bridge between two curves}
\label{subsection_BES3bridge_between_2curves}

In this subsection, 
we assume that $g^-(0)=0$ and $g^-(1)<c<b:=g^+(1)$. 

For $0<t<1$, $0<t_1 < t_2<1$ and 
$y \in (g^-(t), g^+(t))$, $y_i \in (g^-(t_i), g^+(t_i))$ $(i=1,2)$, we define
\begin{align*}
& l(t,y)=(\widehat{C}_{g^-,g^+})^{-1}\frac{1}{\sqrt{t}}q^{(g^-,g^+), (\uparrow)}_{[0,t]}(y)p^{(g^-,g^+)}_{[t,1]}(y,c), \\
& l(t_1, y_1, t_2, y_2) 
=\dfrac{p^{(g^-,g^+)}_{[t_1,t_2]}(y_1, y_2) p^{(g^-,g^+)}_{[t_2,1]}(y_2, c)}{p^{(g^-,g^+)}_{[t_1,1]}(y_1, c) } ,
\end{align*}
where
\begin{align*}
\widehat{C}_{g^-,g^+}:=\sqrt{\frac{\pi}{2}} n_1(c)
\lim_{\varepsilon \downarrow 0} 
\frac{P(B_{[0, 1]}^{0\to c} \in K_{[0, 1]}(g^--\varepsilon, g^+))}
{\varepsilon}.
\end{align*}

$r^{0\to c, (g^-,g^+)}$ denotes $X_{[0,1]}^{0,c,(g^-, g^+)}$, 
which is the weak limit of 
$B_{[0,1]}^{0\to c} |_{K_{[0, 1]}(g^--\varepsilon, g^+)}$ 
as $\varepsilon \downarrow 0$. 
The well-known BES($3$)-bridge $r^{0\to d}$ ($d>0$) is a non-negative process and can be formally denoted by $r^{0\to d, (0,\infty)}$ in our notations.
Therefore, in this study, we call $r^{0\to c, (g^-,g^+)}$ ``BES($3$)-bridge between two curves.'' 
$r^{0\to c, (g^-,g^+)}$ played an important role in \cite{bib_Ishitani}. 
Our aim in this subsection is to prove that 
$r^{0\to c, (g^-,g^+)}$ is an $\mathbb{R}$-valued continuous Markov process on $[0,1]$.

\begin{thm}\label{Thm_Def_and_Decomp_curved_BESbridge_two_curve}
There exists an $\mathbb{R}$-valued continuous Markov process 
$r^{0\to c, (g^-,g^+)}=\{r^{0\to c, (g^-,g^+)}(t)\}_{t\in [0,1]}$ that satisfies
\begin{align}
&E\left[ F(r^{0\to c, (g^-,g^+)}) \right]
=\lim_{\varepsilon \downarrow 0} 
E[F(B_{[0,1]}^{0\to c} |_{K_{[0, 1]}(g^--\varepsilon, g^+)})]
\label{curvedBESbridge_twocurve_def_eq1}\\
&\quad = \int_{g^-(t)}^{g^+(t)} 
E\left[ F(X_{[0,t]}^{0,y,(g^-, g^+)} \oplus_{t} X_{[t,1]}^{y,c,(g^-, g^+)}) \right] l(t,y) dy 
\label{twocurve_BESbridge_decomp_1} \\
&\quad =\int_{g^-(t_1)}^{g^+(t_1)} \int_{g^-(t_2)}^{g^+(t_2)} 
E[F(X_{[0,t_1]}^{0,y_1,(g^-, g^+)} \oplus_{t_1} 
X_{[t_1,t_2]}^{y_1, y_2,(g^-, g^+)}\oplus_{t_2} 
X_{[t_2, 1]}^{y_2,c,(g^-, g^+)})] 
\label{twocurve_BESbridge_decomp_2} \\
&\quad \qquad \qquad \qquad \times 
l(t_1, y_1)l(t_1, y_1, t_2, y_2) dy_1 dy_2
\nonumber
\end{align}
for every $\mathbb{R}$-valued bounded continuous function $F$ on 
$C([0, 1], \mathbb{R})$, $0<t<1$ and $0<t_1<t_2<1$,  
where the respective processes 
that appear in \eqref{twocurve_BESbridge_decomp_1} and \eqref{twocurve_BESbridge_decomp_2} are independent of each other. 
Moreover, for $0 < t < 1$, $0<t_1<t_2<1$ and $y\in (g^-(t), g^+(t))$, $y_i \in (g^-(t_i), g^+(t_i))$ $(i=1,2)$, 
the transition densities for $r^{0\to c, (g^-,g^+)}$ are given by
\begin{align*}
&P(r^{0\to c, (g^-,g^+)}(t) \in dy) = l(t, y)dy, \\
&P(r^{0\to c, (g^-,g^+)}(t_2) \in dy_2 \ \vert \ r^{0\to c, (g^-,g^+)}(t_1)=y_1) =l(t_1, y_1, t_2, y_2) dy_2.
\end{align*}
\end{thm}

\begin{rem}\label{Rem_F_canbetakento_Indicator_BESbridge}
Let $B\in \mathcal{B}(C([0,1], \mathbb{R}))$ be a measurable subset of $C([0,1], \mathbb{R})$. 
Then, it follows 
from the same argument in Remark~\ref{Rem_F_canbetakento_Indicator_MeanderBESbridge} 
that Theorem~\ref{Thm_Def_and_Decomp_curved_BESbridge_two_curve} holds true for $F=1_B$.
\end{rem}

We also prove that $r^{0\to b, (g^-,g^++\eta)}$ 
converges weakly to $H^{g^-\to g^+}$ as $\eta \downarrow 0$.

\begin{thm}\label{Thm_WeakConv_BESbridge_to_Moving}
For every $\mathbb{R}$-valued bounded continuous function $F$ on $C([0, 1], \mathbb{R})$, we have
\begin{align*}
&E\left[ F\big(H^{g^-\to g^+}\big) \right]
=\lim_{\eta \downarrow 0} 
E[F(r^{0\to b, (g^-,g^++\eta)})].
\end{align*}
\end{thm}

\begin{rem}
Let $R=\{R(t)\}_{t\geq 0}$ be $3$-dimensional Bessel process (BES($3$) process for short) starting from $0$, 
and let $\tau_{b}$ $(b>0)$ denotes the first hitting time of the point $b$ by $R$:
\[
\tau_{b}:=\inf\{r\geq 0~|~ R(r)=b\}.
\]
It has been shown in \cite{bib_IshitaniRinYanashima} that
Brownian house-moving $H^{0\to b}=\{H^{0 \to b}(t) \}_{t \in [0,1]}$ satisfies
\begin{align*}
P\left(H^{0\to b}(t)\in dy\right)
&=P\left(R(t)\in dy~|~\tau_{b}=1\right),\\
P\left(H^{0\to b}(t)\in dy~|~H^{0\to b}(s)=x\right)
&=P\left(R(t)\in dy~|~R(s)=x, \tau_{b}=1\right)\notag
\end{align*}
for $0<s<t<1$ and $x, y\in (0, b)$.
\end{rem}

\section{Proofs of Lemma~\ref{Lem_for_H_gm_to_gp} and Lemma~\ref{Lem_WeakConv_Meander_btwn_2crvs}}
\label{section_proof_main_Lemma}

In this section, we prove Lemma~\ref{Lem_for_H_gm_to_gp} and Lemma~\ref{Lem_WeakConv_Meander_btwn_2crvs}.

\subsection{Proof of Lemma~\ref{Lem_for_H_gm_to_gp}}

To prove \eqref{G_bridge_on_t1t2_touch_at_t1}, it suffices to show that the limit
\begin{align}\label{objlim_bridge_1_between_two_curves}
\lim_{\varepsilon \downarrow 0} 
\frac{E[F(a+W_{[t_1,t_2]})\ ;\ a+W_{[t_1,t_2]}(t_2) \in d\widetilde{b}, 
a+W_{[t_1,t_2]} \in K_{[t_1,t_2]}(g^--\varepsilon, g^++\eta(\varepsilon))]}
{P(a+W_{[t_1,t_2]}(t_2) \in d\widetilde{b}, a+W_{[t_1,t_2]} \in K_{[t_1,t_2]}(g^--\varepsilon, g^++\eta(\varepsilon)))}
\bigg\vert_{\widetilde{b}=b}
\end{align}
exists and coincides with the right-hand side of \eqref{G_bridge_on_t1t2_touch_at_t1}. 
For each $F$ and $\varepsilon>0$, we obtain 
\begin{align*}
&E[F(a+W_{[t_1,t_2]})\ ;\ a+W_{[t_1,t_2]}(t_2) \in d\widetilde{b}, 
a+W_{[t_1,t_2]} \in K_{[t_1,t_2]}(g^--\varepsilon, g^++\eta(\varepsilon))] \\
&\quad = 
E[F(B_{[t_1,t_2]}^{0 \to \widetilde{b}-g^-(t_2)}+g^-) 
\widetilde{Z}_{[t_1,t_2]}^{g^--\varepsilon} (B_{[t_1,t_2]}^{0 \to \widetilde{b}-g^-(t_2)})^{-1}
\ ;\ B_{[t_1,t_2]}^{0 \to \widetilde{b}-g^-(t_2)} \in K_{[t_1,t_2]}(-\varepsilon, g^+-g^-+\eta(\varepsilon))] \\
&\qquad \times P(g^-(t_2)+W_{[t_1,t_2]}(t_2) \in d\widetilde{b})\\
&\quad = 
E\Big[F\big(B_{[t_1,t_2]}^{0 \to \widetilde{b}-g^-(t_2)} |_{ K_{[t_1,t_2]}(-\varepsilon, g^+-g^-+\eta(\varepsilon)) } + g^-\big) 
\widetilde{Z}_{[t_1,t_2]}^{g^--a} \big(B_{[t_1,t_2]}^{0 \to \widetilde{b}-g^-(t_2)} |_{ K_{[t_1,t_2]}(-\varepsilon, g^+-g^-+\eta(\varepsilon)) }\big)^{-1}\Big]\\
&\qquad \times P\big( B_{[t_1,t_2]}^{0 \to \widetilde{b}-g^-(t_2)} \in K_{[t_1,t_2]}(-\varepsilon, g^+-g^-+\eta(\varepsilon)) \big) 
P(g^-(t_2)+W_{[t_1,t_2]}(t_2) \in d\widetilde{b})
\end{align*}
by Lemma~\ref{Ap_Lem_Girsanov_all_type} \eqref{Girsanov_pinned_formula}. 
Therefore, 
taking the limit $\varepsilon \downarrow 0$ in \eqref{objlim_bridge_1_between_two_curves}, 
we obtain (1) by Proposition~\ref{Prop_Boundary_purturb_WeakConv_for_Meander_BES3bridge}. 

To prove \eqref{G_bridge_on_t1t2_touch_at_t2}, it suffices to show that the limit
\begin{align}\label{objlim_bridge_2_between_two_curves}
&\lim_{\varepsilon \downarrow 0} 
\frac{E[F(a+W_{[t_1,t_2]})\ ;\ a+W_{[t_1,t_2]}(t_2) \in d\widetilde{b}, 
a+W_{[t_1,t_2]} \in K_{[t_1,t_2]}(g^--\eta(\varepsilon), g^++\varepsilon)]}
{P(a+W_{[t_1,t_2]}(t_2) \in d\widetilde{b}, a+W_{[t_1,t_2]} \in K_{[t_1,t_2]}(g^--\eta(\varepsilon), g^++\varepsilon))}
\bigg\vert_{\widetilde{b}=b}
\end{align}
exists and coincides with the right-hand side of \eqref{G_bridge_on_t1t2_touch_at_t2}. 
Using Lemma~\ref{Ap_Lem_Girsanov_all_type} \eqref{Girsanov_pinned_inv_formula}, we obtain
\begin{align}
&E[F(a+W_{[t_1,t_2]})\ ;\ a+W_{[t_1,t_2]}(t_2) \in d\widetilde{b}, 
a+W_{[t_1,t_2]} \in K_{[t_1,t_2]}(g^--\eta(\varepsilon), g^++\varepsilon)] \nonumber \\
&\quad = 
E\Big[F\Big(\widetilde{b} -b+g^+- \overset{\leftarrow}{B}^{0\to \widetilde{b}-b+g^+(t_1)-a}_{[t_1,t_2]}\Big)
\widetilde{Z}_{[t_1,t_2]}^{b-\overset{\leftarrow}{g}^+} \big(B^{0\to \widetilde{b}-b+g^+(t_1)-a}_{[t_1,t_2]}\big)^{-1} 
\ ;\ \nonumber \\ 
&\qquad \qquad 
B^{0\to \widetilde{b}-b+g^+(t_1)-a}_{[t_1,t_2]}
\in K_{[t_1,t_2]}\big(\widetilde{b}-b-\varepsilon, \widetilde{b}-b+\overset{\leftarrow}{g}^+-\overset{\leftarrow}{g}^-+\eta(\varepsilon)\big)\Big]
\nonumber \\
&\qquad \times P\big(a+b-g^+(t_1)+W_{[t_1,t_2]}(t_2) \in d\widetilde{b}\big)
\nonumber \\
&\quad = E\Big[F\Big(\widetilde{b} -b+g^+-\overset{\leftarrow}B_{[t_1,t_2]}^{0\to \widetilde{b} -b+g^+(t_1)-a} 
|_{K_{[t_1,t_2]}(\widetilde{b} -b-\varepsilon, \widetilde{b} -b+g^+-g^-+\eta(\varepsilon))} \Big) \nonumber \\
&\qquad \qquad \times
\widetilde{Z}_{[t_1,t_2]}^{b-\overset{\leftarrow}{g}^+} \Big(B_{[t_1,t_2]}^{0\to \widetilde{b} -b+g^+(t_1)-a} 
|_{K_{[t_1,t_2]}(\widetilde{b} -b-\varepsilon, \widetilde{b} -b+\overset{\leftarrow}{g}^+-\overset{\leftarrow}{g}^-+\eta(\varepsilon))} \Big)^{-1}\Big] 
\nonumber \\
&\qquad \times 
P\Big( B_{[t_1,t_2]}^{0\to \widetilde{b} -b+g^+(t_1)-a} \in 
K_{[t_1,t_2]}(\widetilde{b} -b-\varepsilon, \widetilde{b} -b+\overset{\leftarrow}{g}^+-\overset{\leftarrow}{g}^-+\eta(\varepsilon)) \Big)
\nonumber \\
&\qquad \times P\big(a+b-g^+(t_1)+W_{[t_1,t_2]}(t_2) \in d\widetilde{b}\big),
\qquad (g^+(t_2)-\eta(\varepsilon)<\widetilde{b}<b+\varepsilon), 
\nonumber 
\end{align}
where $\overset{\leftarrow}{B}_{[t_1,t_2]}^{0 \to \widetilde{b} -b+g^+(t_1)-a}$
denotes the continuous process 
$\big\{ B_{[t_1,t_2]}^{0 \to \widetilde{b} -b+g^+(t_1)-a}(t_1+t_2-t)\big\}_{t\in [t_1, t_2]}$.
Therefore, taking the limit $\varepsilon \downarrow 0$ in \eqref{objlim_bridge_2_between_two_curves}, 
we can obtain \eqref{G_bridge_on_t1t2_touch_at_t2} by Proposition~\ref{Prop_Boundary_purturb_WeakConv_for_Meander_BES3bridge}. 

\subsection{Proof of Lemma~\ref{Lem_WeakConv_Meander_btwn_2crvs}}

It suffices to show that the limit
\begin{align}\label{objlim_meander_between_two_curves}
\lim_{\varepsilon \downarrow 0} 
\frac{E\left[ F \left( a+W_{[t_1,t_2]} \right) 
\ ;\ a+W_{[t_1,t_2]} \in K_{[t_1,t_2]}(g^--\varepsilon, g^++\eta(\varepsilon)) \right]}
{P\left(  a+W_{[t_1,t_2]}  \in K_{[t_1,t_2]}(g^--\varepsilon, g^++\eta(\varepsilon)) \right)}
\end{align}
exists and coincides with the right-hand side of the desired result. 
By Lemma~\ref{Ap_Lem_Girsanov_all_type} \eqref{Girsanov_free_formula}, we obtain
\begin{align*}
&E\left[ F(a+W_{[t_1,t_2]})
\ ;\ a+W_{[t_1,t_2]} \in K_{[t_1,t_2]}(g^--\varepsilon, g^++\eta(\varepsilon)) \right] \\
&\quad = E\left[ F(W_{[t_1,t_2]}+g^-) 
\widetilde{Z}_{[t_1,t_2]}^{g^--\varepsilon} (W_{[t_1,t_2]})^{-1}
\ ;\ W_{[t_1,t_2]} \in K_{[t_1,t_2]}(-\varepsilon, g^+-g^-+\eta(\varepsilon)) \right]\\
&\quad = E\left[ F(W_{[t_1,t_2]} |_{ K_{[t_1,t_2]}(-\varepsilon, g^+-g^-+\eta(\varepsilon)) } + g^-) 
\widetilde{Z}_{[t_1,t_2]}^{g^--a} (W_{[t_1,t_2]} |_{ K_{[t_1,t_2]}(-\varepsilon, g^+-g^-+\eta(\varepsilon)) } )^{-1} \right] \\
&\qquad \times 
P\left( W_{[t_1,t_2]} \in K_{[t_1,t_2]}(-\varepsilon, g^+-g^-+\eta(\varepsilon)) \right). 
\end{align*}
Therefore, taking the limit $\varepsilon \downarrow 0$ in \eqref{objlim_meander_between_two_curves}, 
we obtain \eqref{G_meander_on_t1t2} by Proposition~\ref{Prop_Boundary_purturb_WeakConv_for_Meander_BES3bridge}.

\section{Preparation for proofs of the main results}
\label{section_proof_Preparation_for_Proofs}

In this section, we prove some technical lemmas 
in preparation for proofs of the main results. 

\begin{lem}\label{Lem_MarkovDecomp_B_Bridge}
Assume that $h^-$ and $h^+$ are $\mathbb{R}$-valued $C^2$-functions defined on $[0, 1]$ 
satisfying 
\begin{align*}
h^-(0)<0<h^+(0)\quad \mbox{and}  \quad
\min_{0 \leq t \leq 1}(h^+(t) - h^-(t)) > 0.
\end{align*} 
Then, for every $\mathbb{R}$-valued bounded continuous function $F$ on $C([0,1], \mathbb{R})$,  
$0<s<t<1$ and $h^-(1)<b<h^+(1)$, we have 
\begin{align}
& E[F(W_{[0,1]})\ ;\ W_{[0,1]}(1) \in db, W_{[0,1]} \in K_{[0, 1]}(h^-, h^+)] \nonumber \\
&\quad = \int_{h^-(t)}^{h^+(t)}  
E[F(X_{[0, t]}^{0, y,(h^-, h^+)} \oplus_{t} X_{[t, 1]}^{y, b,(h^-, h^+)})] 
\label{MarkovDecomp_at_t_Bridge} \\
&\qquad \qquad \quad \times 
P(W_{[0, t]} \in K_{[0, t]}(h^-, h^+), W_{[0, t]}(t) \in dy) \nonumber \\
&\qquad \qquad \quad \times 
P(y + W_{[t, 1]} \in K_{[t, 1]}(h^-, h^+), y + W_{[t, 1]}(1) \in db)\nonumber \\
&\quad = \int_{h^-(s)}^{h^+(s)} dx \int_{h^-(t)}^{h^+(t)} dy
\ E[F(X_{[0, s]}^{0, x,(h^-, h^+)} \oplus_{s} 
X_{[s,t]}^{x, y,(h^-, h^+)} \oplus_{t} 
X_{[t, 1]}^{y, b,(h^-, h^+)})] \label{MarkovDecomp_at_st_Bridge} \\
&\qquad \qquad \qquad \times 
P(W_{[0, s]} \in K_{[0, s]}(h^-, h^+), W_{[0, s]}(s) \in dx)/dx \nonumber \\
&\qquad \qquad \qquad \times 
P(x + W_{[s,t]} \in K_{[s,t]}(h^-, h^+), x + W_{[s,t]}(t) \in dy)/dy \nonumber \\
&\qquad \qquad \qquad \times 
P(y + W_{[t, 1]} \in K_{[t, 1]}(h^-, h^+), y + W_{[t, 1]}(1) \in db), \nonumber 
\end{align}
where the respective processes 
that appear in \eqref{MarkovDecomp_at_t_Bridge} 
and \eqref{MarkovDecomp_at_st_Bridge} are independent of each other. 
\end{lem}
\begin{proof}
The Markov property of $W_{[0,1]}$ yields
\begin{align}
&E[F(W_{[0,1]})\ ;\ W_{[0,1]}(1) \in db, W_{[0,1]} \in K_{[0, 1]}(h^-, h^+)]
\nonumber \\
&\quad= \int_{h^-(t)}^{h^+(t)} 
E[F(W_{[0,1]})\ ;\ W_{[0,1]}(1) \in db, 
W_{[0,1]} \in K_{[0, 1]}(h^-, h^+), W_{[0,1]}(t) \in dy]
\nonumber \\
&\quad = \int_{h^-(t)}^{h^+(t)} 
E[F(W_{[0, t]} \oplus_{t} (y + W_{[t, 1]}))\ ; \ 
W_{[0, t]} \in K_{[0, t]}(h^-, h^+), W_{[0, t]}(t) \in dy, 
\label{MarkovDecomp_at_t_Proof_Bridge} \\
&\qquad \qquad \qquad \qquad \qquad \qquad \quad
y + W_{[t, 1]} \in K_{[t, 1]}(h^-, h^+), y + W_{[t, 1]}(1) \in db] \nonumber \\
&\quad = \int_{h^-(s)}^{h^+(s)} \int_{h^-(t)}^{h^+(t)} 
E[F(W_{[0, s]} \oplus_{s} (x + W_{[s,t]}) 
\oplus_{t} (y + W_{[t, 1]}))\ ; 
\label{MarkovDecomp_at_st_Proof_Bridge}\\
&\quad \qquad \qquad \qquad \qquad \qquad 
W_{[0, s]} \in K_{[0, s]}(h^-, h^+), 
W_{[0, s]}(s) \in dx, \nonumber \\
&\quad \qquad \qquad \qquad \qquad \qquad 
x + W_{[s,t]} \in K_{[s,t]}(h^-, h^+), 
x + W_{[s,t]}(t) \in dy, \nonumber \\
&\quad \qquad \qquad \qquad \qquad \qquad 
y + W_{[t, 1]} \in K_{[t, 1]}(h^-, h^+), 
y + W_{[t, 1]}(1) \in db] , \nonumber 
\end{align}
where the respective processes 
that appear in \eqref{MarkovDecomp_at_t_Proof_Bridge} 
and \eqref{MarkovDecomp_at_st_Proof_Bridge} are independent of each other. 
Using \eqref{MarkovDecomp_at_t_Proof_Bridge} and \eqref{MarkovDecomp_at_st_Proof_Bridge}, 
we obtain \eqref{MarkovDecomp_at_t_Bridge} and \eqref{MarkovDecomp_at_st_Bridge}, respectively. 
\end{proof}

In a similar manner to the above lemma, we can obtain the following.

\begin{lem}\label{Lem_MarkovDecomp_BM}
Under the same assumption as that of 
Lemma~\ref{Lem_MarkovDecomp_B_Bridge}, we have 
\begin{align}
& E[F(W_{[0,1]})\ ;\ W_{[0,1]} \in K_{[0, 1]}(h^-, h^+)] \nonumber \\
&\quad = \int_{h^-(t)}^{h^+(t)}  
E[F(X_{[0, t]}^{0, y,(h^-, h^+)} \oplus_{t} X_{[t, 1]}^{y,(h^-, h^+)})] 
\label{MarkovDecomp_at_t_BM} \\
&\qquad \quad \times 
P(W_{[0, t]} \in K_{[0, t]}(h^-, h^+), W_{[0, t]}(t) \in dy) 
P(y + W_{[t, 1]} \in K_{[t, 1]}(h^-, h^+))\nonumber \\
&\quad = \int_{h^-(s)}^{h^+(s)} dx \int_{h^-(t)}^{h^+(t)} dy  
\ E[F(X_{[0, s]}^{0, x,(h^-, h^+)} \oplus_{s} 
X_{[s,t]}^{x, y,(h^-, h^+)} \oplus_{t} 
X_{[t, 1]}^{y,(h^-, h^+)})] \label{MarkovDecomp_at_st_BM} \\
&\qquad \qquad \qquad \qquad \quad \times 
P(W_{[0, s]} \in K_{[0, s]}(h^-, h^+), W_{[0, s]}(s) \in dx) / dx\nonumber \\
&\qquad \qquad \qquad \qquad \quad \times 
P(x + W_{[s,t]} \in K_{[s,t]}(h^-, h^+), x + W_{[s,t]}(t) \in dy) / dy \nonumber \\
&\qquad \qquad \qquad \qquad \quad \times 
P(y + W_{[t, 1]} \in K_{[t, 1]}(h^-, h^+)), \nonumber 
\end{align}
for every $\mathbb{R}$-valued bounded continuous function $F$ on $C([0,1], \mathbb{R})$ 
and $0<s<t<1$, 
where the respective processes 
that appear in \eqref{MarkovDecomp_at_t_BM} 
and \eqref{MarkovDecomp_at_st_BM} are independent of each other. 
\end{lem}

\begin{lem}\label{Lem_Gtrans_0t_dy}
Assume that $h^-$ and $h^+$ are $\mathbb{R}$-valued $C^2$-functions defined on $[0, 1]$ 
satisfying $h^-(0)=0$ and 
\begin{align}\label{LemGtrans_hpm_noncol_cond}
\min_{0 \leq t \leq 1}(h^+(t) - h^-(t)) > 0.
\end{align} 
Then, for $0<t<1$, $\varepsilon>0$ and $y \in (h^-(t)-\varepsilon , h^+(t))$, we have
\begin{align*}
&\frac{P(W_{[0, t]}(t) \in dy, 
W_{[0, t]} \in K_{[0, t]}(h^--\varepsilon, h^+))}
{P(W_{[0, t]} \in K_{[0, t]}^+(-\varepsilon ))} \\
&\quad = E\left[\widetilde{Z}_{[0, t]}^{h^-} 
(B_{[0, t]}^{0 \to y-h^-(t)}|_{K_{[0, t]}(-\varepsilon, h^+-h^- )})^{-1} \right]
P\left(B_{[0, t]}^{0 \to y-h^-(t)}|_{K_{[0, t]}^+(-\varepsilon )} 
\in K_{[0, t]}^-(h^+-h^-) \right)\\
&\qquad \times 
P\left( W_{[0, t]} |_{K_{[0, t]}^+(-\varepsilon )}(t) \in dy-h^-(t) \right).
\end{align*}
\end{lem}
\begin{proof}
Using Lemma~\ref{Ap_Lem_Girsanov_all_type} \eqref{Girsanov_pinned_formula}, we obtain
\begin{align}
&P\big( W_{[0,t]}(t) \in dy, 
W_{[0,t]} \in K_{[0,t]}(h^--\varepsilon, h^+)\big) \nonumber \\
&\quad = 
E\Big[
\widetilde{Z}_{[0,t]}^{h^--\varepsilon} \big(B^{0\to y-h^-(t)}_{[0,t]}\big)^{-1} \ ;\ 
B^{0\to y-h^-(t)}_{[0,t]} \in K_{[0,t]}\big(-\varepsilon, h^+-h^- \big)\Big]\nonumber \\
&\qquad \times P\big(h^-(t)+W_{[0,t]}(t) \in dy\big)\nonumber \\
&\quad = 
E\left[\widetilde{Z}_{[0, t]}^{h^-} 
(B_{[0, t]}^{0 \to y-h^-(t)}|_{K_{[0, t]}(-\varepsilon, h^+-h^- )})^{-1} \right]
P\Big(B^{0\to y-h^-(t)}_{[0,t]} \in K_{[0,t]}\big(-\varepsilon, h^+-h^- \big)\Big)
\nonumber \\
&\qquad \times P\big(h^-(t)+W_{[0,t]}(t) \in dy\big)\nonumber \\
&\quad = 
E\left[\widetilde{Z}_{[0, t]}^{h^-} 
(B_{[0, t]}^{0 \to y-h^-(t)}|_{K_{[0, t]}(-\varepsilon, h^+-h^- )})^{-1} \right]
P\Big(B_{[0, t]}^{0 \to y-h^-(t)}|_{K_{[0, t]}^+(-\varepsilon )} 
\in K_{[0, t]}^-(h^+-h^-)\Big)
\nonumber \\
&\qquad \times 
P\Big(B_{[0, t]}^{0 \to y-h^-(t)}\in K_{[0, t]}^+(-\varepsilon )\Big)
P\big(h^-(t)+W_{[0,t]}(t) \in dy\big). \label{Lem_proof_Gtrans_0t_dy_step1} 
\end{align}
On the other hand, we have
\begin{align}
&P\big(B_{[0, t]}^{0 \to y-h^-(t)}\in K_{[0, t]}^+(-\varepsilon )\big)
=P\big(-h^-(t)+B_{[0, t]}^{h^-(t) \to y}\in K_{[0, t]}^+(-\varepsilon )\big)\nonumber \\
&\quad =\frac{P\big(W_{[0, t]}(t)\in K_{[0, t]}^+(-\varepsilon ), \ h^-(t)+W_{[0, t]}(t)\in dy\big)}
{P\big(h^-(t)+W_{[0, t]}(t)\in dy\big)}.
\label{Lem_proof_Gtrans_0t_dy_step2} 
\end{align} 
Combining \eqref{Lem_proof_Gtrans_0t_dy_step1} and \eqref{Lem_proof_Gtrans_0t_dy_step2}, 
we obtain our assertion. 
\end{proof}

\begin{lem}\label{Lem_Gtrans_t1_dy}
Assume that $h^-$ and $h^+$ are $\mathbb{R}$-valued $C^2$-functions defined on $[0, 1]$ 
satisfying \eqref{LemGtrans_hpm_noncol_cond}. 
Then, for $0<t<1$, $\varepsilon>0$ and $y \in (h^-(t), h^+(t)+\varepsilon )$, we have
\begin{align}
&\frac{P(y + W_{[t, 1]}(1) \in db, 
y + W_{[t, 1]} \in K_{[t, 1]}(h^-, h^++\varepsilon))}
{P(W_{[t, 1]} \in K_{[t, 1]}^+(-\varepsilon)) db} \bigg\vert_{b=h^+(1)} 
\nonumber \\
&\quad = 
E\bigg[
\widetilde{Z}_{[t, 1]}^{h^+(1)-\overset{\leftarrow}{h}^+}\Big(B_{[t, 1]}^{0 \to h^+(t)-y}|
_{K_{[t, 1]}(-\varepsilon , \overset{\leftarrow}{h}^+-\overset{\leftarrow}{h}^-)} \Big)^{-1}
\bigg]
P\Big( B_{[t, 1]}^{0 \to h^+(t)-y}|_{K_{[t, 1]}^+(-\varepsilon )} \in 
K_{[t, 1]}^-\big(\overset{\leftarrow}{h}^+-\overset{\leftarrow}{h}^-\big)\Big)
\nonumber \\
&\qquad \times 
P\big( W_{[t, 1]}|_{K_{[t, 1]}^+(-\varepsilon )}(1) \in h^+(t) -dy \big)/dy,
\label{convergence_to_ht1_y_Lem_Gtrans_t1_dy}
\end{align}
where we define $\overset{\leftarrow}{h}^{\pm}\in C([t,1], \mathbb{R})$ as 
$\overset{\leftarrow}{h}^{\pm}(\cdot)=h^{\pm}(t+1-\cdot)$, respectively. 
\end{lem}
\begin{proof}
Using Lemma~\ref{Ap_Lem_Girsanov_all_type} \eqref{Girsanov_pinned_inv_formula}, we obtain
\begin{align}
&P\big( y+W_{[t,1]}(1) \in db, 
y+W_{[t,1]} \in K_{[t,1]}(h^-, h^++\varepsilon)\big)/db \nonumber \\
&\quad = 
E\Big[
\widetilde{Z}_{[t,1]}^{h^+(1)-\overset{\leftarrow}{h}^+} \big(B^{0\to b-h^+(1)+h^+(t)-y}_{[t,1]}\big)^{-1} 
\ ;\ \nonumber \\ 
&\qquad \qquad 
B^{0\to b-h^+(1)+h^+(t)-y}_{[t,1]}
\in K_{[t,1]}\big(b-h^+(1)-\varepsilon, b-h^+(1)+\overset{\leftarrow}{h}^+-\overset{\leftarrow}{h}^-\big)\Big]
\nonumber \\
&\qquad \times P\big(y+h^+(1)-h^+(t)+W_{[t,1]}(1) \in db\big)/db.
\nonumber
\end{align}
Thus, it follows that
\begin{align}
&\frac{P\big( y+W_{[t,1]}(1) \in db, 
y+W_{[t,1]} \in K_{[t,1]}(h^-, h^++\varepsilon)\big)}{P(W_{[t, 1]} \in K_{[t, 1]}^+(-\varepsilon)) db} \Big\vert_{b=h^+(1)}\nonumber \\
&\quad = 
E\Big[
\widetilde{Z}_{[t,1]}^{h^+(1)-\overset{\leftarrow}{h}^+} 
\Big(B^{0\to h^+(t)-y}_{[t,1]}
\vert_{K_{[t,1]}\big(-\varepsilon, \overset{\leftarrow}{h}^+-\overset{\leftarrow}{h}^-\big)}
\Big)^{-1} 
\Big]
\frac{P\Big(B^{0\to h^+(t)-y}_{[t,1]}
\in K_{[t,1]}\big(-\varepsilon, \overset{\leftarrow}{h}^+-\overset{\leftarrow}{h}^-\big)\Big)}
{P(W_{[t, 1]} \in K_{[t, 1]}^+(-\varepsilon)) } \nonumber \\
&\qquad \times \frac{P\big(y+h^+(1)-h^+(t)+W_{[t,1]}(1) \in db\big)}{db}\Big\vert_{b=h^+(1)}
\nonumber \\
&\quad = 
E\Big[
\widetilde{Z}_{[t,1]}^{h^+(1)-\overset{\leftarrow}{h}^+} 
\Big(B^{0\to h^+(t)-y}_{[t,1]}
\vert_{K_{[t,1]}\big(-\varepsilon, \overset{\leftarrow}{h}^+-\overset{\leftarrow}{h}^-\big)}
\Big)^{-1} 
\Big]
P\Big( B_{[t, 1]}^{0 \to h^+(t)-y}|_{K_{[t, 1]}^+(-\varepsilon )} \in 
K_{[t, 1]}^-\big(\overset{\leftarrow}{h}^+-\overset{\leftarrow}{h}^-\big)\Big)
\nonumber \\
&\qquad \times 
\frac{P\Big(B^{0\to h^+(t)-y}_{[t,1]}
\in K_{[t,1]}^+\big(-\varepsilon \big)\Big)}
{P(W_{[t, 1]} \in K_{[t, 1]}^+(-\varepsilon)) }
\cdot 
\frac{P\big(y+h^+(1)-h^+(t)+W_{[t,1]}(1) \in db\big)}{db}\Big\vert_{b=h^+(1)}.
\label{Lem_proof_Gtrans_t1_dy_step2}
\end{align}
On the other hand, we have
\begin{align}
&P\Big(B^{0\to h^+(t)-y}_{[t,1]}
\in K_{[t,1]}^+\big(-\varepsilon \big)\Big)
=P\Big(h^+(t)-B^{h^+(t)\to y}_{[t,1]}
\in K_{[t,1]}^+\big(-\varepsilon \big)\Big)
\nonumber \\
&\quad =\frac{P\Big(h^+(t)-(h^+(t)+W_{[t,1]})
\in K_{[t,1]}^+\big(-\varepsilon \big), \ h^+(t)+W_{[t,1]}(1)\in dy\Big)}
{P(h^+(t)+W_{[t,1]}(1)\in dy)}
\nonumber \\
&\quad =\frac{P\Big(-W_{[t,1]}
\in K_{[t,1]}^+\big(-\varepsilon \big), \  h^+(t)-(-W_{[t,1]}(1))\in dy\Big)}
{P(h^+(t)+W_{[t,1]}(1)\in dy)}
\nonumber \\
&\quad =\frac{P\Big(W_{[t,1]}
\in K_{[t,1]}^+\big(-\varepsilon \big), \  h^+(t)-W_{[t,1]}(1)\in dy\Big)}
{P(h^+(t)-W_{[t,1]}(1)\in dy)}
\label{Lem_proof_Gtrans_t1_dy_step3}
\end{align} 
and
\begin{align}
\frac{P\big(y+h^+(1)-h^+(t)+W_{[t,1]}(1) \in db\big)}{db}\Big\vert_{b=h^+(1)}
=P(h^+(t)-W_{[t,1]}(1)\in dy)/dy.
\label{Lem_proof_Gtrans_t1_dy_step4}
\end{align}
Combining \eqref{Lem_proof_Gtrans_t1_dy_step2}, 
\eqref{Lem_proof_Gtrans_t1_dy_step3} and \eqref{Lem_proof_Gtrans_t1_dy_step4}, 
we obtain \eqref{convergence_to_ht1_y_Lem_Gtrans_t1_dy}.
\end{proof}

\begin{lem}\label{Lem_L1_convergence}
For $0\leq s<u<t\leq 1$, we have
\begin{align}\label{eq_Lem_L1_conv}
&\lim_{\varepsilon \downarrow 0} \int_{\mathbb{R}} \left| 
P\left( W_{[s, t]} |_{K_{[s, t]}^+(-\varepsilon )}(u) \in dy \right) \big/ dy 
- P\left( W_{[s, t]}^+(u) \in dy \right) \big/ dy \right| dy = 0.
\end{align}
\end{lem}
\begin{proof}
From the scaling property we can easily show that it suffices to prove \eqref{eq_Lem_L1_conv} in the case when $s=0$ and $t=1$. 
Let $0<u<1$. Using Lemma~\ref{Ap_BM_Density_Formulas_type1}, we have
\begin{align}
&P\left( W_{[0,1]}(u)\in dy, \ m(W_{[0,1]})\geq -\varepsilon \right)
\nonumber \\
&\quad = P(W_{[0,1]}(u)\in dy, \ m_u(W_{[0,1]})\geq -\varepsilon) \cdot P\left( m_{1-u}(W_{[0,1]})\geq -y-\varepsilon \right)
\nonumber \\
&\quad = 2\left( n_u(y) - n_u(y+2\varepsilon) \right) \int_0^{y+\varepsilon} n_{1-u}(x)dx dy, 
\quad (y \geq -\varepsilon).
\label{prepare_for_densityaeconv_meander}
\end{align}
Therefore, using \eqref{prepare_for_densityaeconv_meander} and L'H\^{o}pital's rule, we can deduce for $y>0$ that
\begin{align}
&P\left( W_{[0, 1]} |_{K_{[0, 1]}^+(-\varepsilon )}(u) \in dy \right)/dy 
=\frac{n_u(y) - n_u(y+2\varepsilon) }{\int_0^{\varepsilon} n_1(x)dx}
\int_0^{y+\varepsilon} n_{1-u}(x)dx 
\nonumber \\
&\quad \to 2\sqrt{2\pi}\frac{y n_u(y)}{u} \int_0^y n_{1-u}(x)dx
=P\left( W^+(u) \in dy \right)/dy, \quad \varepsilon \downarrow 0.
\label{density_ae_conv_meander}
\end{align}
Applying \eqref{density_ae_conv_meander} and Scheff\'{e}'s lemma, 
we obtain \eqref{eq_Lem_L1_conv}. 
\end{proof}

\section{Proofs of the main results in Subsection~\ref{subsection_Moving_between_2curves}}
\label{section_proof_Moving_between_2curves}

In this section, we prove the main results in Subsection~\ref{subsection_Moving_between_2curves}.

\subsection{Proof of Theorem~\ref{Thm_Def_and_Decomp_curved_Moving}}

In this subsection, we assume that all $X_{[s, t]}^{x,y,(g^-, g^+)}$ are independent. 
For each $\mathbb{R}$-valued bounded continuous function $G$ on $C([0,1], \mathbb{R})$ and $\varepsilon>0$, we define
\begin{align*}
I(\varepsilon, G):=
E[G(W_{[0,1]})\ ;\ W_{[0,1]}(1) \in d\widetilde{b}, W_{[0,1]} \in K_{[0, 1]}(g^--\eta^-(\varepsilon), g^++\eta^+(\varepsilon))]/d\widetilde{b} \ \big\vert_{\widetilde{b}=b}.
\end{align*}
Then, we have
\begin{equation}\label{suzue1}
E[F(B_{[0,1]}^{0 \to b} |_{K_{[0, 1]}(g^--\eta^-(\varepsilon), g^++\eta^+(\varepsilon))})] 
=\frac{I(\varepsilon, F)}{I(\varepsilon, 1)}.
\end{equation}
Further, by Lemma~\ref{Lem_MarkovDecomp_B_Bridge}, we obtain
\begin{align}
I(\varepsilon, F) 
&= \int_{g^-(t)-\eta^-(\varepsilon)}^{g^+(t)+\eta^+(\varepsilon)}  
E\left[F\left(X_{[0, t]}^{0, y,(g^--\eta^-(\varepsilon), g^++\eta^+(\varepsilon))} \oplus_{t} 
X_{[t, 1]}^{y, b,(g^--\eta^-(\varepsilon), g^++\eta^+(\varepsilon))}\right)\right] 
\label{Moving_MarkovDecomp_single} \\
&\qquad \times 
P(W_{[0, t]} \in K_{[0, t]}(g^--\eta^-(\varepsilon), g^++\eta^+(\varepsilon)), W_{[0, t]}(t) \in dy) \nonumber \\
&\qquad \times 
P(y + W_{[t, 1]} \in K_{[t, 1]}(g^--\eta^-(\varepsilon), g^++\eta^+(\varepsilon)), y + W_{[t, 1]}(1) \in d\widetilde{b})/d\widetilde{b}\ \big\vert_{\widetilde{b}=b}.
\nonumber
\end{align}
It follows from \eqref{Moving_MarkovDecomp_single}, Lemmas~\ref{Lem_Gtrans_0t_dy}, \ref{Lem_Gtrans_t1_dy}, 
Proposition~\ref{Prop_Boundary_purturb_WeakConv_for_Meander_BES3bridge}, 
Lemma~\ref{Lem_L1_convergence}, 
Lemma~\ref{Lem_for_H_gm_to_gp} 
and Lemma~\ref{Ap_Lem_weakconv_product} that
\begin{align}
I(F):=&\lim_{\varepsilon \downarrow 0} 
\frac{I(\varepsilon, F)}{\eta^-(\varepsilon) \eta^+(\varepsilon)}
\nonumber \\
=&\lim_{\varepsilon \downarrow 0} 
\frac{I(\varepsilon, F)}{P(W_{[0, t]} \in K_{[0, t]}^+(-\eta^-(\varepsilon) ))P(W_{[t, 1]} \in K_{[t, 1]}^+(-\eta^+(\varepsilon)))} 
\nonumber \\
&\quad \times \lim_{\varepsilon \downarrow 0} 
\frac{P(W_{[0, t]} \in K_{[0, t]}^+(-\eta^-(\varepsilon) ))P(W_{[t, 1]} \in K_{[t, 1]}^+(-\eta^+(\varepsilon)))}
{\eta^-(\varepsilon) \eta^+(\varepsilon)}
\nonumber \\
=&\frac{2}{\pi} 
\int_{g^-(t)}^{g^+(t)}
E\left[F\left(X_{[0, t]}^{0, y,(g^-, g^+)} \oplus_{t} X_{[t, 1]}^{y, b,(g^-, g^+)}\right)\right] 
\frac{1}{\sqrt{t}}q^{(g^-,g^+), (\uparrow)}_{[0,t]}(y)\frac{1}{\sqrt{1-t}}q^{(g^-,g^+), (\downarrow)}_{[t,1]}(y)dy .
\label{IF_t_decomp}
\end{align}
Here, note that 
\begin{align}
&\lim_{\varepsilon \downarrow 0}
E\left[ \widetilde{Z}_{[0,t]}^{g^-}
\left(B_{[0,t]}^{0\to y-g^-(t)}\big\vert_{K_{[0,t]}(-\eta^-(\varepsilon), g^+-g^-+\eta^+(\varepsilon))}\right)^{-1}
\right]
=E\left[ \widetilde{Z}_{[0,t]}^{g^-}
\left(r_{[0,t]}^{0\to y-g^-(t)}\big\vert_{K_{[0,t]}^-(g^+-g^-)}\right)^{-1}
\right]
\label{conv_Exp_Z_inv}
\end{align}
was used to obtain \eqref{IF_t_decomp}. 
Because $\displaystyle \widetilde{Z}_{[0,t]}^{g^-}(\cdot)^{-1}$ in \eqref{conv_Exp_Z_inv} 
can be regarded as a bounded continuous function on $C([0, t], \mathbb{R})$, 
we can prove \eqref{conv_Exp_Z_inv}. 
According to \eqref{IF_t_decomp}, we get
\begin{align}
&\frac{2}{\pi} 
\int_{g^-(t)}^{g^+(t)}
\frac{1}{\sqrt{t}}q^{(g^-,g^+), (\uparrow)}_{[0,t]}(y)\frac{1}{\sqrt{1-t}}q^{(g^-,g^+), (\downarrow)}_{[t,1]}(y)dy 
\nonumber \\
&\quad =I(1)=\lim_{\varepsilon \downarrow 0} 
\frac{P(W(1) \in d\widetilde{b}, W \in K_{[0, 1]}(g^--\eta^-(\varepsilon), g^++\eta^+(\varepsilon)))}{\eta^-(\varepsilon) \eta^+(\varepsilon)d\widetilde{b}}
\ \Big\vert_{\widetilde{b}=b}
\nonumber \\
&\quad =n_1(b)\lim_{\varepsilon \downarrow 0} 
\frac{P\big(B_{[0, 1]}^{0 \to b} \in K_{[0, 1]}(g^--\eta^-(\varepsilon), g^++\eta^+(\varepsilon))\big)}{\eta^-(\varepsilon) \eta^+(\varepsilon)} 
=\frac{2}{\pi} C_{g^-,g^+} . 
\label{I1_t_decomp}
\end{align}
Combining \eqref{suzue1}, \eqref{IF_t_decomp} and \eqref{I1_t_decomp}, we obtain
\begin{align*}
C_{g^-,g^+}=\int_{g^-(t)}^{g^+(t)}
\frac{1}{\sqrt{t}}q^{(g^-,g^+), (\uparrow)}_{[0,t]}(y)\frac{1}{\sqrt{1-t}}q^{(g^-,g^+), (\downarrow)}_{[t,1]}(y)dy \in (0, \infty)
\end{align*}
and
\begin{align*}
\lim_{\varepsilon \downarrow 0} 
E[F(B_{[0, 1]}^{0 \to b} |_{K_{[0, 1]}(g^--\eta^-(\varepsilon), g^++\eta^+(\varepsilon))})]
= \frac{I(F)}{I(1)} 
= \int_{g^-(t)}^{g^+(t)} 
E\left[ F\left(X_{[0,t]}^{0,y,(g^-, g^+)} \oplus_{t} X_{[t,1]}^{y,b,(g^-, g^+)} \right) \right] 
h(t, y)dy .
\end{align*}
Therefore, we can define the probability measure $\widetilde{P}_H$
on $(C([0,1], \mathbb{R}), \mathcal{B}(C([0,1], \mathbb{R})))$ as
\begin{align*}
\widetilde{P}_H(A):= \int_{g^-(t)}^{g^+(t)} 
P\left( X_{[0,t]}^{0,y,(g^-, g^+)} \oplus_{t} X_{[t,1]}^{y,b,(g^-, g^+)} \in A \right) h(t,y) dy
\quad (A\in \mathcal{B}(C([0,1], \mathbb{R}))),
\end{align*}
and there exists an $\mathbb{R}$-valued 
continuous stochastic process $H^{g^-\to g^+}=\{H^{g^-\to g^+}(t)\}_{t\in [0,1]}$
that satisfies (\ref{curvedMoving_def_eq1}) and (\ref{6.3}). 
Thus, a limit argument on $F$ yields 
\begin{align*}
P(H^{g^-\to g^+}(t) \in dy) = h(t, y)dy\quad (y \in (g^-(t), g^+(t))).
\end{align*}

On the other hand, by Lemma~\ref{Lem_MarkovDecomp_B_Bridge}, we obtain
\begin{align}
&I(\varepsilon, F) 
= \int_{g^-(t_2)-\eta^-(\varepsilon)}^{g^+(t_2)+\eta^+(\varepsilon)} dy_2
\int_{g^-(t_1)-\eta^-(\varepsilon)}^{g^+(t_1)+\eta^+(\varepsilon)} dy_1 
\label{Moving_MarkovDecomp_double} \\
&\quad \times E\left[F\left(X_{[0, t_1]}^{0, y_1,(g^--\eta^-(\varepsilon), g^++\eta^+(\varepsilon))} \oplus_{t_1} 
X_{[t_1,t_2]}^{y_1, y_2,(g^--\eta^-(\varepsilon), g^++\eta^+(\varepsilon))} \oplus_{t_2} 
X_{[t_2, 1]}^{y_2, b,(g^--\eta^-(\varepsilon), g^++\eta^+(\varepsilon))}\right)\right] 
\nonumber \\
&\quad \times 
P(y_2 + W_{[t_2, 1]} \in K_{[t_2, 1]}(g^--\eta^-(\varepsilon), g^++\eta^+(\varepsilon)), 
y_2 + W_{[t_2, 1]}(1) \in d\widetilde{b})/d\widetilde{b} \ \big\vert_{\widetilde{b}=b} 
\nonumber \\
&\quad \times 
P(y_1 + W_{[t_1,t_2]} \in K_{[t_1,t_2]}(g^--\eta^-(\varepsilon), g^++\eta^+(\varepsilon)), 
y_1 + W_{[t_1,t_2]}(t_2) \in dy_2)/dy_2 \nonumber \\
&\quad \times 
P(W_{[0, t_1]} \in K_{[0, t_1]}(g^--\eta^-(\varepsilon), g^++\eta^+(\varepsilon)), W_{[0, t_1]}(t_1) \in dy_1)/dy_1.
\nonumber 
\end{align}
By \eqref{Moving_MarkovDecomp_double}, Lemmas~\ref{Lem_Gtrans_0t_dy}, \ref{Lem_Gtrans_t1_dy}, 
Proposition~\ref{Prop_Boundary_purturb_WeakConv_for_Meander_BES3bridge}, 
Lemma~\ref{Lem_L1_convergence}, 
Lemma~\ref{Lem_for_H_gm_to_gp} 
and Lemma~\ref{Ap_Lem_weakconv_product}, 
$I(F)$ satisfies
\begin{align}
I(F)=&\lim_{\varepsilon \downarrow 0} 
\frac{I(\varepsilon, F)}{P(W_{[0, t_1]} \in K_{[0, t_1]}^+(-\eta^-(\varepsilon)))P(W_{[t_2, 1]} \in K_{[t_2, 1]}^+(-\eta^+(\varepsilon)))} 
\nonumber \\
&\times \lim_{\varepsilon \downarrow 0} 
\frac{P(W_{[0, t_1]} \in K_{[0, t_1]}^+(-\eta^-(\varepsilon)))P(W_{[t_2, 1]} \in K_{[t_2, 1]}^+(-\eta^+(\varepsilon)))}
{\eta^-(\varepsilon)\eta^+(\varepsilon)} \nonumber \\
=&\frac{2}{\pi}
\int_{g^-(t_1)}^{g^+(t_1)} \int_{g^-(t_2)}^{g^+(t_2)} 
E\left[F\left(X_{[0,t_1]}^{0,y_1,(g^-, g^+)} \oplus_{t_1} 
X_{[t_1,t_2]}^{y_1, y_2,(g^-, g^+)} \oplus_{t_2} 
X_{[t_2, 1]}^{y_2, b,(g^-, g^+)}\right)\right] \nonumber \\
&\qquad \quad \times 
\frac{1}{\sqrt{t_1}}q^{(g^-,g^+), (\uparrow)}_{[0,t_1]}(y_1)p^{(g^-,g^+)}_{[t_1,t_2]}(y_1, y_2) \frac{1}{\sqrt{1-t_2}}q^{(g^-,g^+), (\downarrow)}_{[t_2,1]}(y_2) dy_1 dy_2 
\nonumber \\
=&\frac{2}{\pi}C_{g^-,g^+}
\int_{g^-(t_1)}^{g^+(t_1)} \int_{g^-(t_2)}^{g^+(t_2)} 
E\left[F\left(X_{[0,t_1]}^{0,y_1,(g^-, g^+)} \oplus_{t_1} 
X_{[t_1,t_2]}^{y_1, y_2,(g^-, g^+)} \oplus_{t_2} 
X_{[t_2, 1]}^{y_2, b,(g^-, g^+)}\right) \right] \label{IF_doubleIntg_h_t1t2}\\
&\qquad \qquad \qquad \qquad \quad \times 
h(t_1, y_1) h(t_1, y_1, t_2, y_2) dy_1 dy_2 .
\nonumber 
\end{align}
It follows from \eqref{suzue1}, \eqref{I1_t_decomp} and \eqref{IF_doubleIntg_h_t1t2} that
\begin{align*}
&E[F(H^{g^-\to g^+})] 
=\lim_{\varepsilon \downarrow 0} 
E[F(B_{[0,1]}^{0 \to b} |_{K_{[0, 1]}(g^--\eta^-(\varepsilon), g^++\eta^+(\varepsilon))})]\\
&\quad =\int_{g^-(t_1)}^{g^+(t_1)} \int_{g^-(t_2)}^{g^+(t_2)} 
E\left[F\left(X_{[0,t_1]}^{0,y_1,(g^-, g^+)} \oplus_{t_1} 
X_{[t_1,t_2]}^{y_1, y_2,(g^-, g^+)}\oplus_{t_2} 
X_{[t_2, 1]}^{y_2, b,(g^-, g^+)}\right)\right] \\
&\qquad \qquad \qquad \qquad \times h(t_1, y_1)h(t_1, y_1, t_2, y_2) dy_1 dy_2 .
\end{align*}
Hence, (\ref{6.4}) holds. 
Similarly, using a limit argument on $F$, we can deduce for 
$y_1 \in (g^-(t_1), g^+(t_1))$ and $y_2 \in (g^-(t_2), g^+(t_2))$ that
\begin{align}
P(H^{g^-\to g^+}(t_1) \in dy_1, H^{g^-\to g^+}(t_2) \in dy_2) &= h(t_1, y_1)h(t_1, y_1, t_2, y_2) dy_1 dy_2, 
\nonumber \\
P(H^{g^-\to g^+}(t_2) \in dy_2 \ \vert \ H^{g^-\to g^+}(t_1)=y_1) &=h(t_1, y_1, t_2, y_2) dy_2.
\label{Proof_H_transdensity}
\end{align}

If we define $I_{t_1}(\varepsilon , y_1)$ $(y_1 \in (g^-(t_1), g^+(t_1)))$ to be
\begin{align*}
&I_{t_1}(\varepsilon ,y_1):= 
P(y_1 + W_{[t_1, 1]} \in K_{[t_1, 1]}(g^--\eta^-(\varepsilon), g^++\eta^+(\varepsilon)), y_1 + W_{[t_1, 1]}(1) \in d\widetilde{b})/d\widetilde{b} \ \big\vert_{\widetilde{b}=b},
\end{align*}
then 
\begin{align}
I_{t_1}(y_1):=&\lim_{\varepsilon \downarrow 0} \frac{I_{t_1}(\varepsilon ,y_1)}{\eta^+(\varepsilon)}
\nonumber \\
=&\lim_{\varepsilon \downarrow 0} \frac{I_{t_1}(\varepsilon ,y_1)}{P(W_{[t_1, 1]} \in K_{[t_1, 1]}^+(-\eta^+(\varepsilon)))}
\lim_{\varepsilon \downarrow 0} \frac{P(W_{[t_1, 1]} \in K_{[t_1, 1]}^+(-\eta^+(\varepsilon)))}{\eta^+(\varepsilon)}
\nonumber \\
=&\sqrt{\frac{2}{\pi}}\frac{1}{\sqrt{1-t_1}} q^{(g^-,g^+), (\downarrow)}_{[t_1,1]}(y_1)
\qquad (y_1 \in (g^-(t_1), g^+(t_1)))
\label{It1y1_expression_1}
\end{align}
holds by Lemma~\ref{Lem_Gtrans_t1_dy} and 
Proposition~\ref{Prop_Boundary_purturb_WeakConv_for_Meander_BES3bridge}. 
On the other hand, because we have
\begin{align*}
I_{t_1}(\varepsilon , y_1) 
&= \int_{g^-(t_2)-\eta^-(\varepsilon)}^{g^+(t_2)+\eta^+(\varepsilon)}  
P(y_1 + W_{[t_1,t_2]} \in K_{[t_1,t_2]}(g^--\eta^-(\varepsilon), g^++\eta^+(\varepsilon)), 
y_1 + W_{[t_1,t_2]}(t_2) \in dy_2) \\
&\qquad  \times 
P(y_2 + W_{[t_2, 1]} \in K_{[t_2, 1]}(g^--\eta^-(\varepsilon), g^++\eta^+(\varepsilon)), 
y_2 + W_{[t_2, 1]}(1) \in d\widetilde{b})/d\widetilde{b} \ \big\vert_{\widetilde{b}=b}, 
\end{align*}
for $y_1 \in (g^-(t_1), g^+(t_1))$, 
it follows from Lemma~\ref{Lem_Gtrans_t1_dy}, 
Proposition~\ref{Prop_Boundary_purturb_WeakConv_for_Meander_BES3bridge} 
and Lemma~\ref{Lem_L1_convergence} that
\begin{align}
I_{t_1}(y_1)
&=\lim_{\varepsilon \downarrow 0} \frac{I_{t_1}(\varepsilon ,y_1)}{P(W_{[t_2, 1]} \in K_{[t_2, 1]}^+(-\eta^+(\varepsilon)))}
\lim_{\varepsilon \downarrow 0} \frac{P(W_{[t_2, 1]} \in K_{[t_2, 1]}^+(-\eta^+(\varepsilon)))}{\eta^+(\varepsilon)}
\nonumber \\
&= \sqrt{\frac{2}{\pi}} \int_{g^-(t_2)}^{g^+(t_2)} 
p^{(g^-,g^+)}_{[t_1,t_2]} (y_1, y_2) \frac{1}{\sqrt{1-t_2}} q^{(g^-,g^+), (\downarrow)}_{[t_2,1]}(y_2) dy_2 
\qquad (y_1 \in (g^-(t_1), g^+(t_1))).
\label{It1y1_expression_2}
\end{align} 
Combining \eqref{It1y1_expression_1} and \eqref{It1y1_expression_2}, we obtain
\begin{align*}
\frac{1}{\sqrt{1-t_1}} q^{(g^-,g^+), (\downarrow)}_{[t_1,1]}(y_1)
&=\int_{g^-(t_2)}^{g^+(t_2)} 
p^{(g^-,g^+)}_{[t_1,t_2]} (y_1, y_2) \frac{1}{\sqrt{1-t_2}} q^{(g^-,g^+), (\downarrow)}_{[t_2,1]}(y_2) dy_2 
\qquad (y_1 \in (g^-(t_1), g^+(t_1)))
\end{align*}
and
\begin{align}
\int_{g^-(t_2)}^{g^+(t_2)} h(t_1, y_1, t_2, y_2) dy_2 = 1
\qquad (y_1 \in (g^-(t_1), g^+(t_1))).
\label{ht1y1t2y2_density}
\end{align}
Assume that $t_3$ satisfies $0<t_1<t_2<t_3<1$. 
Because we have
\begin{align}\label{Chapman_id_p_y1_y3}
p^{(g^-,g^+)}_{[t_1,t_3]} (y_1, y_3)=
\int_{g^-(t_2)}^{g^+(t_2)} 
p^{(g^-,g^+)}_{[t_1,t_2]} (y_1, y_2)
p^{(g^-,g^+)}_{[t_2,t_3]} (y_2, y_3) dy_2
\end{align}
for $y_1 \in (g^-(t_1), g^+(t_1))$ and $y_3 \in (g^-(t_3), g^+(t_3))$, 
we can deduce for 
$y_1 \in (g^-(t_1), g^+(t_1))$ and $y_3 \in (g^-(t_3), g^+(t_3))$ that 
the following Chapman--Kolmogorov identity holds: 
\begin{align}\label{ht1y1t2y2_transdensity}
h(t_1, y_1, t_3, y_3) 
=\int_{g^-(t_2)}^{g^+(t_2)} h(t_1, y_1, t_2, y_2) h(t_2, y_2, t_3, y_3) dy_2 .
\end{align}
Therefore, \eqref{Proof_H_transdensity}, \eqref{ht1y1t2y2_density} and \eqref{ht1y1t2y2_transdensity} imply that 
$H^{g^-\to g^+}=\{H^{g^-\to g^+}(t)\}_{t\in [0,1]}$ is a Markov process.

\subsection{Proof of Corollary~\ref{Cor_Decomp_flat_Moving}}

Let $0<s<t<1$ and $x, y \in (0, b)$. 
Then, by Corollary~\ref{cor_scaled_BES3bridge_max_dist}, we obtain
\begin{align*}
q^{(0,b), (\uparrow)}_{[0,t]}(y)
= \frac{t J^{(b)}(t,y)}{2y n_t(y)} P\big(\sqrt{t}W^+(1) \in dy\big)/dy 
= \sqrt{ \frac{\pi t}{2} } J^{(b)}(t,y)
\end{align*}
and
\begin{align*}
q^{(0,b), (\downarrow)}_{[t,1]}(y)
&=\frac{(1-t) J^{(b)}(1-t, b-y)}{2(b-y) n_{1-t}(b-y)} P\big(\sqrt{1-t}W^+(1) \in b-dy\big)/dy 
= \sqrt{ \frac{\pi (1-t)}{2} }J^{(b)}(1-t, b-y). 
\end{align*}
Further, by Lemma~\ref{Ap_Lem_BM_Density_Diff_Formulae_2} and L'H\^{o}pital's rule, it holds that
\begin{align*}
C_{0,b}
&=\frac{\pi n_1(b) }{2}
\lim_{\varepsilon \downarrow 0} 
\frac{P(B_{[0, 1]}^{0 \to b} \in K_{[0, 1]}(-\varepsilon, b+\varepsilon))}{\varepsilon^2}\\
&=\frac{\pi }{2}
\lim_{\varepsilon \downarrow 0} 
\frac{P(W(1)\in d\widetilde{b}, \ -\varepsilon \leq m(W)<M(W) \leq b+\varepsilon)}{\varepsilon^2 d\widetilde{b}}\ \Big\vert_{\widetilde{b}=b}
=\frac{\pi }{2} \overline{J}^{(b)}(1, b).
\end{align*}
On the other hand, $p^{(0,b)}_{[s,t]}(x, y)$ is written as
\begin{align*}
&p^{(0,b)}_{[s,t]}(x, y)
=P(x + W_{[s,t]} \in K_{[s,t]}(0, b), x + W_{[s,t]}(t) \in dy)/dy \\
&\quad =P\left(W(t-s) \in dy-x, -x \leq m_{t-s}(W) < M_{t-s}(W) \leq b-x \right) /dy\\
&\quad = \sum_{k=-\infty}^{\infty} 
\left( n_{t-s}(y-x+2kb) 
- n_{t-s}(2(k+1)b-y-x) \right) =J^{(b)}(t-s, x, y)
\end{align*}
by \eqref{BM_Density_No2_Wt_mt_Mt}. 
Therefore, by Theorem~\ref{Thm_Def_and_Decomp_curved_Moving}, we obtain
\begin{align*}
P\left( H^{0 \to b}(t) \in dy \right) 
&=\dfrac{\frac{1}{\sqrt{t}}q^{(0,b), (\uparrow)}_{[0,t]}(y)\frac{1}{\sqrt{1-t}}q^{(0,b), (\downarrow)}_{[t,1]}(y)}{C_{0,b}}
=\dfrac{ J^{(b)}(t,y) J^{(b)}(1-t, b-y)}{\overline{J}^{(b)}(1, b)}, \\
P\left( H^{0 \to b}(t) \in dy~|~H^{0 \to b}(s)=x \right) 
&=\dfrac{p^{(0,b)}_{[s,t]}(x, y) \frac{1}{\sqrt{1-t}}q^{(0,b), (\downarrow)}_{[t,1]}(y)}{\frac{1}{\sqrt{1-s}}q^{(0,b), (\downarrow)}_{[s,1]}(x) } 
=\dfrac{J^{(b)}(t-s, x, y)J^{(b)}(1-t, b-y)}{J^{(b)}(1-s, b-x)}.
\end{align*}

\subsection{Proof of Corollary~\ref{Cor_30_curve_max_0t}}

Let $A_i$ $(i=1,2)$ be closed subsets of $C([0,1], \mathbb{R})$ given by
\begin{align*}
&A_1:=\left\{ w \in C([0,1], \mathbb{R})\ \big\vert \ \min_{u\in [0,t]}\left\{ g(u)-w(u) \right\}= 0 \right\}, \\
&A_2:=\left\{ w \in C([0,1], \mathbb{R})
\ \big\vert \ \min_{u\in [0,t]}\left\{ g(u)-w(u) \right\} \geq 0,\ w(t) \leq z \right\}.
\end{align*}
Remark~\ref{Rem_F_canbetakento_Indicator_Moving} implies that 
Theorem~\ref{Thm_Def_and_Decomp_curved_Moving} 
can be applied for $F=1_{A_i}$ $(i=1,2)$. Thus, we obtain
\begin{align}
&P\left( \min_{u\in [0,t]} \left\{ g(u)-H^{g^-\to g^+}(u) \right\} = 0 \right)
=\int_{g^-(t)}^{g(t)} P\left( X_{[0,t]}^{0,y,(g^-, g^+)}\in \partial K_{[0,t]}^-(g) \right) h(t,y)dy,
\label{Moving_Decomp_touch_g} \\
&P\left( \min_{u\in [0,t]}\left\{ g(u)-H^{g^-\to g^+}(u) \right\} \geq 0, H^{g^-\to g^+}(t) \leq z \right) 
=\int_{g^-(t)}^z
P\left( X_{[0,t]}^{0,y,(g^-, g^+)}\in K_{[0,t]}^-(g)
\right) h(t,y)dy .
\label{Moving_Decomp_leq_g}
\end{align}
It follows from Lemma~\ref{Lem_Boundary_Prob_for_Meander_BES3bridge} that
\begin{align}\label{eq_E_1_partialKg_eq_0}
P\Big(r_{[0,t]}^{0 \to y-g^-(t)} \in \partial K_{[0,t]}^-(g-g^-)\Big) =0 \qquad (g^-(t)<y<g(t)).
\end{align}
Combining Remark~\ref{Rem_F_canbetakento_Indicator_MeanderBESbridge} and \eqref{eq_E_1_partialKg_eq_0}, 
we have 
\begin{align}\label{Cor2_eq_prob_zero}
P\left( X_{[0,t]}^{0,y,(g^-, g^+)}\in \partial K_{[0,t]}^-(g) \right)=0
\qquad (g^-(t)<y<g(t)).
\end{align} 
Thus, by \eqref{Moving_Decomp_touch_g} and \eqref{Cor2_eq_prob_zero}, 
we obtain \eqref{eq_curve_max_0t_bdry}.
On the other hand, we have
\begin{align}
&P\left( X_{[0,t]}^{0,y,(g^-, g^+)}\in K_{[0,t]}^-(g) \right)
\nonumber \\
&\quad = \frac{ E\left[ \widetilde{Z}_{[0,t]}^{g^-} \Big(r_{[0,t]}^{0 \to y-g^-(t)} \big|_{K_{[0,t]}^-(g-g^-)} \Big)^{-1} \right]}
{E\left[\widetilde{Z}_{[0,t]}^{g^-} \Big(r_{[0,t]}^{0 \to y-g^-(t)} \big|_{K_{[0,t]}^-(g^+-g^-)}\Big)^{-1} \right] }
\cdot \frac{P\left( r_{[0,t]}^{0 \to y-g^-(t)} \in K_{[0,t]}^-(g-g^-)\right)}{P\left(r_{[0,t]}^{0 \to y-g^-(t)} \in K_{[0,t]}^-(g^+-g^-)\right)}
\qquad (g^-(t)<y<g(t))
\label{ProbChange_curve_max_0t_g}
\end{align}
by Remark~\ref{Rem_F_canbetakento_Indicator_MeanderBESbridge}. 
Combining \eqref{Moving_Decomp_leq_g} and \eqref{ProbChange_curve_max_0t_g}, 
we obtain \eqref{eq_curve_max_0t_jointdist}.

\subsection{Proof of Corollary~\ref{Cor_30_curve_min_t1}}

The proof of Corollary~\ref{Cor_30_curve_min_t1} is quite similar to the proof of Corollary~\ref{Cor_30_curve_max_0t}. 
Therefore, below we outline the proof of Corollary~\ref{Cor_30_curve_min_t1}. 
Let $b = g^+(1)$. Then, we obtain
\begin{align}
&P\left( \min_{u\in [t,1]} \left\{ H^{g^-\to g^+}(u)-g(u) \right\} = 0 \right)
=\int_{g(t)}^{g^+(t)} P\left( X_{[t,1]}^{y,b,(g^-, g^+)}\in \partial K_{[t,1]}^+(g) \right) h(t,y)dy,
\label{Moving_Decomp_geq_g_touch_g} \\
&P\left( \min_{u\in [t,1]}\left\{ H^{g^-\to g^+}(u) - g(u) \right\} \geq 0, H^{g^-\to g^+}(t) \leq z \right) 
=\int_{g(t)}^z
P\left( X_{[t,1]}^{y,b,(g^-, g^+)} \in K_{[t,1]}^+(g)
\right) h(t,y)dy .
\label{Moving_Decomp_geq_g}
\end{align}
On the other hand, we have 
\begin{align}
&P\left( X_{[t,1]}^{y,b,(g^-, g^+)}\in \partial K_{[t,1]}^+(g) \right)=0, \label{Cor3_eq_prob_zero}\\
&P\left( X_{[t,1]}^{y, b,(g^-, g^+)} \in K_{[t,1]}^+(g) \right) \nonumber \\
&\quad = 
\frac{ E\left[ 
\widetilde{Z}_{[t,1]}^{b-\overset{\leftarrow}{g}^+} 
\Big(r_{[t,1]}^{0 \to g^+(t)-y} \big|_{K_{[t,1]}^-(\overset{\leftarrow}{g}^+-\overset{\leftarrow}{g})}\Big)^{-1}
\right]}
{E\left[\widetilde{Z}_{[t,1]}^{b-\overset{\leftarrow}{g}^+} 
\Big(r_{[t,1]}^{0 \to g^+(t)-y} \big|_{K_{[t,1]}^-(\overset{\leftarrow}{g}^+-\overset{\leftarrow}{g}^-)}\Big)^{-1}
\right]}
\cdot
\frac{ P\left( r_{[t,1]}^{0 \to g^+(t)-y} \in K_{[t,1]}^-(\overset{\leftarrow}{g}^+-\overset{\leftarrow}{g})\right)}
{P\left( r_{[t,1]}^{0 \to g^+(t)-y} \in K_{[t,1]}^-(\overset{\leftarrow}{g}^+-\overset{\leftarrow}{g}^-)\right)}
\label{Change_curve_gmin_to_g}
\end{align}
for $y \in (g(t), g^+(t))$. 
Combining \eqref{Moving_Decomp_geq_g_touch_g}, \eqref{Moving_Decomp_geq_g}, 
\eqref{Cor3_eq_prob_zero} and \eqref{Change_curve_gmin_to_g}, we obtain 
\eqref{eq_curve_min_t1_bdry} and \eqref{eq_curve_min_t1_jointdist}.

\subsection{Proof of Theorem~\ref{Thm_abs_conti}}

Let $F$ be an $\mathbb{R}$-valued bounded continuous function on $C([0, t], \mathbb{R})$. 
By the Markov property of $B_{[0,1]}^{0 \to b}$ and Lemma~\ref{Lem_abs_conti_B_motion_and_B_bridge}, we obtain
\begin{align}
&E[F(\pi_{[0,t]} ( B_{[0,1]}^{0 \to b} )) 1_{K_{[0, 1]}(g^--\eta^-(\varepsilon), g^++\eta^+(\varepsilon))}(B_{[0,1]}^{0 \to b} )]
\nonumber \\
&=\int_{C([0,1],\mathbb{R})}F(\pi_{[0,t]} (w)) 
1_{K_{[0, t]}(g^--\eta^-(\varepsilon), g^++\eta^+(\varepsilon))}(\pi_{[0,t]} (w))
\nonumber \\
&\qquad \qquad \times P\big( B_{[t,1]}^{w(t) \to b}\in K_{[t, 1]}(g^--\eta^-(\varepsilon), g^++\eta^+(\varepsilon))\big)
P(B_{[0,1]}^{0 \to b} \in dw)\nonumber \\
&=\int_{C([0,t],\mathbb{R})}F(w) 
1_{K_{[0, t]}(g^--\eta^-(\varepsilon), g^++\eta^+(\varepsilon))}(w)\nonumber \\
&\qquad \qquad \times P\big( B_{[t,1]}^{w(t) \to b}\in K_{[t, 1]}(g^--\eta^-(\varepsilon), g^++\eta^+(\varepsilon))\big)
P(\pi_{[0,t]}(B^{0 \to b})\in dw) \nonumber \\ 
&=\int_{C([0,t],\mathbb{R})}F(w) 
P\big( B_{[t,1]}^{w(t) \to b}\in K_{[t, 1]}(g^--\eta^-(\varepsilon), g^++\eta^+(\varepsilon))\big)
\label{bunshi_Markov_B0tob} \\
&\qquad \qquad \times \frac{n_{1-t}(w(t)-b)}{n_1(b)}
1_{K_{[0,t]}(g^--\eta^-(\varepsilon), g^++\eta^+(\varepsilon))}(w) P(\pi_{[0,t]}(W)\in dw).
\nonumber
\end{align}
Then, by \eqref{bunshi_Markov_B0tob}, it holds that
\begin{align}
&E[F(\pi_{[0,t]} ( B_{[0,1]}^{0 \to b} |_{K_{[0, 1]}(g^--\eta^-(\varepsilon), g^++\eta^+(\varepsilon))}))]\nonumber \\
&\quad =\frac{E[F(\pi_{[0,t]} ( B_{[0,1]}^{0 \to b} ))\ ;\ B_{[0,1]}^{0 \to b} \in K_{[0, 1]}(g^--\eta^-(\varepsilon), g^++\eta^+(\varepsilon))]}
{P(B_{[0,1]}^{0 \to b} \in K_{[0, 1]}(g^--\eta^-(\varepsilon), g^++\eta^+(\varepsilon)))} \nonumber \\
&\quad =\frac{\pi}{2}\cdot
\frac{ 2\eta^-(\varepsilon) \eta^+(\varepsilon)}
{\pi n_1(b) P\big(B_{[0,1]}^{0 \to b} \in K_{[0, 1]}(g^--\eta^-(\varepsilon), g^++\eta^+(\varepsilon)) \big)}\nonumber \\
&\qquad \times 
\frac{P\big( W_{[0,t]} \in K_{[0,t]}(g^--\eta^-(\varepsilon), g^++\eta^+(\varepsilon)) \big)}{\eta^-(\varepsilon)} \nonumber \\
&\qquad \times \int_{C([0,t],\mathbb{R})}F(w) 
\frac{P\big( B_{[t,1]}^{w(t) \to b}\in K_{[t, 1]}(g^--\eta^-(\varepsilon), g^++\eta^+(\varepsilon))\big)}{\eta^+(\varepsilon)}
\label{bunshi_bunbo_Markov_B0tob} \\
&\qquad \qquad \qquad \times 
n_{1-t}(w(t)-b) P\big(W_{[0,t]}\vert_{K_{[0,t]}(g^--\eta^-(\varepsilon), g^++\eta^+(\varepsilon))} \in dw \big). 
\nonumber
\end{align}
On the other hand, using Lemma~\ref{Lem_Gtrans_t1_dy} and 
Proposition~\ref{Prop_Boundary_purturb_WeakConv_for_Meander_BES3bridge}, 
we obtain
\begin{align}
&\frac{P\big( B_{[t,1]}^{a \to b}\in K_{[t, 1]}(g^--\eta^-(\varepsilon), g^++\eta^+(\varepsilon))\big)}{\eta^+(\varepsilon)} n_{1-t}(a-b)
\nonumber \\
&=\frac{P\big( a+W_{[t,1]}(1) \in d\widetilde{b},\ a+W_{[t,1]} \in K_{[t, 1]}(g^--\eta^-(\varepsilon), g^++\eta^+(\varepsilon)) \big)}{P(W_{[t,1]}\in K^+_{[t, 1]}(-\eta^+(\varepsilon)))d\widetilde{b}}
\ \Big\vert_{\widetilde{b}=b} \nonumber \\
&\quad \times
\frac{P(W_{[t,1]}\in K^+_{[t, 1]}(-\eta^+(\varepsilon)))}{\eta^+(\varepsilon)} \nonumber \\
&\to q^{(g^-,g^+), (\downarrow)}_{[t,1]}(a)\cdot \sqrt{\frac{2}{\pi}} \frac{1}{\sqrt{1-t}},
\qquad \varepsilon \downarrow 0, 
\label{Lem_for_radon_nikodym_conv1}
\end{align}
for $g^-(t)<a<g^+(t)$. 
In addition, by Lemma~\ref{Ap_Lem_Girsanov_all_type} \eqref{Girsanov_free_formula} and 
Proposition~\ref{Prop_Boundary_purturb_WeakConv_for_Meander_BES3bridge}, we have
\begin{align}
&\frac{P(W_{[0, t]} \in K_{[0, t]}(g^--\eta^-(\varepsilon), g^++\eta^+(\varepsilon)))}{\eta^-(\varepsilon)}
\nonumber \\
&\quad = 
\frac{E[\widetilde{Z}_{[0, t]}^{g^-} (W_{[0, t]})^{-1}
\ ;\ W_{[0, t]} \in K_{[0, t]}(-\eta^-(\varepsilon), g^+-g^-+\eta^+(\varepsilon))]}
{P(W_{[0, t]} \in K_{[0, t]}^+(-\eta^-(\varepsilon) ))}  
\cdot \frac{P(W_{[0, t]} \in K_{[0, t]}^+(-\eta^-(\varepsilon) ) )}{\eta^-(\varepsilon)} \nonumber \\
&\quad = E\left[\widetilde{Z}_{[0, t]}^{g^-} 
(W_{[0, t]}|_{K_{[0, t]}(-\eta^-(\varepsilon), g^+-g^-+\eta^+(\varepsilon) )})^{-1} \right]
P\left(W_{[0, t]}|_{K_{[0, t]}^+(-\eta^-(\varepsilon) )} 
\in K_{[0, t]}^-(g^+-g^-+\eta^+(\varepsilon)) \right) \nonumber \\
&\qquad \times \frac{P(W_{[0, t]} \in K_{[0, t]}^+(-\eta^-(\varepsilon) ) )}{\eta^-(\varepsilon)} 
\nonumber \\
&\quad \to 
E\left[\widetilde{Z}_{[0, t]}^{g^-} 
(W_{[0, t]}^+|_{K_{[0, t]}^-(g^+-g^-)})^{-1} \right]
P\left(W_{[0, t]}^+ \in K_{[0, t]}^-(g^+-g^-) \right)\sqrt{\frac{2}{\pi}} \cdot \frac{1}{\sqrt{t}}, 
\qquad \varepsilon \downarrow 0. 
\label{Lem_for_radon_nikodym_conv2}
\end{align}
Therefore, it follows from 
Theorem~\ref{Thm_Def_and_Decomp_curved_Moving}, 
\eqref{bunshi_bunbo_Markov_B0tob}, 
\eqref{Def_Cgmgp}, 
\eqref{Lem_for_radon_nikodym_conv1}, 
\eqref{Lem_for_radon_nikodym_conv2} and 
Lemma~\ref{Lem_WeakConv_Meander_btwn_2crvs} that
\begin{align}
&E\left[ F(\pi_{[0,t]}(H^{g^-\to g^+})) \right]
=\lim_{\varepsilon \downarrow 0} 
E[F(\pi_{[0,t]}(B_{[0,1]}^{0 \to b} |_{K_{[0, 1]}(g^--\eta^-(\varepsilon), g^++\eta^+(\varepsilon))}))]
\nonumber \\
&\quad =E\left[ F\left(W_{[0,t]}^+ +g^-\right) 
\dfrac{ q^{(g^-,g^+), (\downarrow)}_{[t,1]}\left(W_{[0,t]}^+(t) +g^-(t)\right) }
{ C_{g^-,g^+}\sqrt{t(1-t)} \cdot \widetilde{Z}_{[0,t]}^{g^-} \big(W_{[0,t]}^+\big) }
1_{K_{[0,t]}^-(g^+-g^-)}\left( W_{[0,t]}^+\right) \right] .
\label{CVF_Meander} 
\end{align}
Further, combining \eqref{CVF_Meander} and 
a change of measure formula between Brownian meander and BES$(3)$-process (\cite{bib_Imhof}), we obtain
\begin{align}
&E\left[ F(\pi_{[0,t]}(H^{g^-\to g^+})) \right] \nonumber \\
&\quad = \sqrt{ \frac{\pi }{2} }
E\left[ F\left(R_{[0,t]} +g^-\right) 
\dfrac{ q^{(g^-,g^+), (\downarrow)}_{[t,1]}\left(R_{[0,t]}(t) +g^-(t)\right) }
{ C_{g^-,g^+}\sqrt{1-t} \cdot R_{[0,t]}(t)\cdot \widetilde{Z}_{[0,t]}^{g^-} \big(R_{[0,t]}\big) }
1_{K_{[0,t]}^-(g^+-g^-)}\left( R_{[0,t]} \right) \right] \nonumber \\
&\quad = \sqrt{ \frac{\pi }{2} }
\int_{C([0,t],\mathbb{R})}F(w)
\dfrac{ q^{(g^-,g^+), (\downarrow)}_{[t,1]}\left(w(t)\right) }
{ C_{g^-,g^+}\sqrt{1-t} \cdot (w(t)-g^-(t))\cdot Z_{[0,t]}^{g^-}(w)} \label{proof_abs_conti_Functional_rep} \\
&\qquad \qquad \qquad \qquad \times
1_{K_{[0,t]}^-(g^+)}\left( w\right) 
P\left( R_{[0,t]} +g^- \in dw \right). \nonumber
\end{align}

Using \eqref{proof_abs_conti_Functional_rep}, 
Lebesgue's dominated convergence theorem, and Dynkin's lemma, 
we can prove Theorem~\ref{Thm_abs_conti}. 
\qed

\section{Proofs of Theorem~\ref{Thm_Def_and_Decomp_curved_Meander_two_curve} }
\label{section_proof_Meander_between_2curves}

In this subsection, we assume that all $X_{[s, t]}^{x,y,(g^-, g^+)}$ and $X_{[s, t]}^{z,(g^-, g^+)}$ are independent. 
For each $\mathbb{R}$-valued bounded continuous function $G$ on $C([0,1], \mathbb{R})$ and $\varepsilon>0$, we define
\begin{align*}
I(\varepsilon, G):=E[G(W_{[0,1]})\ ;\ W_{[0,1]} \in K_{[0, 1]}(g^--\varepsilon, g^+)].
\end{align*}
Then, we have
\begin{equation}\label{ishitani_meander_2curves}
E[F(W_{[0,1]} |_{K_{[0, 1]}(g^--\varepsilon, g^+)})] 
= \frac{E[F(W_{[0,1]})\ ;\ W_{[0,1]} \in K_{[0, 1]}(g^--\varepsilon, g^+)]}
{P(W_{[0,1]} \in K_{[0, 1]}(g^--\varepsilon, g^+))}
=\frac{I(\varepsilon, F)}{I(\varepsilon, 1)}.
\end{equation}
Further, by Lemma~\ref{Lem_MarkovDecomp_BM}, we obtain
\begin{align}
I(\varepsilon, F) 
&= \int_{g^-(t)-\varepsilon}^{g^+(t)}  
E\left[F\left(X_{[0, t]}^{0, y,(g^--\varepsilon, g^+)} \oplus_{t} 
X_{[t, 1]}^{y, (g^--\varepsilon, g^+)}\right)\right] \label{Meander_MarkovDecomp_single} \\
&\qquad \qquad \times
P(W_{[0, t]} \in K_{[0, t]}(g^--\varepsilon, g^+), W_{[0, t]}(t) \in dy) \nonumber \\
&\qquad \qquad \times 
P(y + W_{[t, 1]} \in K_{[t, 1]}(g^--\varepsilon, g^+) ).
\nonumber 
\end{align}
It follows from \eqref{Meander_MarkovDecomp_single}, 
Lemma~\ref{Lem_Gtrans_0t_dy}, 
Proposition~\ref{Prop_Boundary_purturb_WeakConv_for_Meander_BES3bridge}, 
Lemma~\ref{Lem_L1_convergence},  
Lemma~\ref{Lem_for_H_gm_to_gp} 
and Lemma~\ref{Ap_Lem_weakconv_product} that
\begin{align}
I(F):=&\lim_{\varepsilon \downarrow 0} 
\frac{I(\varepsilon, F)}{\varepsilon}
=\lim_{\varepsilon \downarrow 0} 
\frac{I(\varepsilon, F)}{P(W_{[0, t]} \in K_{[0, t]}^+(-\varepsilon ))} 
\times \lim_{\varepsilon \downarrow 0} 
\frac{P(W_{[0, t]} \in K_{[0, t]}^+(-\varepsilon ))}{\varepsilon}
\nonumber \\
=& \sqrt{\frac{2}{\pi}} 
\int_{g^-(t)}^{g^+(t)}
E\left[F\left(X_{[0, t]}^{0, y,(g^-, g^+)} \oplus_{t} X_{[t, 1]}^{y, (g^-, g^+)}\right)\right] 
\frac{1}{\sqrt{t}}q^{(g^-,g^+), (\uparrow)}_{[0,t]}(y) 
p^{(g^-,g^+)}_{[t,1]}(y) dy.
\label{IF_t_Meander_decomp}
\end{align}
According to \eqref{IF_t_Meander_decomp}, we get
\begin{align}\label{I1_t_Meander_decomp}
\sqrt{\frac{2}{\pi}} 
\int_{g^-(t)}^{g^+(t)}
\frac{1}{\sqrt{t}}q^{(g^-,g^+), (\uparrow)}_{[0,t]}(y)p^{(g^-,g^+)}_{[t,1]}(y)dy 
=I(1)=\lim_{\varepsilon \downarrow 0} 
\frac{P(W \in K_{[0, 1]}(g^--\varepsilon, g^+))}{\varepsilon}
=\sqrt{\frac{2}{\pi}} \widetilde{C}_{g^-,g^+} . 
\end{align}
Combining \eqref{ishitani_meander_2curves}, \eqref{IF_t_Meander_decomp} and \eqref{I1_t_Meander_decomp}, we obtain
\begin{align*}
\widetilde{C}_{g^-,g^+}=\int_{g^-(t)}^{g^+(t)}
\frac{1}{\sqrt{t}}q^{(g^-,g^+), (\uparrow)}_{[0,t]}(y)p^{(g^-,g^+)}_{[t,1]}(y) dy \in (0, \infty)
\end{align*}
and
\begin{align*}
\lim_{\varepsilon \downarrow 0} 
E[F(W_{[0, 1]} |_{K_{[0, 1]}(g^--\varepsilon, g^+)})]
=\frac{I(F)}{I(1)}
=\int_{g^-(t)}^{g^+(t)} 
E\left[ F(X_{[0,t]}^{0,y,(g^-, g^+)} \oplus_{t} X_{[t,1]}^{y,(g^-, g^+)}) \right]  k(t,y) dy.
\end{align*}
Therefore, we can define the probability measure $\widetilde{P}_+$
on $(C([0,1], \mathbb{R}), \mathcal{B}(C([0,1], \mathbb{R})))$ as
\begin{align*}
\widetilde{P}_+(A):= \int_{g^-(t)}^{g^+(t)} 
P\left( X_{[0,t]}^{0,y,(g^-, g^+)} \oplus_{t} X_{[t,1]}^{y,(g^-, g^+)} \in A \right) k(t,y) dy
\quad (A\in \mathcal{B}(C([0,1], \mathbb{R}))),
\end{align*}
and there exists an $\mathbb{R}$-valued 
continuous stochastic process $W^{+,(g^-,g^+)}=\{W^{+,(g^-,g^+)}(t)\}_{t\in [0,1]}$
that satisfies \eqref{curvedMeander_twocurve_def_eq1} 
and \eqref{twocurve_meander_decomp_1}.
Thus, a limit argument on $F$ yields 
\begin{align*}
P(W^{+,(g^-,g^+)}(t) \in dy) = k(t, y)dy\quad (y \in (g^-(t), g^+(t))).
\end{align*}

On the other hand, by Lemma~\ref{Lem_MarkovDecomp_BM}, we obtain
\begin{align}
I(\varepsilon, F) 
&= \int_{g^-(t_2)-\varepsilon}^{g^+(t_2)} dy_2
\int_{g^-(t_1)-\varepsilon}^{g^+(t_1)} dy_1
\label{Meander_MarkovDecomp_double} \\
&\qquad \times E\left[F\left(X_{[0, t_1]}^{0, y_1,(g^--\varepsilon, g^+)} \oplus_{t_1} 
X_{[t_1,t_2]}^{y_1, y_2,(g^--\varepsilon, g^+)} \oplus_{t_2} 
X_{[t_2, 1]}^{y_2, (g^--\varepsilon, g^+)}\right)\right] 
\nonumber \\
&\qquad \times 
P(y_2 + W_{[t_2, 1]} \in K_{[t_2, 1]}(g^--\varepsilon, g^+)) 
\nonumber \\
&\qquad \times 
P(y_1 + W_{[t_1,t_2]} \in K_{[t_1,t_2]}(g^--\varepsilon, g^+), 
y_1 + W_{[t_1,t_2]}(t_2) \in dy_2)/dy_2 
\nonumber \\
&\qquad \times 
P(W_{[0, t_1]} \in K_{[0, t_1]}(g^--\varepsilon, g^+), W_{[0, t_1]}(t_1) \in dy_1)/dy_1.
\nonumber 
\end{align}
By \eqref{Meander_MarkovDecomp_double}, Lemma~\ref{Lem_Gtrans_0t_dy}, 
Proposition~\ref{Prop_Boundary_purturb_WeakConv_for_Meander_BES3bridge}, 
Lemma~\ref{Lem_L1_convergence}, 
Lemma~\ref{Lem_for_H_gm_to_gp} 
and Lemma~\ref{Ap_Lem_weakconv_product}, 
$I(F)$ satisfies
\begin{align}
I(F)=&\lim_{\varepsilon \downarrow 0} 
\frac{I(\varepsilon, F)}{P(W_{[0, t_1]} \in K_{[0, t_1]}^+(-\varepsilon))} 
\times \lim_{\varepsilon \downarrow 0} 
\frac{P(W_{[0, t_1]} \in K_{[0, t_1]}^+(-\varepsilon))}{\varepsilon} \nonumber \\
=&\sqrt{\frac{2}{\pi}}
\int_{g^-(t_1)}^{g^+(t_1)} \int_{g^-(t_2)}^{g^+(t_2)} 
E\left[F\left(X_{[0,t_1]}^{0,y_1,(g^-, g^+)} \oplus_{t_1} 
X_{[t_1,t_2]}^{y_1, y_2,(g^-, g^+)} \oplus_{t_2} 
X_{[t_2, 1]}^{y_2, (g^-, g^+)}\right)\right] \nonumber \\
&\qquad \qquad \qquad \qquad \times 
\frac{1}{\sqrt{t_1}}
q^{(g^-,g^+), (\uparrow)}_{[0,t_1]}(y_1)p^{(g^-,g^+)}_{[t_1,t_2]}(y_1, y_2) p^{(g^-,g^+)}_{[t_2,1]}(y_2) dy_1 dy_2 \nonumber \\
=&\sqrt{\frac{2}{\pi}}\widetilde{C}_{g^-,g^+}
\int_{g^-(t_1)}^{g^+(t_1)} \int_{g^-(t_2)}^{g^+(t_2)} 
E\left[F\left(X_{[0,t_1]}^{0,y_1,(g^-, g^+)} \oplus_{t_1} 
X_{[t_1,t_2]}^{y_1, y_2,(g^-, g^+)} \oplus_{t_2} 
X_{[t_2, 1]}^{y_2, (g^-, g^+)}\right)\right] \label{IF_doubleIntg_kt1t2}\\
&\qquad \qquad \qquad \qquad \qquad \times 
k(t_1, y_1) k(t_1, y_1, t_2, y_2) dy_1 dy_2 . \nonumber
\end{align}
It follows from \eqref{ishitani_meander_2curves}, \eqref{I1_t_Meander_decomp} and \eqref{IF_doubleIntg_kt1t2} that
\begin{align*}
&E[F(W^{+,(g^-,g^+)})] 
=\lim_{\varepsilon \downarrow 0} 
E[F(W_{[0, 1]} |_{K_{[0, 1]}(g^--\varepsilon, g^+)})]\\
&\quad =\int_{g^-(t_1)}^{g^+(t_1)} \int_{g^-(t_2)}^{g^+(t_2)} 
E\left[F\left(X_{[0,t_1]}^{0,y_1,(g^-, g^+)} \oplus_{t_1} 
X_{[t_1,t_2]}^{y_1, y_2,(g^-, g^+)}\oplus_{t_2} 
X_{[t_2, 1]}^{y_2, (g^-, g^+)}\right)\right] \nonumber \\
&\qquad \qquad \qquad \qquad \times k(t_1, y_1)k(t_1, y_1, t_2, y_2) dy_1 dy_2 .
\end{align*}
Hence, \eqref{twocurve_meander_decomp_2} holds. 
Similarly, using a limit argument on $F$, we can deduce for 
$y_1 \in (g^-(t_1), g^+(t_1))$ and $y_2 \in (g^-(t_2), g^+(t_2))$ that
\begin{align}
P(W^{+,(g^-,g^+)}(t_1) \in dy_1, W^{+,(g^-,g^+)}(t_2) \in dy_2) 
&= k(t_1, y_1)k(t_1, y_1, t_2, y_2) dy_1 dy_2, \nonumber \\
P(W^{+,(g^-,g^+)}(t_2) \in dy_2 \ \vert \ W^{+,(g^-,g^+)}(t_1)=y_1) 
&=k(t_1, y_1, t_2, y_2) dy_2.
\label{Proof_Wplus_transdensity}
\end{align}

Because we have
\begin{align*}
p^{(g^-,g^+)}_{[t_1,1]}(y_1) &= 
\int_{g^-(t_2)}^{g^+(t_2)} 
p^{(g^-,g^+)}_{[t_1,t_2]} (y_1, y_2) p^{(g^-,g^+)}_{[t_2,1]}(y_2) dy_2 
\qquad (y_1 \in (g^-(t_1), g^+(t_1))),
\end{align*}
we can deduce that
\begin{align}\label{kt1y1t2y2_density}
\int_{g^-(t_2)}^{g^+(t_2)} k(t_1, y_1, t_2, y_2) dy_2 = 1
\qquad (y_1 \in (g^-(t_1), g^+(t_1))).
\end{align}
The Chapman--Kolmogorov identity for $k$ is obtained by \eqref{Chapman_id_p_y1_y3}.
Therefore, \eqref{Proof_Wplus_transdensity} and \eqref{kt1y1t2y2_density} imply that 
$W^{+,(g^-,g^+)}=\{W^{+,(g^-,g^+)}(t)\}_{t\in [0,1]}$ is a Markov process.

\section{Proofs of the main results in Subsection~\ref{subsection_BES3bridge_between_2curves}}
\label{section_proof_3dBesselbridge_between_2curves}

In this section, we prove 
Theorems~\ref{Thm_Def_and_Decomp_curved_BESbridge_two_curve} and \ref{Thm_WeakConv_BESbridge_to_Moving}.

\subsection{Proof of Theorem~\ref{Thm_Def_and_Decomp_curved_BESbridge_two_curve}}

In this subsection, we assume that all $X_{[s, t]}^{x, y,(g^-, g^+)}$ are independent. 
For each $\mathbb{R}$-valued bounded continuous function $G$ on $C([0,1], \mathbb{R})$ and $\varepsilon>0$, we define
\begin{align*}
I(\varepsilon, G):=E[G(W_{[0,1]})\ ;\ W_{[0,1]}(1)\in d\widetilde{c}, W_{[0,1]} \in K_{[0, 1]}(g^--\varepsilon, g^+)]/d\widetilde{c}
\ \big\vert_{\widetilde{c}=c}.
\end{align*}
Then, we have
\begin{equation}\label{ishitani_BESbridge_2curves}
E[F(B_{[0,1]}^{0\to c} |_{K_{[0, 1]}(g^--\varepsilon, g^+)})] 
= \frac{I(\varepsilon, F)}{I(\varepsilon, 1)}.
\end{equation}
Further, by Lemma~\ref{Lem_MarkovDecomp_B_Bridge}, we obtain
\begin{align}
I(\varepsilon, F) 
&= \int_{g^-(t)-\varepsilon}^{g^+(t)}  
E\left[F\left(X_{[0, t]}^{0, y,(g^--\varepsilon, g^+)} \oplus_{t} 
X_{[t, 1]}^{y, c,(g^--\varepsilon, g^+)}\right)\right] 
\label{BESbridge_MarkovDecomp_single}\\
&\qquad \qquad \times
P(W_{[0, t]} \in K_{[0, t]}(g^--\varepsilon, g^+), W_{[0, t]}(t) \in dy) 
\nonumber \\
&\qquad \qquad \times 
P(y + W_{[t, 1]} \in K_{[t, 1]}(g^--\varepsilon, g^+), y + W_{[t, 1]}(1) \in d\widetilde{c} )/d\widetilde{c}
\ \big\vert_{\widetilde{c}=c}.
\nonumber 
\end{align}
It follows from \eqref{BESbridge_MarkovDecomp_single}, 
Lemma~\ref{Lem_Gtrans_0t_dy}, 
Proposition~\ref{Prop_Boundary_purturb_WeakConv_for_Meander_BES3bridge}, 
Lemma~\ref{Lem_L1_convergence}, 
Lemma~\ref{Lem_for_H_gm_to_gp} 
and Lemma~\ref{Ap_Lem_weakconv_product} that
\begin{align}
I(F):=&\lim_{\varepsilon \downarrow 0} 
\frac{I(\varepsilon, F)}{\varepsilon}
=\lim_{\varepsilon \downarrow 0} 
\frac{I(\varepsilon, F)}{P(W_{[0, t]} \in K_{[0, t]}^+(-\varepsilon ))} 
\times \lim_{\varepsilon \downarrow 0} 
\frac{P(W_{[0, t]} \in K_{[0, t]}^+(-\varepsilon ))}{\varepsilon}
\nonumber \\
=& \sqrt{\frac{2}{\pi}} 
\int_{g^-(t)}^{g^+(t)}
E\left[F\left(X_{[0, t]}^{0, y,(g^-, g^+)} \oplus_{t} X_{[t, 1]}^{y,c,(g^-, g^+)}\right)\right] 
\frac{1}{\sqrt{t}}q^{(g^-,g^+), (\uparrow)}_{[0,t]}(y) 
p^{(g^-,g^+)}_{[t,1]}(y, c) dy .
\label{IF_t_BESbridge_decomp}
\end{align}
According to \eqref{IF_t_BESbridge_decomp}, we get
\begin{align}
&\sqrt{\frac{2}{\pi}} 
\int_{g^-(t)}^{g^+(t)}
\frac{1}{\sqrt{t}}q^{(g^-,g^+), (\uparrow)}_{[0,t]}(y)p^{(g^-,g^+)}_{[t,1]}(y,c)dy 
\nonumber \\
&\quad =I(1)=\lim_{\varepsilon \downarrow 0} 
\frac{P(W(1) \in d\widetilde{c}, W \in K_{[0, 1]}(g^--\varepsilon, g^+))}{\varepsilon d\widetilde{c}}
\ \Big\vert_{\widetilde{c}=c}
\nonumber \\
&\quad =n_1(c)\lim_{\varepsilon \downarrow 0} 
\frac{P(B_{[0,1]}^{0\to c} \in K_{[0, 1]}(g^--\varepsilon, g^+))}{\varepsilon}
=\sqrt{\frac{2}{\pi}} \widehat{C}_{g^-,g^+} . 
\label{I1_t_BESbridge_decomp}
\end{align}
Combining \eqref{ishitani_BESbridge_2curves}, \eqref{IF_t_BESbridge_decomp} and \eqref{I1_t_BESbridge_decomp}, we obtain
\begin{align*}
\widehat{C}_{g^-,g^+}=\int_{g^-(t)}^{g^+(t)}
\frac{1}{\sqrt{t}}q^{(g^-,g^+), (\uparrow)}_{[0,t]}(y)p^{(g^-,g^+)}_{[t,1]}(y, c) dy \in (0, \infty)
\end{align*}
and
\begin{align*}
&\lim_{\varepsilon \downarrow 0} 
E[F(B_{[0, 1]}^{0\to c} |_{K_{[0, 1]}(g^--\varepsilon, g^+)})]
= \frac{I(F)}{I(1)} 
= \int_{g^-(t)}^{g^+(t)} 
E\left[ F\left(X_{[0,t]}^{0,y,(g^-, g^+)} \oplus_{t} X_{[t,1]}^{y,c,(g^-, g^+)}\right) \right] l(t,y) dy .
\end{align*}
Therefore, we can define the probability measure $\widehat{P}^{0\to c}$
on $(C([0,1], \mathbb{R}), \mathcal{B}(C([0,1], \mathbb{R})))$ as
\begin{align*}
\widehat{P}^{0\to c}(A):= \int_{g^-(t)}^{g^+(t)} 
P\left( X_{[0,t]}^{0,y,(g^-, g^+)} \oplus_{t} X_{[t,1]}^{y,c,(g^-, g^+)} \in A \right) l(t,y) dy
\quad (A\in \mathcal{B}(C([0,1], \mathbb{R}))),
\end{align*}
and there exists an $\mathbb{R}$-valued 
continuous stochastic process $r^{0\to c,(g^-,g^+)}=\{r^{0\to c,(g^-,g^+)}(t)\}_{t\in [0,1]}$
that satisfies \eqref{curvedBESbridge_twocurve_def_eq1} and \eqref{twocurve_BESbridge_decomp_1}. 
Thus, a limit argument on $F$ yields 
\begin{align*}
P(r^{0\to c,(g^-,g^+)}(t) \in dy) = l(t, y)dy\quad (y \in (g^-(t), g^+(t))).
\end{align*}

On the other hand, by Lemma~\ref{Lem_MarkovDecomp_B_Bridge}, we obtain
\begin{align}
I(\varepsilon, F) 
&= \int_{g^-(t_2)-\varepsilon}^{g^+(t_2)} dy_2
\int_{g^-(t_1)-\varepsilon}^{g^+(t_1)} dy_1 \label{BESbridge_MarkovDecomp_double} \\
&\qquad \times E\left[F\left(X_{[0, t_1]}^{0, y_1,(g^--\varepsilon, g^+)} \oplus_{t_1} 
X_{[t_1,t_2]}^{y_1, y_2,(g^--\varepsilon, g^+)} \oplus_{t_2} 
X_{[t_2, 1]}^{y_2, c,(g^--\varepsilon, g^+)}\right)\right] 
\nonumber \\
&\qquad \times 
P(y_2 + W_{[t_2, 1]} \in K_{[t_2, 1]}(g^--\varepsilon, g^+), y_2 + W_{[t_2, 1]}(1) \in d\widetilde{c})/d\widetilde{c} \ \big\vert_{\widetilde{c}=c} 
\nonumber \\
&\qquad \times 
P(y_1 + W_{[t_1,t_2]} \in K_{[t_1,t_2]}(g^--\varepsilon, g^+), 
y_1 + W_{[t_1,t_2]}(t_2) \in dy_2)/dy_2 
\nonumber \\
&\qquad \times 
P(W_{[0, t_1]} \in K_{[0, t_1]}(g^--\varepsilon, g^+), W_{[0, t_1]}(t_1) \in dy_1)/dy_1. 
\nonumber
\end{align}
By \eqref{BESbridge_MarkovDecomp_double}, Lemma~\ref{Lem_Gtrans_0t_dy}, 
Proposition~\ref{Prop_Boundary_purturb_WeakConv_for_Meander_BES3bridge}, 
Lemma~\ref{Lem_L1_convergence}, 
Lemma~\ref{Lem_for_H_gm_to_gp} 
and Lemma~\ref{Ap_Lem_weakconv_product}, 
$I(F)$ satisfies
\begin{align}
I(F)=&\lim_{\varepsilon \downarrow 0} 
\frac{I(\varepsilon, F)}{P(W_{[0, t_1]} \in K_{[0, t_1]}^+(-\varepsilon))} 
\times \lim_{\varepsilon \downarrow 0} 
\frac{P(W_{[0, t_1]} \in K_{[0, t_1]}^+(-\varepsilon))}{\varepsilon} 
\nonumber \\
=&\sqrt{\frac{2}{\pi}}
\int_{g^-(t_1)}^{g^+(t_1)} \int_{g^-(t_2)}^{g^+(t_2)} 
E\left[F\left(X_{[0,t_1]}^{0,y_1,(g^-, g^+)} \oplus_{t_1} 
X_{[t_1,t_2]}^{y_1, y_2,(g^-, g^+)} \oplus_{t_2} 
X_{[t_2, 1]}^{y_2, c,(g^-, g^+)}\right)\right] 
\nonumber \\
&\qquad \qquad \qquad \qquad \times 
\frac{1}{\sqrt{t_1}}q^{(g^-,g^+), (\uparrow)}_{[0,t_1]}(y_1)p^{(g^-,g^+)}_{[t_1,t_2]}(y_1, y_2) p^{(g^-,g^+)}_{[t_2,1]}(y_2, c) dy_1 dy_2
\nonumber \\
=&\sqrt{\frac{2}{\pi}}\widehat{C}_{g^-,g^+}
\int_{g^-(t_1)}^{g^+(t_1)} \int_{g^-(t_2)}^{g^+(t_2)} 
E\left[F\left(X_{[0,t_1]}^{0,y_1,(g^-, g^+)} \oplus_{t_1} 
X_{[t_1,t_2]}^{y_1, y_2,(g^-, g^+)} \oplus_{t_2} 
X_{[t_2, 1]}^{y_2, c,(g^-, g^+)}\right)\right] \label{IF_doubleIntg_lt1t2}\\
&\qquad \qquad \qquad \qquad \qquad \times 
l(t_1, y_1) l(t_1, y_1, t_2, y_2) dy_1 dy_2 . \nonumber
\end{align}
It follows from \eqref{ishitani_BESbridge_2curves}, \eqref{I1_t_BESbridge_decomp} and \eqref{IF_doubleIntg_lt1t2} that
\begin{align*}
&E[F(r^{0\to c,(g^-,g^+)})] 
=\lim_{\varepsilon \downarrow 0} 
E[F(B_{[0, 1]}^{0\to c} |_{K_{[0, 1]}(g^--\varepsilon, g^+)})]\\
&\quad =\int_{g^-(t_1)}^{g^+(t_1)} \int_{g^-(t_2)}^{g^+(t_2)} 
E\left[F\left(X_{[0,t_1]}^{0,y_1,(g^-, g^+)} \oplus_{t_1} 
X_{[t_1,t_2]}^{y_1, y_2,(g^-, g^+)}\oplus_{t_2} 
X_{[t_2, 1]}^{y_2, c,(g^-, g^+)}\right)\right] \\
&\qquad \qquad \qquad \quad \times 
l(t_1, y_1)l(t_1, y_1, t_2, y_2) dy_1 dy_2 .
\end{align*}
Hence, \eqref{twocurve_BESbridge_decomp_2} holds. 
Similarly, using a limit argument on $F$, we can deduce for 
$y_1 \in (g^-(t_1), g^+(t_1))$ and $y_2 \in (g^-(t_2), g^+(t_2))$ that
\begin{align}
P(r^{0\to c,(g^-,g^+)}(t_1) \in dy_1, r^{0\to c,(g^-,g^+)}(t_2) \in dy_2) &= l(t_1, y_1)l(t_1, y_1, t_2, y_2) dy_1 dy_2, 
\nonumber \\
P(r^{0\to c,(g^-,g^+)}(t_2) \in dy_2 \ \vert \ r^{0\to c,(g^-,g^+)}(t_1)=y_1) &=l(t_1, y_1, t_2, y_2) dy_2.
\label{r0toc_transdensity}
\end{align}

Because we have
\begin{align*}
p^{(g^-,g^+)}_{[t_1,1]}(y_1, c) &= 
\int_{g^-(t_2)}^{g^+(t_2)} 
p^{(g^-,g^+)}_{[t_1,t_2]} (y_1, y_2) p^{(g^-,g^+)}_{[t_2,1]}(y_2, c) dy_2 
\qquad (y_1 \in (g^-(t_1), g^+(t_1))),
\end{align*}
we can deduce that
\begin{align}\label{lt1y1t2y2_density}
\int_{g^-(t_2)}^{g^+(t_2)} l(t_1, y_1, t_2, y_2) dy_2 = 1
\qquad (y_1 \in (g^-(t_1), g^+(t_1))).
\end{align}
The Chapman--Kolmogorov identity for $l$ is obtained by \eqref{Chapman_id_p_y1_y3}. 
Therefore, \eqref{r0toc_transdensity} and \eqref{lt1y1t2y2_density} imply that 
$r^{0\to c,(g^-,g^+)}=\{r^{0\to c,(g^-,g^+)}(t)\}_{t\in [0,1]}$ is a Markov process.

\subsection{Proof of Theorem~\ref{Thm_WeakConv_BESbridge_to_Moving}}

In this subsection, we assume that all $X_{[s, t]}^{x,y,(g^-, g^+)}$ are independent. 
For each $F$, $\varepsilon>0$ and $\eta>0$, we have
\begin{equation}\label{BESbrige_to_Moving_ep_eta}
E[F(B_{[0,1]}^{0 \to b} |_{K_{[0, 1]}(g^--\varepsilon, g^++\eta)})] 
= \frac{I(\varepsilon, \eta, F)}{I(\varepsilon, \eta, 1)},
\end{equation}
where 
\begin{align*}
I(\varepsilon, \eta, F):=E[F(W_{[0,1]})\ ;\ W_{[0,1]}(1) \in d\widetilde{b}, W_{[0,1]} \in K_{[0, 1]}(g^--\varepsilon, g^++\eta)]/d\widetilde{b} \ \big\vert_{\widetilde{b} = b}.
\end{align*} 
Then, by Lemma~\ref{Lem_MarkovDecomp_B_Bridge}, we obtain
\begin{align}
&I(\varepsilon, \eta, F) 
= \int_{g^-(t)-\varepsilon}^{g^+(t)+\eta}  
E\left[F\left(X_{[0, t]}^{0, y,(g^--\varepsilon, g^++\eta)} \oplus_{t} 
X_{[t, 1]}^{y, b,(g^--\varepsilon, g^++\eta)}\right)\right] 
\label{BESbridge_MarkovDecomp_epeta_single}\\
&\qquad \qquad \quad \times 
P(W_{[0, t]} \in K_{[0, t]}(g^--\varepsilon, g^++\eta), W_{[0, t]}(t) \in dy) 
\nonumber \\
&\qquad \qquad \quad \times 
P(y + W_{[t, 1]} \in K_{[t, 1]}(g^--\varepsilon, g^++\eta), y + W_{[t, 1]}(1) \in d\widetilde{b})/d\widetilde{b} \ \big\vert_{\widetilde{b} = b}.
\nonumber 
\end{align}
By \eqref{BESbridge_MarkovDecomp_epeta_single}, Lemma~\ref{Lem_Gtrans_0t_dy}, Lemma~\ref{Lem_Gtrans_t1_dy}, 
Proposition~\ref{Prop_Boundary_purturb_WeakConv_for_Meander_BES3bridge}, 
Lemma~\ref{Lem_L1_convergence}, 
Lemma~\ref{Lem_for_H_gm_to_gp} 
and Lemma~\ref{Ap_Lem_weakconv_product}, it holds that
\begin{align}
I(F)
:=&\lim_{\eta \downarrow 0}\lim_{\varepsilon \downarrow 0} 
\frac{I(\varepsilon, \eta, F)}{\varepsilon \eta}
\nonumber \\
=& \frac{2}{\pi} \int_{g^-(t)}^{g^+(t)}
E\left[F\left(X_{[0, t]}^{0, y,(g^-, g^+)} \oplus_{t} X_{[t, 1]}^{y, b,(g^-, g^+)}\right)\right] 
\frac{1}{\sqrt{t}}q^{(g^-,g^+), (\uparrow)}_{[0,t]}(y)\frac{1}{\sqrt{1-t}}q^{(g^-,g^+), (\downarrow)}_{[t,1]}(y)dy .
\label{IF_t_decomp_again}
\end{align}
By \eqref{I1_t_decomp} and \eqref{IF_t_decomp_again}, it holds that
\begin{align}
\frac{2}{\pi} 
\int_{g^-(t)}^{g^+(t)}
\frac{1}{\sqrt{t}}q^{(g^-,g^+), (\uparrow)}_{[0,t]}(y)\frac{1}{\sqrt{1-t}}q^{(g^-,g^+), (\downarrow)}_{[t,1]}(y)dy 
=I(1)=\frac{2}{\pi} C_{g^-,g^+} . 
\label{I1_t_decomp_again}
\end{align}
Combining Theorem~\ref{Thm_Def_and_Decomp_curved_BESbridge_two_curve} 
and \eqref{BESbrige_to_Moving_ep_eta}, we obtain
\begin{align}
\lim_{\eta \downarrow 0} 
E[F(r^{0\to b, (g^-,g^++\eta)})]
=\lim_{\eta \downarrow 0} 
\lim_{\varepsilon \downarrow 0} 
E[F(B_{[0, 1]}^{0 \to b} |_{K_{[0, 1]}(g^--\varepsilon, g^++\eta)})]
=\lim_{\eta \downarrow 0} 
\lim_{\varepsilon \downarrow 0} \frac{I(\varepsilon, \eta, F)}{I(\varepsilon, \eta, 1)}
= \frac{I(F)}{I(1)}.
\label{limeq_IepetaF_over_Iepeta1}
\end{align}
Therefore, it follows from \eqref{IF_t_decomp_again}, \eqref{I1_t_decomp_again}, \eqref{limeq_IepetaF_over_Iepeta1} 
and Theorem~\ref{Thm_Def_and_Decomp_curved_Moving} that
\begin{align*}
\lim_{\eta \downarrow 0} 
E[F(r^{0\to b, (g^-,g^++\eta)})]
= \int_{g^-(t)}^{g^+(t)} E\left[ F\left(X_{[0,t]}^{0,y,(g^-, g^+)} \oplus_{t} X_{[t,1]}^{y,b,(g^-, g^+)}\right) \right] h(t, y)dy 
=E\left[ F(H^{g^-\to g^+}) \right].
\end{align*}

\section{Conclusion and future work}\label{Sec_conclusion_future_work}

We introduced a stochastic process called Brownian house-moving. 
We also studied the sample path properties of Brownian house-moving. 

In \cite{bib_Ishitani}, we proposed a chain rule for Wiener path integrals between two curves, 
which is used to compute first-order Greeks for barrier options.
In this chain rule, the following terms
\begin{align}
&\int_0^1 
\widetilde{\nu}^{(-)}_{g^-, g^+}(t)
E\Big[F\big(
X_{[0,t]}^{0,g^-(t),(g^-, g^+)} \oplus_{t}  
X_{[t, 1]}^{g^-(t),(g^-, g^+)}\big)\Big] dt ,
\label{FirstOrderBdryTerm_m}\\
&\int_0^1 
\widetilde{\nu}^{(+)}_{g^-, g^+}(t)
E\Big[F\big(
X_{[0,t]}^{0,g^+(t),(g^-, g^+)} \oplus_{t}  
X_{[t, 1]}^{g^+(t),(g^-, g^+)}\big)\Big] dt 
\label{FirstOrderBdryTerm_p}
\end{align}
appear, where 
$g^{\pm}$ are $\mathbb{R}$-valued $C^2$-functions defined on $[0, 1]$ that satisfy 
$ g^-(0)<0<g^+(0)$ and 
$\min_{0 \leq t \leq 1}(g^+(t) - g^-(t)) > 0$, 
$F$ is a $\mathbb{R}$-valued continuous function on $C([0,1], \mathbb{R})$, 
$\{ \widetilde{\nu}^{(\pm)}_{g^-, g^+}(t)\}_{0\leq t \leq 1}$ are some functions. 
Further, we are currently investigating higher-order chain rules 
to compute higher-order Greeks for barrier options. 
For example, the following terms 
\begin{align}
&\int_0^1 \left(
\int_0^{t_2}
\widetilde{\nu}^{(-,-)}_{g^-, g^+}(t_1, t_2)
E\Big[F\big(
X_{[0,t_1]}^{0,g^-(t_1),(g^-, g^+)} \oplus_{t_1} 
X_{[t_1,t_2]}^{g^-(t_1), g^-(t_2),(g^-, g^+)}\oplus_{t_2} 
X_{[t_2, 1]}^{g^-(t_2),(g^-, g^+)}\big)\Big] dt_1 \right) dt_2 ,
\label{SecondOrderBdryTerm_mm}\\
&\int_0^1 \left(
\int_0^{t_2}
\widetilde{\nu}^{(-,+)}_{g^-, g^+}(t_1, t_2)
E\Big[F\big(
X_{[0,t_1]}^{0,g^-(t_1),(g^-, g^+)} \oplus_{t_1} 
X_{[t_1,t_2]}^{g^-(t_1), g^+(t_2),(g^-, g^+)}\oplus_{t_2} 
X_{[t_2, 1]}^{g^+(t_2),(g^-, g^+)}\big)\Big] dt_1 \right) dt_2 ,
\label{SecondOrderBdryTerm_mp}\\
&\int_0^1 \left(
\int_0^{t_2}
\widetilde{\nu}^{(+,-)}_{g^-, g^+}(t_1, t_2)
E\Big[F\big(
X_{[0,t_1]}^{0,g^+(t_1),(g^-, g^+)} \oplus_{t_1} 
X_{[t_1,t_2]}^{g^+(t_1), g^-(t_2),(g^-, g^+)}\oplus_{t_2} 
X_{[t_2, 1]}^{g^-(t_2),(g^-, g^+)}\big)\Big] dt_1 \right) dt_2 ,
\label{SecondOrderBdryTerm_pm}\\
&\int_0^1 \left(
\int_0^{t_2}
\widetilde{\nu}^{(+,+)}_{g^-, g^+}(t_1, t_2)
E\Big[F\big(
X_{[0,t_1]}^{0,g^+(t_1),(g^-, g^+)} \oplus_{t_1} 
X_{[t_1,t_2]}^{g^+(t_1), g^+(t_2),(g^-, g^+)}\oplus_{t_2} 
X_{[t_2, 1]}^{g^+(t_2),(g^-, g^+)}\big)\Big] dt_1 \right) dt_2 
\label{SecondOrderBdryTerm_pp}
\end{align}
have been found to appear in a second-order chain rule, 
where $\{ \widetilde{\nu}^{(\pm, \pm)}_{g^-, g^+}(t_1,t_2)\}_{0\leq t_1<t_2 \leq 1}$ are some functions. 
Note that the stochastic processes
\begin{align*}
X_{[0,t_1]}^{0,g^-(t_1),(g^-, g^+)}, \quad X_{[0,t_1]}^{0,g^+(t_1),(g^-, g^+)}, \quad 
X_{[t_1,t_2]}^{g^-(t_1), g^-(t_2),(g^-, g^+)}, \quad X_{[t_1,t_2]}^{g^+(t_1), g^+(t_2),(g^-, g^+)}
\end{align*}
appearing in 
\eqref{SecondOrderBdryTerm_mm}, 
\eqref{SecondOrderBdryTerm_mp}, 
\eqref{SecondOrderBdryTerm_pm}, and
\eqref{SecondOrderBdryTerm_pp}
are already constructed in Lemma~\ref{Lem_for_H_gm_to_gp}. 
Note also that the stochastic processes
\begin{align*}
X_{[t_2, 1]}^{g^-(t_2),(g^-, g^+)}, \quad 
X_{[t_2, 1]}^{g^+(t_2),(g^-, g^+)}
\end{align*} 
appearing in 
\eqref{SecondOrderBdryTerm_mm}, 
\eqref{SecondOrderBdryTerm_mp}, 
\eqref{SecondOrderBdryTerm_pm}, and
\eqref{SecondOrderBdryTerm_pp}
are already constructed in Lemma~\ref{Lem_WeakConv_Meander_btwn_2crvs}.
Furthermore, the stochastic processes
\begin{align*}
X_{[t_1,t_2]}^{g^-(t_1), g^+(t_2),(g^-, g^+)}, \quad
X_{[t_1,t_2]}^{g^+(t_1), g^-(t_2),(g^-, g^+)}
\end{align*}
which appear in 
\eqref{SecondOrderBdryTerm_mp} and
\eqref{SecondOrderBdryTerm_pm}, 
are scaled Brownian house-movings. 
We aim to develop these higher-order chain rules further 
and apply them to compute higher-order Greeks for barrier options in our future work.

\appendix
\section{Appendix}

In this appendix, we prepare several results. 
Although some of the results in this appendix are either well known or easy to obtain, 
we prove them for completeness.

\subsection{Some basic results}

Here, we prepare some basic results which are used throughout the paper.

\begin{lem}\label{Ap_BM_Density_Formulas_type1}
Let $W = \{ W(t) \}_{t \geq 0}$ be the standard one-dimensional Brownian motion 
defined on $(\Omega, \mathcal{F}, P)$.  For $t>0$, we have
\begin{align}
&P\left(W(t) \in dz, m_t(W) \geq -\varepsilon \right)
= \left( n_t(z) - n_t(z+2\varepsilon) \right) dz, 
\qquad (z>-\varepsilon), 
\label{BM_Density_No1_Wt_mt}\\
&P\left(W(t) \in dz, -\varepsilon \leq m_t(W) < M_t(W) \leq \eta \right) 
\label{BM_Density_No2_Wt_mt_Mt}  \\
&\quad = \sum_{k=-\infty}^{\infty} 
\left( n_t(z+2k(\eta +\varepsilon )) 
- n_t(2\eta -z+2k(\eta +\varepsilon )) \right) dz ,
\qquad (-\varepsilon<z< \eta).
\nonumber
\end{align}
For $0<t<u$, we have
\begin{align}
&P\left( W(t) \in dy, W(u) \in dz, m_u(W) \geq -\varepsilon \right)
\label{BM_Density_No3_Wt_Wu_mu}\\
&\quad =\left( n_{u-t}(z-y) - n_{u-t}(z+y+2\varepsilon ) \right) 
\left( n_t(y) - n_t(y+2\varepsilon ) \right) dy dz, 
\qquad (y,z>-\varepsilon).
\nonumber 
\end{align}
For $0<s<t<u$, we have
\begin{align}
&P\left(  W(s) \in dx, W(t) \in dy, W(u) \in dz, m_u(W) \geq -\varepsilon \right) 
\label{BM_Density_No5_Ws_Wt_Wu_mu}\\
&\quad = 
\left( n_{u-t}(z-y) - n_{u-t}(z+y+2\varepsilon ) \right)\nonumber \\
&\qquad \times \left( n_{t-s}(y-x) - n_{t-s}(y+x+2\varepsilon ) \right) \nonumber \\
&\qquad \times \left( n_s(x) - n_s(x+2\varepsilon ) \right) dx dy dz. 
\qquad (x,y,z>-\varepsilon) \nonumber
\end{align}
\end{lem}
\begin{proof}
In this proof, 
$(\Omega, \mathcal{F}), W = \{ W(t) \}_{t \geq 0} , 
\left(P^a \right)_{a \in \mathbb{R}}$ 
denotes the one-dimensional Brownian family, 
and $P^0$ is written simply as $P$. 
We can find (\ref{BM_Density_No1_Wt_mt}) and 
(\ref{BM_Density_No2_Wt_mt_Mt}) in \cite{bib_Billingsley}.
Using the Markov property of $W$ and (\ref{BM_Density_No1_Wt_mt}), 
we have
\begin{align*}
&P\left( W(t) \in dy, W(u) \in dz, m_u(W) \geq -\varepsilon \right) \\
&\quad = 
E\left[ P\left( W(t) \in dy, m_u(W) \geq -\varepsilon, W(u) \in dz 
~|~ \mathcal{F}_t^W \right) \right] \\
&\quad = P^y \left( W(u-t) \in dz, m_{u-t}(W) \geq -\varepsilon \right) 
P\left( W(t) \in dy, m_t(W) \geq -\varepsilon \right) \\
&\quad = P\left( y+W(u-t) \in dz, y+m_{u-t}(W) \geq -\varepsilon \right) 
P\left( W(t) \in dy, m_t(W) \geq -\varepsilon \right) \\
&\quad = \left( n_{u-t}(z-y) - n_{u-t}(z+y+2\varepsilon ) \right) 
\left( n_t(y) - n_t(y+2\varepsilon ) \right) dy dz .
\end{align*}
Using the Markov property of $W$,  (\ref{BM_Density_No1_Wt_mt}), 
and (\ref{BM_Density_No3_Wt_Wu_mu}), we also have
\begin{align*}
&P\left( W(s) \in dx, W(t) \in dy, W(u) \in dz, m_u(W) \geq -\varepsilon \right) \\
&\quad = E\left[ P\left( W(s) \in dx, W(t) \in dy, W(u) \in dz, 
m_u(W) \geq -\varepsilon~|~\mathcal{F}_t^W \right) \right] \\
&\quad = P^y \left( W(u-t) \in dz, m_{u-t}(W) \geq -\varepsilon \right) \\
&\qquad \times P\left( W(s) \in dx, W(t) \in dy, m_t(W) \geq -\varepsilon \right) \\
&\quad = P\left( y+W(u-t) \in dz, y+m_{u-t}(W) \geq -\varepsilon \right) \\
&\qquad \times
P\left( W(s) \in dx, W(t) \in dy, m_t(W) \geq -\varepsilon \right) \\
&\quad = 
\left( n_{u-t}(z-y) - n_{u-t}(z+y+2\varepsilon )\right)
\left( n_{t-s}(y-x) - n_{t-s}(y+x+2\varepsilon ) \right) \\
&\qquad \times
\left( n_s(x) - n_s(x+2\varepsilon ) \right) dx dy dz.
\end{align*}
\end{proof}

The following lemma is obtained by Prohorov's theorem.
\begin{lem}\label{Ap_Lem_weakconv_product}
Let $S_1$ and $S_2$ be Polish spaces.  
Let $X_n$ and $Y_n$ be random variables 
defined on $(\Omega_n, \mathcal{F}_n, P_n)$ 
that take values in $S_1$ and $S_2$, respectively. 
If $X_n$ and $Y_n$ are independent and 
$P_n \circ X_n^{-1}$ and $P_n \circ Y_n^{-1}$ converge weakly to 
probability measures $Q$ on $S_1$ and $R$ on $S_2$, respectively, 
then $P_n \circ (X_n, Y_n)^{-1}$ 
converges weakly to the product measure $Q \times R$.
\end{lem}

\begin{lem}\label{Ap_Lem_HittingTime_of_BES3bridge}
Let $T>0$, and 
let $\mathbb{R}$-valued $C^1$-function $g$ defined on $[0, T]$ that satisfies $\min_{0\leq t\leq T}g(t)>0$.  
Then, for the BES($3$)-process $R_{[0, T]}$ 
starting at $0$ on $[0, T]$ and $b \in (0, g(T))$, we have
\begin{align*}
P\left(T_g \leq T, R_{[0, T]} \in K_{[0, T]}^-(g), R_{[0, T]}(T) \in db \right) = 0.
\end{align*}
Here, $T_g$ denotes the hitting time of $R_{[0, T]}$ to $g$.
\end{lem}
\begin{proof}
Since $b \in (0, g(T))$, what we must prove is
\begin{align}\label{HittingTimeBES3bridge_proof_eq_step1}
P\left(T_g < T, R_{[0, T]} \in K_{[0, T]}^-(g), R_{[0, T]}(T) \in db \right) = 0.
\end{align}
Here, $R_{[0, T]}$ satisfies
\begin{align*}
R_{[0, T]}(t)=g(T_g)+\int_{T_g}^t \frac{1}{R_{[0, T]}(s)}ds
+(\beta(t)-\beta(T_g)), \quad t\in [T_g, T]
\end{align*}
on $\{T_g < T\}$, where $\beta$ is a standard one-dimensional Brownian motion. 
Using sample path properties of $\beta$, 
we obtain \eqref{HittingTimeBES3bridge_proof_eq_step1}.
\end{proof}

\begin{lem}\label{Lem_Boundary_Prob_for_Meander_BES3bridge}
Let $g$ be an $\mathbb{R}$-valued $C^1$-function defined on $[0,1]$ that satisfies $\min_{0 \leq t \leq 1} g(t) > 0$.  
Then, for $b\in (0, g(1))$, we have  
$P(r^{0 \to b} \in \partial K^-(g)) = 0$ and $P(W^+ \in \partial K^-(g)) = 0$. 
\end{lem}
\begin{proof}
We obtain
\begin{align}\label{partialKming}
&\partial K^-(g) = K^-(g) - {\rm int}(K^-(g))= \left\{ w \in K^-(g) \ \big\vert \ \min_{0 \leq t \leq 1} (g(t) - w(t)) = 0 \right\}.
\end{align}
Then, \eqref{partialKming} and Lemma~\ref{Ap_Lem_HittingTime_of_BES3bridge} imply that
$P(r^{0 \to b} \in \partial K^-(g)) = 0$ and 
\begin{align*}
P(W^+ \in \partial K^-(g) ) = \int_0^\infty P(r^{0 \to b} \in \partial K^-(g)) P(W^+(1) \in db) = 0
\end{align*}
hold.
\end{proof}

We can find the following proposition in \cite{bib_Durret_1977}, 
which is stated there without proof.
\begin{prop}\label{Ap_Prop_cond_Markov_is_Markov}
Let $(T, \mathcal{T})$ be a measurable space and 
$(\Omega, \mathcal{F}, P)$ be a probability space, 
and let $Y = \{Y(t), \mathcal{F}_t^{Y}, 0 \leq t \leq 1 \}$ 
be a $T$-valued Markov process on $(\Omega, \mathcal{F}, P)$. 
For $\Lambda \in \mathcal{F}$ with $P(\Lambda) > 0$, 
we define a new probability space 
$(\Lambda, \Lambda \cap \mathcal{F}, P_{\Lambda})$ by 
$\Lambda \cap \mathcal{F} := \{ \Lambda \cap F \mid F \in \mathcal{F} \}$ and
\begin{align*}
P_{\Lambda}(F) := \frac {P(\Lambda \cap F)} {P(\Lambda)}.
\end{align*}
Assume that for $t \in [0,1]$ there exist $A_t \in \mathcal{F}_t^{Y}$ 
and $B_t \in \sigma(Y(s)~|~t \leq s \leq 1)$ 
that satisfy $\Lambda = A_t \cap B_t$. 
If we write the restriction $Y$ to 
$(\Lambda, \Lambda \cap \mathcal{F}, P_{\Lambda})$ as $Y_{\Lambda}$, 
then 
$Y_{\Lambda} = \{Y_{\Lambda}(t), \mathcal{F}_t^{Y_{\Lambda}}, 0 \leq t \leq 1 \}$ 
is a $T$-valued Markov process
on $(\Lambda, \Lambda \cap \mathcal{F}, P_{\Lambda})$.
\end{prop}
\begin{proof}
For $0<t < s\leq1$ and $\Gamma \in \mathcal{T}$, 
we must show that 
$P_{\Lambda}(Y_{\Lambda}(s) \in \Gamma\, |\, \mathcal{F}_t^{Y_{\Lambda}})$ 
has a $\sigma(Y_{\Lambda}(t))$-measurable version. 
Let $0= t_0 < t_1 < \cdots < t_n = t < s \leq 1$, 
$K_0, K_1, \dots, K_n \in \mathcal{T}$ be given. 
We define a measure $\mu$ on $(T^{n+1}, \mathcal{T}^{n+1})$ by
\begin{align*}
\mu (C) := P(\{ (Y(t_0), Y(t_1), \dots, Y(t_n)) \in C \} \cap A_t), 
\qquad C \in \mathcal{T}^{n+1}.
\end{align*}
Using the Markov property of $Y$, we obtain
\begin{align}
&P(\{ Y(t_0) \in K_0, Y(t_1) \in K_1, \dots, Y(t_n) \in K_n, Y(s) \in \Gamma \} \cap \Lambda) \nonumber \\
&\quad = E[P(\{ Y(s) \in \Gamma \} \cap B_t\, |\, \mathcal{F}_t^{Y})
\ ;\ \{ Y(t_0) \in K_0, Y(t_1) \in K_1, \dots, Y(t_n) \in K_n \} \cap A_t] \nonumber \\
&\quad = \int_{K_0 \times K_1 \times \cdots \times K_n} 
P(\{ Y(s) \in \Gamma \} \cap B_t\, |\, Y(t_n) = x_n) \mu(dx_0 dx_1 \cdots dx_n).
\label{eq1_Makov_Prop}
\end{align}
On the other hand, 
for any $\mathcal{T}/\mathcal{B}([0,\infty))$-measurable function 
$f :T \to [0, \infty)$, we have
\begin{align}
&E[f(Y(t_n))\ ;\ \{ Y(t_0) \in K_0, Y(t_1) \in K_1, \dots, Y(t_n) \in K_n \} \cap \Lambda] 
\nonumber \\
&\quad = E[P(B_t\, |\, \mathcal{F}_t^{Y})f(Y(t_n)) 
\ ;\ \{ Y(t_0) \in K_0, Y(t_1) \in K_1, \dots, Y(t_n) \in K_n \} \cap A_t] \nonumber \\
&\quad = \int_{K_0 \times K_1 \times \cdots \times K_n} P(B_t\, |\, Y(t_n) = x_n)f(x_n) \mu(dx_0 dx_1 \cdots dx_n), 
\label{eq2_f_Makov_Prop}
\end{align}
where $h(x_n)=P(B_t\, |\, Y(t_n) = x_n)$ is a Borel measurable function satisfying 
$h(Y(t_n))=E[1_{B_t} \ \vert \ \sigma(Y(t_n))]$.
Therefore, applying \eqref{eq2_f_Makov_Prop} for
\begin{align*}
f(x_n) := &
\begin{cases}
\dfrac{P(\{ Y(s) \in \Gamma \} \cap B_t\, |\, Y(t) = x_n)}{P(B_t\, |\, Y(t) = x_n)}, 
& \text{{\rm for}} ~P(B_t\, |\, Y(t) = x_n) > 0, \\
0, & \text{{\rm otherwise}},
\end{cases} 
\end{align*}
we can obtain
\begin{align*}
&P(\{ Y(t_0) \in K_0, Y(t_1) \in K_1, \dots, Y(t_n) \in K_n, Y(s) \in \Gamma \} \cap \Lambda) \\
&\quad = E[f(Y(t_n))\ ;\ \{ Y(t_0) \in K_0, Y(t_1) \in K_1, \dots, Y(t_n) \in K_n \} \cap \Lambda]
\end{align*}
by \eqref{eq1_Makov_Prop}. Dividing by $P(\Lambda)$, we obtain
\begin{align*}
&P_{\Lambda}(Y_{\Lambda}(t_0) \in K_0, Y_{\Lambda}(t_1) \in K_1, \dots, 
Y_{\Lambda}(t_n) \in K_n, Y_{\Lambda}(s) \in \Gamma) \\
&\quad = E_{\Lambda}[f(Y_{\Lambda}(t_n)) 
\ ;\ Y_{\Lambda}(t_0) \in K_0, Y_{\Lambda}(t_1) \in K_1, \dots, Y_{\Lambda}(t_n) \in K_n],
\end{align*}
and, hence, Dynkin's $\pi$-$\lambda$ theorem yields 
$P_{\Lambda}(Y_{\Lambda}(s) \in \Gamma\, |\, \mathcal{F}_t^{Y_{\Lambda}}) 
= f(Y_{\Lambda}(t_n)) = f(Y_{\Lambda}(t))$.
\end{proof}


\subsection{Weak convergence to BES($3$)-bridge}
\label{section_Brownianbridge_to_BES3_bridge}

It has been shown in \cite{bib_Durret_1977} 
that the one-dimensional Brownian bridge from $0$ to $0$ 
conditioned to stay in $[-\varepsilon, \infty)$ converges weakly to 
the Brownian excursion (i.e., the BES($3$)-bridge from $0$ to $0$). 
Motivated by this research, we prove the following weak convergence 
that is used to construct the Brownian house-moving. 

\begin{thm}\label{Th_conv_Brownianbridge_to_3dBesselbridge}
Let $b\geq 0$ and $B^{0 \to b} = \{ B^{0 \to b} (t) \}_{t \in [0,1]}$ 
be the one-dimensional Brownian bridge from $0$ to $b$ on $[0, 1]$, 
and let $r^{0 \to b} = \{ r^{0 \to b}(t) \}_{t \in [0,1]}$ 
be the BES($3$)-bridge from $0$ to $b$ on $[0, 1]$. Then, we have 
\begin{align*}
B^{0 \to b}|_{K^+(-\varepsilon )} \xrightarrow{\mathcal{D}} r^{0 \to b}, 
\qquad \varepsilon \downarrow 0,
\end{align*}
where $K^+(-\varepsilon ) := 
\{ w = \{ w(t) \}_{t \in [0,1]} \in C([0,1], \mathbb{R} )~|~
-\varepsilon \leq w(t), ~0 \leq t \leq 1 \}$. 
\end{thm}

Here, we prepare Lemma~\ref{Ap_Lem_sameDist_BESbridge_BMbridge} 
which is used in the proof of 
Theorem~\ref{Th_conv_Brownianbridge_to_3dBesselbridge}.

\begin{lem}\label{Ap_Lem_sameDist_BESbridge_BMbridge}
Let $T>0$ and $a, b>0$. 
Then, for $0 < s < t < T$ and $x, y > 0 $, we have
\begin{align}
&P\left( B_{[0,T]}^{a \to b}|_{K^+(0)}(t) \in dy \right) 
= P\left( r_{[0,T]}^{a \to b}(t) \in dy \right)\nonumber \\
&\quad = \left( \frac{T}{2\pi t(T-t)} \right)^{\frac{1}{2}} 
\frac{\left( e^{- \frac {(y-a)^2} {2t}} - e^{- \frac {(y+a)^2} {2t}} \right) 
\left( e^{- \frac {(b-y)^2} {2(T-t)}} - e^{- \frac {(b+y)^2} {2(T-t)}} \right)}
{e^{- \frac {(b-a)^2} {2T}} - e^{- \frac {(b+a)^2} {2T}}}dy, 
\label{Lem_markov_condBES_Density}\\
&P\left( B_{[0,T]}^{a \to b}|_{K^+(0)}(t) \in dy
~|~B_{[0,T]}^{a \to b}|_{K^+(0)}(s) = x \right) 
= P\left( r_{[0,T]}^{a \to b}(t) \in dy~|~r_{[0,T]}^{a \to b}(s) = x \right) \nonumber \\
&\quad = \left( \frac {T-s} {2 \pi (t-s)(T-t)} \right)^{\frac {1} {2}} 
\frac{\left( e^{- \frac {(y-x)^2} {2(t-s)}} - e^{- \frac {(y+x)^2} {2(t-s)}} \right) 
\left( e^{- \frac {(b-y)^2} {2(T-t)}} - e^{- \frac {(b+y)^2} {2(T-t)}} \right)}
{e^{- \frac {(b-x)^2} {2(T-s)}} - e^{- \frac {(b+x)^2} {2(T-s)}}}dy.
\label{Lem_markov_condBES_TransDensity}
\end{align}
Therefore, the Markov processes $B_{[0,T]}^{a \to b}|_{K^+(0)}$ 
and $r_{[0,T]}^{a \to b}$ follow the same distribution.
\end{lem}
\begin{proof}
$P\left( r_{[0,T]}^{a \to b}(t) \in dy \right)$ and 
$P\left( r_{[0,T]}^{a \to b}(t) \in dy~|~r_{[0,T]}^{a \to b}(s) = x \right)$ 
are given in \cite{bib_RevuzYor} p.~463. 
In the following, 
$(\Omega, \mathcal{F}), W = \{ W(t) \}_{t \geq 0} , 
\left(P^a \right)_{a \in \mathbb{R}}$ 
denotes the one-dimensional Brownian family, 
and $P^0$ is written simply as $P$. 
Using the Markov property of $W$ and (\ref{BM_Density_No1_Wt_mt}), we have
\begin{align*}
&P \left( B_{[0,T]}^{a \rightarrow b}|_{K^{+}(0)}(t) \in dy \right) \\
&\quad = \frac{P^{y} \left( W(T-t) \in db,m_{T-t}(W) \geq 0 \right) P \left( W(t) \in dy-a, m_{t}(W) \geq -a \right)}{P \left( m_{T}(a+W) \geq 0, a+W(T) \in db \right)} \\
&\quad = \frac{\left( n_{T-t}(b-y) - n_{T-t}(b+y) \right) \left( n_{t}(y-a) - n_{t}(y+a) \right)}{n_{T}(b-a)-n_{T}(b+a)} dy.
\end{align*}
Hence, \eqref{Lem_markov_condBES_Density} holds. 
Proposition~\ref{Ap_Prop_cond_Markov_is_Markov} implies that 
$B_{[0,T]}^{a \to b}|_{K^+(0)}$ is a Markov process.
Therefore we obtain \eqref{Lem_markov_condBES_TransDensity}.
\end{proof}

\subsubsection{Proof of Theorem~\ref{Th_conv_Brownianbridge_to_3dBesselbridge}}

Take an arbitrary sequence $\{\varepsilon_n\}_n$ such that $\varepsilon_n \downarrow 0$ $(n\to \infty)$. 
Then, by Proposition~2.4 in \cite{bib_Wu},  
there exist Bessel bridges $r^{c\to d}$ $(c, d\geq 0)$ 
defined on a common probability space $(\Omega, \mathcal{F}, P)$ such that 
$\{r^{\varepsilon_n \to b+\varepsilon_n}\}_n$ converges to $r^{0 \to b}$, $P$-almost surely. 
On the other hand, 
using Lemma~\ref{Ap_Lem_sameDist_BESbridge_BMbridge}, we obtain
\begin{align*}
B^{0\to b}\vert_{K^+(-\varepsilon_n)}
\overset{\mathcal{D}}{=}
B^{\varepsilon_n \to b+\varepsilon_n}\vert_{K^+(0)}-\varepsilon_n 
\overset{\mathcal{D}}{=}
r^{\varepsilon_n \to b+\varepsilon_n}-\varepsilon_n \qquad (n\in \mathbb{N}).
\end{align*}
Therefore, we obtain our assertion.

\subsection{Distribution of the maximal value of the BES($3$)-bridge}
\label{section_maxdist_BES3_bridge}

As an application of Theorem~\ref{Th_conv_Brownianbridge_to_3dBesselbridge}, 
we derive the distribution of the maximal value of the BES($3$)-bridge 
$r^{0 \to b}$ $(b>0)$.

\begin{prop}\label{Prop_max_dist_of_BES3bridge}
For each $x > b>0$, we have
\begin{align*}
P\left( M(r^{0 \to b}) \leq x \right)
=\frac{J^{(x)}(1,b)}{2b n_1(b)}>0.
\end{align*}
\end{prop}
\begin{proof}
Using (\ref{BM_Density_No3_Wt_Wu_mu}) and 
(\ref{BM_Density_No5_Ws_Wt_Wu_mu}), we have
\begin{align*}
&P\left( M(B^{0 \to b}|_{K^+(-\varepsilon )}) \leq x \right)
= P\left(  M(B^{0 \to b}) \leq x ~|~-\varepsilon \leq m(B^{0 \to b}) \right) \\
&\quad = \frac{P\left(-\varepsilon \leq m(W) < M(W) \leq x , W(1) \in db \right)}
{P\left( -\varepsilon \leq m(W), W(1) \in db \right)}
=\dfrac{\psi_1(\varepsilon)}{\psi_2(\varepsilon)}
\end{align*}
for $\varepsilon>0$, where
\begin{align*}
&\psi_1(\varepsilon):=
\sum_{k=-\infty}^{\infty} \left( n_1( b+2k(x+\varepsilon ) ) 
- n_1( 2\varepsilon+b+2k(x+\varepsilon ) )\right),
\qquad
\psi_2(\varepsilon):=n_1(b) - n_1(b+2\varepsilon ).
\end{align*}
By simple calculations, we obtain
\begin{align*}
&\lim_{\varepsilon \downarrow 0}\psi_i(\varepsilon)=0, \quad i=1,2, \quad
\lim_{\varepsilon \downarrow 0} 
\frac{d}{d \varepsilon} \psi_2(\varepsilon) = 2b n_1(b), \quad
\lim_{\varepsilon \downarrow 0}
\frac{d}{d \varepsilon} \psi_1(\varepsilon) =J^{(x)}(1,b).
\end{align*}
By combining Theorem~\ref{Th_conv_Brownianbridge_to_3dBesselbridge}, 
the Portmanteau theorem, 
Lemma~\ref{Lem_Boundary_Prob_for_Meander_BES3bridge}, 
and L'H\^{o}pital's rule, we obtain
\begin{align}
P\left( M(r^{0 \to b}) \leq x \right)
=\lim_{\varepsilon \downarrow 0} 
P\left( M(B^{0 \to b}|_{K^+(-\varepsilon )}) \leq x \right) 
=\frac{J^{(x)}(1,b)}{2b n_1(b)}. \label{eq_Max_dist_cond_bridge}
\end{align}
Now, we define the domain $D$, and the function $f$ on $D$ as
\begin{align*}
D &= \left\{ z= x+ {\rm i}y \mid x \in (0, \infty), y \in (-b/2, b/2) \right\},\\
f(z) &= \sum_{k=-\infty}^{\infty} (b+2k(b+z)) 
\exp\left( -\frac{(b+2k(b+z))^2}{2} \right),\qquad z \in D,
\end{align*}
where ${\rm i}$ is the imaginary unit. 
Then, we have
\begin{align*}
P\left(M(r^{0 \to b}) \leq b+\eta \right) 
=\frac{J^{(b+\eta)}(1,b)}{2b n_1(b)}
= \frac{f(\eta)}{b \exp \left( -\frac{b^2}{2} \right)},\qquad \eta > 0,
\end{align*}
by \eqref{eq_Max_dist_cond_bridge}. Furthermore, we define
\begin{align*}
D_R &= \left\{ z= x+ {\rm i}y \mid x \in (0, R), y \in (-b/2, b/2) \right\}, 
\qquad R>0.
\end{align*}
For $R>0$, $z\in D_R$, and $k\in \mathbb{Z}$, we have
\begin{align*}
|b+2k(b+z)| \leq b + 2|k|(b+|x|+|y|) \leq b+ 2|k|(2b+R) 
\end{align*}
and
\begin{align*}
&\exp\left( -\frac{(b+2k(b+z))^2}{2} \right) 
=\exp\left( -\frac{1}{2}b^2-2k(b+x)b-2k^2(b+x)^2  +2k^2 y^2 \right) \\
&\quad \leq \exp\left( 2|k|(b+R)b-2k^2b^2  +\frac{1}{2}k^2 b^2 \right) 
= \exp\left( -\frac{3}{2}k^2b^2 +2|k|(b+R)b \right).
\end{align*}
Thus, we see that $f$ is a holomorphic function on $D$. 

For the sake of contradiction, assume that 
$f(\eta_0) = 0$ holds for some $\eta_0 > 0$. 
Then, because $f$ is a non-decreasing and non-negative function on $(0, \infty)$, 
$f(z)=0$, $z \in D$ holds by the identity theorem. However, this contradicts
\begin{align*}
\lim_{\eta \to \infty} f(\eta) 
= b \exp \left( -\frac{b^2}{2} \right) 
\lim_{\eta \to \infty} P(M(r^{0 \to b}) \leq b+\eta) 
= b \exp \left( -\frac{b^2}{2} \right) > 0.
\end{align*}
\end{proof}

\begin{rem}
More generally, in \cite {bib_PitmanYor1999} p.~8 (28), 
Proposition~\ref{Prop_max_dist_of_BES3bridge} 
has been shown by the expanded Gikhman--Kiefer formula 
for BES($\delta$)-bridges. 
To derive Proposition~\ref{Prop_max_dist_of_BES3bridge} from \cite {bib_PitmanYor1999} p.~8 (28), 
we require a proof that employs the Fourier expansion of the heat kernel, 
which is more complex than the proof of Proposition~\ref{Prop_max_dist_of_BES3bridge}.
The Fourier expansion technique for the heat kernel is also employed 
in the proof of Proposition~8.1 in \cite{bib_IshitaniRinYanashima}.
\end{rem}

\begin{cor}\label{cor_scaled_BES3bridge_max_dist}
For $0\leq s<t < \infty$, it holds that
\begin{align*}
P\left( r_{[s,t]}^{0 \to y} \in K_{[s,t]}^-(c) \right)
=\frac{(t-s)J^{(c)}(t-s,y)}{2y n_{t-s}(y)}>0,
\qquad 0<y<c .
\end{align*}
\end{cor}
\begin{proof}
Using Proposition~\ref{Prop_max_dist_of_BES3bridge}, we obtain
\begin{align*}
&P\left( r_{[s,t]}^{0 \to y} \in K_{[s,t]}^-(c) \right)
=P\left( M(r^{0 \to y/\sqrt{t-s}}) \leq c/\sqrt{t-s} \right) \\
&\quad =\frac{\sqrt{t-s}\cdot J^{(c/\sqrt{t-s})}(1,y/\sqrt{t-s})}
{2yn_1(y/\sqrt{t-s})}
=\frac{(t-s)J^{(c)}(t-s,y)}{2yn_{t-s}(y)}.
\end{align*}
\end{proof}

\begin{cor}\label{Cor_Prob_BES3bridge_under_single_curve}
Assume that $g \in C([0,1], \mathbb{R})$ satisfies 
$\min_{0 \leq t \leq 1} g(t) > 0$. 
Then, we have
\begin{align*}
(A)\quad P(W^+ \in K^-(g) ) > 0 \qquad \mbox{and}\qquad 
(B)\quad P(r^{0 \to b} \in K^-(g)) > 0, \quad 0<b<g(1).
\end{align*}
\end{cor}

\begin{rem}
Corollary~\ref{Cor_Prob_BES3bridge_under_single_curve} implies that 
we can define the following random variables:
\begin{align}\label{restricted_meander_and_bridge}
(A)\quad W^+\vert_{K^-(g)} \qquad \mbox{and}\qquad 
(B)\quad r^{0 \to b}\vert_{K^-(g)}, \quad 0<b<g(1).
\end{align}
Note that the conditioned random variables 
in Lemmas~\ref{Lem_for_H_gm_to_gp} and \ref{Lem_WeakConv_Meander_btwn_2crvs} 
are essentially the same as the conditioned random variables in \eqref{restricted_meander_and_bridge}.
\end{rem}

Here, we prepare some results which are used in the proof of  
Corollary~\ref{Cor_Prob_BES3bridge_under_single_curve}.

\begin{lem}\label{Ap_Lem_Joint_Dist_BES_bridge_1}
Let $\delta>0$, $a \geq 0$, and $b>0$. Then, for the BES($\delta$)-bridge 
$r^{a \to b} = \{r^{a \to b}(t) \}_{t \in [0,1]}$ 
from $a$ to $b$ on $[0, 1]$, we have
\begin{align*}
P\left( r^{a \to b}(t) \in dy, M(r^{a \to b}) \leq x \right)
= P\left( r_{[0,t]}^{a \to y} \in K_{[0,t]}^-(x) \right) 
P\left( r_{[t,1]}^{y \to b} \in K_{[t,1]}^-(x) \right) 
P\left( r^{a \to b}(t) \in dy \right)
\end{align*}
for $0<t<1$ and $b \leq x,\ 0 \leq y \leq x$. 
Here, $r_{[t_1,t_2]}^{c \to d} = \{r_{[t_1,t_2]}^{c \to d}(t) \}_{t \in [t_1,t_2]}$ 
denotes the BES($\delta$)-bridge from $c$ to $d$ on $[t_1,t_2]$.
\end{lem}
\begin{proof}
In this proof, the pair $(R=\{R(t)\}_{t\geq 0}, P^{(\delta)}_a)$ denotes 
a BES($\delta$)-process starting from $a \geq 0$: $P^{(\delta)}_a(R(0)=a)=1$. 
Then, by the Markov property of $R$, we have
\begin{align*}
&P\left( r^{a \to b}(t) \in dy, M(r^{a \to b}) \leq x \right)
= \frac{P^{(\delta)}_a \left( R(t) \in dy, M(R) \leq x, R(1) \in db \right)}
{P^{(\delta)}_a \left( R(1) \in db \right)} \\
&\quad = \frac{P^{(\delta)}_y \left( R(1-t) \in db, M_{1-t}(R) \leq x \right) 
P^{(\delta)}_a \left( R(t) \in dy, M_t(R) \leq x \right)}
{P^{(\delta)}_a \left( R(1) \in db \right)}
\end{align*}
and
\begin{align*}
P\left( r^{a \to b}(t) \in dy \right)
= \frac{P^{(\delta)}_a \left( R(t) \in dy, R(1) \in db \right)}
{P^{(\delta)}_a \left( R(1) \in db \right)} 
= \frac{P^{(\delta)}_y \left( R(1-t) \in db \right) 
P^{(\delta)}_a \left( R(t) \in dy \right)}
{P^{(\delta)}_a \left( R(1) \in db \right)} .
\end{align*}
Therefore, because we have
\begin{align*}
P^{(\delta)}_y \left( R(1-t) \in db, M_{1-t}(R) \leq x \right)
& = P\left( M_{1-t}(r^{y \to b}) \leq x \right) 
P^{(\delta)}_y \left( R(1-t) \in db \right)\\ 
& = P\left( r_{[t,1]}^{y \to b} \in K_{[t,1]}^-(x) \right) 
P^{(\delta)}_y \left( R(1-t) \in db \right), \\
P^{(\delta)}_a \left( R(t) \in dy, M_t(R) \leq x \right)
&= P\left( M_t(r^{a \to y}) \leq x \right) 
P^{(\delta)}_a \left( R(t) \in dy \right)\\
&= P\left( r_{[0,t]}^{a \to y} \in K_{[0,t]}^-(x) \right) 
P^{(\delta)}_a \left( R(t) \in dy \right) ,
\end{align*}
it follows that
\begin{align*}
&P\left( r^{a \to b}(t) \in dy, M(r^{a \to b}) \leq x \right)\\
&\quad = P\left( r_{[0,t]}^{a \to y} \in K_{[0,t]}^-(x) \right) 
P\left( r_{[t,1]}^{y \to b} \in K_{[t,1]}^-(x) \right) 
\frac{P^{(\delta)}_a \left( R(t) \in dy \right) 
P^{(\delta)}_y \left( R(1-t) \in db \right)}
{P^{(\delta)}_a \left( R(1) \in db \right)} \\
&\quad = P\left( r_{[0,t]}^{a \to y} \in K_{[0,t]}^-(x) \right) 
P\left( r_{[t,1]}^{y \to b} \in K_{[t,1]}^-(x) \right) 
P\left( r^{a \to b}(t) \in dy \right).
\end{align*}
\end{proof}

\begin{lem}\label{Ap_Lem_Brwbridge_RectProb_Positive}
For $c>0$ and $a, b\in (0,c)$, it holds that $P(B^{a \to b} \in K(0, c))>0$.
\end{lem}
\begin{proof}
It holds that
\begin{align*}
f(z) = \sum_{k = - \infty}^{\infty} 
\left(n_1(b-a + 2k(a \vee b + z)) -n_1(b+a + 2k(a \vee b + z)) \right)
\end{align*}
defines a holomorphic function on 
\begin{align*}
D = \{ z= x+ {\rm i}y \mid x \in (0, \infty), y \in (-(b+a)/2, (b+a)/2) \}.
\end{align*} 
Lemma~\ref{Ap_BM_Density_Formulas_type1} and 
the following equations
\begin{align*}
&n_1(\xi)=n_1(-\xi), \quad
\sum_{k=-\infty}^{\infty} n_1(\xi+2k \tau)=\sum_{k=-\infty}^{\infty} n_1(\xi-2k \tau)
=\sum_{k=-\infty}^{\infty} n_1(\xi+2\tau+2k \tau), 
\end{align*}
imply that
\begin{align*}
f(\eta) = 
P(W \in K(-a, a \vee b + \eta -a), W(1) \in db-a) / db,\qquad \eta > 0.
\end{align*}
Observe that $f$ is non-decreasing and non-negative on $(0, \infty)$. 
Assume, for the sake of contradiction, 
that $f(\eta_0) = 0$ holds for some $\eta_0 > 0$. 
Then, it follows from the identity theorem that $f(z)=0$ holds for every $z \in D$. 
This contradicts 
\begin{align*}
&\lim_{\eta \to \infty} f(\eta) = P(W \in K^+(-a), W(1) \in db-a) / db \\
&\quad =P(B^{a\to b}\in K^+(0)) \frac{P(W(1) \in db-a)}{db}
=(1-e^{-2ab}) \frac{P(W(1) \in db-a)}{db}> 0.
\end{align*}
Here, it should be noted according to \cite{bib_KarazasShreve} that $P(B^{a\to b}\in K^+(0))=1-e^{-2ab}$ holds.
Therefore, $f(\eta) > 0$ holds for $\eta > 0$, and hence we obtain
\begin{align*}
P(B^{a \to b} \in K(0, c)) = \frac{f(c- a\vee b) db}{P(W(1) \in db-a)} > 0.
\end{align*}
\end{proof}

\begin{lem}\label{Ap_Lem_Brwbridge_TubeProb_Positive}
Let $a, b\in \mathbb{R}$. 
Assume that $\mathbb{R}$-valued continuous functions $g^-, g^+$ defined on $[0,1]$ 
satisfy the following conditions:
\begin{align*}
\min_{0 \leq t \leq 1} (g^+(t) - g^-(t)) > 0, \quad
g^-(0) < a < g^+(0), \quad g^-(1) < b < g^+(1).
\end{align*}
Then, we have
$P(B^{a \to b} \in K(g^-, g^+)) > 0$.
\end{lem}
\begin{proof}
Note that we can find $n\in \mathbb{N}$, 
$(c_i^{\pm})_{i=1}^n \subset \mathbb{R}$, 
$(t_i)_{i=0}^n \subset [0,1]$ such that
\begin{align*}
&0=t_0<t_1<\cdots <t_n=1,\quad c_1^-<a<c_1^+, \ c_n^-<b<c_n^+,\\
&\max_{ t\in [t_{i-1},t_i]}g^-(t)<c_i^-<c_i^+<\min_{ t\in [t_{i-1},t_i]}g^+(t), \quad (1 \leq i \leq n),\\
&(c_i^-,c_i^+)\cap (c_{i+1}^-,c_{i+1}^+)\neq \emptyset, \quad (1\leq i \leq n-1).
\end{align*}
For $1\leq i \leq n-1$, we can find $e^-_i$ and $e^+_i$ such that
\begin{align*}
c_i^-\vee c_{i+1}^- < e^-_i < e^+_i<c_i^+\wedge c_{i+1}^+.
\end{align*}
Then, we have
\begin{align*}
P(B^{a \to b} \in K(g^-, g^+)) 
&\geq \int_{e^-_1}^{e^+_1}  \int_{e^-_2}^{e^+_2} \cdots \int_{e^-_{n-1}}^{e^+_{n-1}} h(y_1,y_2, \cdots, y_{n-1})\\
&\qquad \times
\frac{n_{t_1}(y_1-a) n_{t_2-t_1}(y_2-y_1)\cdots n_{1-t_{n-1}}(b-y_{n-1})}{n_1(a,b)}dy_1 dy_2\cdots dy_{n-1},
\end{align*}
where
\begin{align*}
h(y_1,y_2, \cdots, y_{n-1})=&P\left(B_{[0,t_1]}^{a \to y_1} \in K_{[0,t_1]}(c_1^-, c_1^+)\right) 
P\left(B_{[t_1,t_2]}^{y_1 \to y_2} \in K_{[t_1,t_2]}(c_2^-, c_2^+)\right)\\
&\times \cdots \times P\left(B_{[t_{n-1},1]}^{y_{n-1} \to b} \in K_{[t_{n-1},1]}(c_n^-, c_n^+)\right) .
\end{align*}
By Lemma~\ref{Ap_Lem_Brwbridge_RectProb_Positive}, 
it holds that 
\begin{align*}
h(y_1,y_2, \cdots, y_{n-1})>0 \quad \mbox{on} \quad (y_1,y_2, \cdots , y_{n-1})\in [e^-_1, e^+_1]\times [e^-_2, e^+_2]\times \cdots \times [e^-_{n-1}, e^+_{n-1}]. 
\end{align*}
Therefore, we obtain our assertion. 
\end{proof}

\subsubsection{Proof of Corollary~\ref{Cor_Prob_BES3bridge_under_single_curve}}

Let $c=\min_{0 \leq u \leq 1} g(u)>0$. 
Using Theorem~6.1 in \cite{bib_Durret_1977} and the identity theorem, 
we have
\begin{align*}
P(W^+ \in K^-(g))&\geq P(W^+ \in K^-(c))>0, 
\end{align*}
and obtain $(A)$.
Let $b\in (0, g(1))$. 
Take $\delta \in (0, 1)$ such that 
\begin{align*}
c_1:=\min_{t\in [1-\delta, 1]}g(t) \geq \frac{1}{2} (g(1) + b)
\end{align*} 
holds. Let $t_0 = 1-\delta$ and $c_0=b\wedge \min_{0 \leq u \leq 1} g(u)>0$. 
Then, Lemmas~\ref{Ap_Lem_sameDist_BESbridge_BMbridge} and
\ref{Ap_Lem_Joint_Dist_BES_bridge_1} imply
\begin{align*}
P(r^{0 \to b} \in K^-(g)) &\geq 
P\left(r^{0 \to b} \in \pi_{[0,t_0]}^{-1}(K_{[0,t_0]}^-(c_0))
\cap \pi_{[t_0,1]}^{-1} (K_{[t_0,1]}^-(c_1))\right)
=\int_0^{c_0} \kappa_{t_0}(y) P\left( r^{0 \rightarrow b}(t_0) \in dy \right), 
\end{align*}
where
\begin{align*}
\kappa_{t_0}(y)=P\left(r_{[0,t_0]}^{0 \to y} \in K_{[0,t_0]}^-(c_0)\right)
\frac{P(B_{[t_0,1]}^{y \rightarrow b} \in K_{[t_0,1]}(0,c_1)) }
{P(B_{[t_0,1]}^{y \rightarrow b} \in K_{[t_0,1]}^+(0))}.
\end{align*}
Using Proposition~\ref{Prop_max_dist_of_BES3bridge} and 
Lemma~\ref{Ap_Lem_Brwbridge_TubeProb_Positive}, 
we have $\kappa_{t_0}(y)>0$ on $y \in (0, c_0)$, and obtain $(B)$. 

\subsection{Preparation for the proof of Corollary~\ref{Cor_Decomp_flat_Moving}}

Lemma~\ref{Ap_Lem_BM_Density_Diff_Formulae_2} is used to prove Corollary~\ref{Cor_Decomp_flat_Moving}.

\begin{lem}\label{Ap_Lem_BM_Density_Diff_Formulae_2}
Let $W = \{ W(t) \}_{t \geq 0}$ 
be the standard one-dimensional Brownian motion 
defined on $(\Omega, \mathcal{F}, P)$.  For $t, \varepsilon, \eta>0$ 
and $-\varepsilon \leq z \leq \eta$, we define
\begin{align*}
\psi_t(\varepsilon, \eta, z):=
P\left(W(t) \in dz, -\varepsilon \leq m_t(W) < M_t(W) \leq \eta \right) /dz.
\end{align*}
Then, we have
\begin{align}
&\lim_{\varepsilon \downarrow 0}
\frac{\partial}{\partial \varepsilon} 
\psi_t(\varepsilon, \eta+\varepsilon, \eta)=0, \quad
\lim_{\varepsilon \downarrow 0}
\frac{\partial^2}{\partial \varepsilon^2} 
\psi_t(\varepsilon, \eta+\varepsilon, \eta)
=2\overline{J}^{(\eta)}(t,\eta), \qquad \eta > 0,
\label{F4_delep_delep_psi} 
\end{align}
where $\overline{J}^{(\eta)}$ is defined in \eqref{Def_of_functions_J}.
\end{lem}
\begin{proof}
By (\ref{BM_Density_No2_Wt_mt_Mt}), 
derivatives of $\psi_t(\varepsilon, \eta+\varepsilon, \eta)$ satisfy
\begin{align*}
\frac{\partial}{\partial \varepsilon} \psi_t(\varepsilon, \eta+\varepsilon, \eta)
&= \sum_{k=-\infty}^{\infty} 
\left( 4k n_t'(\eta+2k(\eta +2\varepsilon )) 
- 2(2k+1)n_t'((2k+1)(\eta +2\varepsilon )) \right)  \\
&\to 
-2\sum_{k=-\infty}^{\infty} n_t'((2k+1)\eta)=0, 
\qquad \varepsilon \downarrow 0 ,
\end{align*}
and
\begin{align*}
\frac{\partial^2}{\partial \varepsilon^2} 
\psi_t(\varepsilon, \eta+\varepsilon, \eta)
&= \sum_{k=-\infty}^{\infty} 
\left( (4k)^2 n_t''(\eta+2k(\eta +2\varepsilon )) 
- 2^2(2k+1)^2n_t''((2k+1)(\eta +2\varepsilon )) \right) \\
&\to \sum_{k=-\infty}^{\infty} 
\left( (4k)^2 n_t''((2k+1)\eta) 
- 2^2(2k+1)^2n_t''((2k+1)\eta) \right)=:\widehat{\Phi},
\qquad \varepsilon \downarrow 0.
\end{align*}
Here, by $n_t''(z)=-n_t(z)/t+(z/t)^2n_t(z)$, it holds that
\begin{align*}
\widehat{\Phi}=-4 \sum_{k=-\infty}^{\infty} (4k+1)n_t''((2k+1)\eta) 
=4 \sum_{k=-\infty}^{\infty} (4k+1)
\left\{ \frac{1}{t}-\frac{(\eta+2k\eta)^2}{t^2}\right\}n_t((2k+1)\eta).
\end{align*}
Because we have
\begin{align*}
\sum_{k=-\infty}^{\infty} (2k+1)
\left\{ \frac{1}{t}-\frac{((2k+1)\eta)^2}{t^2}\right\}n_t((2k+1)\eta)=0,
\end{align*}
it follows that
\begin{align*}
\widehat{\Phi}=4 \sum_{k=-\infty}^{\infty} 2k
\left\{ \frac{1}{t}-\frac{(\eta+2k\eta)^2}{t^2}\right\}n_t((2k+1)\eta)
=2\overline{J}^{(\eta)}(t,\eta).
\end{align*}
\end{proof}

\subsection{Girsanov's theorem}
\label{subsection_Girsanov_appendix}


\begin{lem}\label{Ap_Lem_Girsanov_all_type}
Let $0\leq t_1<t_2\leq 1$.
Assume that $h^-$ and $h^+$ are $\mathbb{R}$-valued $C^2$-functions defined on $[t_1, t_2]$ 
satisfying the following condition:
\begin{align*}
\min_{t\in [t_1,t_2]}(h^+(t)-h^-(t))>0.
\end{align*} 
Then, for every $\mathbb{R}$-valued bounded continuous function $F$ on 
$C([t_1, t_2], \mathbb{R})$ and $h^-(t_1) < a < h^+(t_1)$, we have
\begin{align}
&E[F(a+W_{[t_1,t_2]})\ ;\ a+W_{[t_1,t_2]} \in K_{[t_1,t_2]}(h^-, h^+)]\nonumber \\
&\quad = 
E\Big[F\big(h^-+a-h^-(t_1)+W_{[t_1,t_2]}\big)\widetilde{Z}_{[t_1,t_2]}^{h^-} \big(W_{[t_1,t_2]}\big)^{-1} \ ;\ \nonumber \\
&\qquad \qquad 
W_{[t_1,t_2]} \in K_{[t_1,t_2]}\big(h^-(t_1)-a, h^+-h^- +h^-(t_1)-a\big)\Big]
\label{Girsanov_free_formula}
\end{align}
and
\begin{align}
&E[F(a+W_{[t_1,t_2]})\ ;\ a+W_{[t_1,t_2]}(t_2) \in db, 
a+W_{[t_1,t_2]} \in K_{[t_1,t_2]}(h^-, h^+)] \nonumber \\
&\quad = 
E\Big[F\big(h^-+a-h^-(t_1)+B^{0\to b-a+h^-(t_1)-h^-(t_2)}_{[t_1,t_2]}\big)
\widetilde{Z}_{[t_1,t_2]}^{h^-} \big(B^{0\to b-a+h^-(t_1)-h^-(t_2)}_{[t_1,t_2]}\big)^{-1} \ ;\ \nonumber \\
&\qquad \qquad 
B^{0\to b-a+h^-(t_1)-h^-(t_2)}_{[t_1,t_2]} \in K_{[t_1,t_2]}\big(h^-(t_1)-a, h^+-h^- +h^-(t_1)-a\big)\Big]\nonumber \\
&\qquad \times P\big(a+h^-(t_2)-h^-(t_1)+W_{[t_1,t_2]}(t_2) \in db\big)
\label{Girsanov_pinned_formula} \\
&\quad = 
E\Big[F\Big(h^++b -h^+(t_2)- \overset{\leftarrow}{B}^{0\to b-a+h^+(t_1)-h^+(t_2)}_{[t_1,t_2]}\Big)
\widetilde{Z}_{[t_1,t_2]}^{h^+(t_2)-\overset{\leftarrow}{h}^+} \big(B^{0\to b-a+h^+(t_1)-h^+(t_2)}_{[t_1,t_2]}\big)^{-1} 
\ ;\ \nonumber \\ 
&\qquad \qquad 
B^{0\to b-a+h^+(t_1)-h^+(t_2)}_{[t_1,t_2]}
\in K_{[t_1,t_2]}\big(b-h^+(t_2), \overset{\leftarrow}{h}^+-\overset{\leftarrow}{h}^-+b-h^+(t_2)\big)\Big]
\nonumber \\
&\qquad \times P\big(a+h^+(t_2)-h^+(t_1)+W_{[t_1,t_2]}(t_2) \in db\big).
\label{Girsanov_pinned_inv_formula}
\end{align}
\end{lem}
\begin{proof}
First, we prove \eqref{Girsanov_pinned_formula}. 
We define the process $\widehat{W}_{[t_1,t_2]}=\{ \widehat{W}_{[t_1,t_2]}(t)\}_{t\in [t_1,t_2]}$ as
\begin{align*}
\widehat{W}_{[t_1,t_2]}(t)=W_{[t_1,t_2]}(t)-\int_{t_1}^t \frac{d}{ds}h^-(s)ds \quad (t\in [t_1, t_2]).
\end{align*}
Further, we define the probability measure $\widehat{P}$ on 
$(\Omega, \mathcal{F}_{t_2}^{W_{[t_1,t_2]}})$ as
\begin{align*}
\widehat{P}(A)=E\left[ Z_{[t_1,t_2]}^{h^-} (W_{[t_1,t_2]}) ; A\right] 
=E\left[ \widetilde{Z}_{[t_1,t_2]}^{h^-} (\widehat{W}_{[t_1,t_2]}) ; A\right]\quad (A\in \mathcal{F}). 
\end{align*}
It follows from Girsanov's theorem that 
$\widehat{W}_{[t_1,t_2]}$ is a one-dimensional Brownian motion 
on the time interval $[t_1, t_2]$ under the probability measure $\widehat{P}$. 
Therefore, we obtain
\begin{align}
&E[F(a+W_{[t_1,t_2]})\ ;\ a+W_{[t_1,t_2]}(t_2) \in db, 
a+W_{[t_1,t_2]} \in K_{[t_1,t_2]}(h^-, h^+)]\nonumber \\
&\quad = 
E\Big[F\big(h^--h^-(t_2)+c^-_a+W_{[t_1,t_2]}\big)\widetilde{Z}_{[t_1,t_2]}^{h^-} \big(W_{[t_1,t_2]}\big)^{-1} \ ;\ \nonumber \\
&\qquad \qquad c^-_a+W_{[t_1,t_2]}(t_2) \in db,
\ c^-_a+W_{[t_1,t_2]} \in K_{[t_1,t_2]}\big(h^-(t_2), h^+-h^- +h^-(t_2)\big)\Big]
\nonumber \\
&\quad = 
E\Big[F\big(h^--h^-(t_2)+B^{c^-_a\to b}_{[t_1,t_2]}\big)\widetilde{Z}_{[t_1,t_2]}^{h^-} \big(-c^-_a+B^{c^-_a\to b}_{[t_1,t_2]}\big)^{-1} \ ;\ \nonumber \\
&\qquad \qquad 
B^{c^-_a\to b}_{[t_1,t_2]} \in K_{[t_1,t_2]}\big(h^-(t_2), h^+-h^- +h^-(t_2)\big)\Big]
P\big(c^-_a+W_{[t_1,t_2]}(t_2) \in db\big),
\label{Girs_pinned_formula_proof_step1}
\end{align}
where $c^-_a=a+h^-(t_2)-h^-(t_1)$. 
Then, it follows from \eqref{Girs_pinned_formula_proof_step1} and 
$-c^-_a+B^{c^-_a\to b}_{[t_1,t_2]}(\cdot) 
\overset{\mathcal{D}}{=} B^{0\to b-c^-_a}_{[t_1,t_2]}(\cdot) $ 
that
\begin{align}
&E[F(a+W_{[t_1,t_2]})\ ;\ a+W_{[t_1,t_2]}(t_2) \in db, 
a+W_{[t_1,t_2]} \in K_{[t_1,t_2]}(h^-, h^+)]\nonumber \\
&\quad = 
E\Big[F\big(h^--h^-(t_2)+c^-_a-c^-_a+B^{c^-_a\to b}_{[t_1,t_2]}\big)\widetilde{Z}_{[t_1,t_2]}^{h^-} \big(-c^-_a+B^{c^-_a\to b}_{[t_1,t_2]}\big)^{-1} \ ;\ \nonumber \\
&\qquad \qquad 
c^-_a-c^-_a+B^{c^-_a\to b}_{[t_1,t_2]} \in K_{[t_1,t_2]}\big(h^-(t_2), h^+-h^- +h^-(t_2)\big)\Big]
P\big(c^-_a+W_{[t_1,t_2]}(t_2) \in db\big)
\nonumber \\
&\quad = 
E\Big[F\big(h^--h^-(t_2)+c^-_a+B^{0\to b-c^-_a}_{[t_1,t_2]}\big)\widetilde{Z}_{[t_1,t_2]}^{h^-} \big(B^{0\to b-c^-_a}_{[t_1,t_2]}\big)^{-1} \ ;\ \nonumber \\
&\qquad \qquad 
c^-_a+B^{0\to b-c^-_a}_{[t_1,t_2]} \in K_{[t_1,t_2]}\big(h^-(t_2), h^+-h^- +h^-(t_2)\big)\Big]
P\big(c^-_a+W_{[t_1,t_2]}(t_2) \in db\big).
\label{Girs_pinned_formula_proof_step2}
\end{align}
Thus, we obtain \eqref{Girsanov_pinned_formula} by \eqref{Girs_pinned_formula_proof_step2}. 

Second, we can obtain \eqref{Girsanov_free_formula} in a similar manner to the proof of \eqref{Girs_pinned_formula_proof_step1}. 

Finally, we prove \eqref{Girsanov_pinned_inv_formula}. Using
\begin{align*}
&(W_{[t_1,t_2]}(\cdot), W_{[t_1,t_2]}(t_2)) \\
&\overset{\mathcal{D}}{=} 
(W_{[t_1,t_2]}(t_2) - W_{[t_1,t_2]}(t_1+t_2-\cdot), W_{[t_1,t_2]}(t_2) - W_{[t_1,t_2]}(t_1+t_2-t_2))\\
&=(W_{[t_1,t_2]}(t_2) - \overset{\leftarrow}{W}_{[t_1,t_2]}(\cdot), W_{[t_1,t_2]}(t_2)), 
\end{align*} 
we obtain
\begin{align}
&E\Big[F(a+W_{[t_1,t_2]})1_{K_{[t_1,t_2]}(h^-, h^+)}(a+W_{[t_1,t_2]})
\ ;\ a+W_{[t_1,t_2]}(t_2) \in A \Big] \nonumber \\
&\quad =E\Big[F\big(a+W_{[t_1,t_2]}(t_2) - \overset{\leftarrow}{W}_{[t_1,t_2]}\big)
1_{K_{[t_1,t_2]}(h^-, h^+)}\big(a+W_{[t_1,t_2]}(t_2) - \overset{\leftarrow}{W}_{[t_1,t_2]}\big)
\ ;\ \nonumber \\
&\qquad \qquad \ a+W_{[t_1,t_2]}(t_2) \in A \Big] \nonumber \\
&\quad = 
\int_A E\Big[F\big(b - \overset{\leftarrow}{W}_{[t_1,t_2]}\big)
1_{K_{[t_1,t_2]}(h^-, h^+)}\big(b - \overset{\leftarrow}{W}_{[t_1,t_2]}\big)
\ ;\ a+W_{[t_1,t_2]}(t_2) \in db\Big] \nonumber \\
&\quad = 
\int_A E\Big[F\big(b - \overset{\leftarrow}{W}_{[t_1,t_2]}\big)
1_{K_{[t_1,t_2]}(b-\overset{\leftarrow}{h}^+, b-\overset{\leftarrow}{h}^-)}(W_{[t_1,t_2]})
\ ;\ a+W_{[t_1,t_2]}(t_2) \in db\Big] \label{pinned_inv_Girs_step1}
\end{align}
for $A\in \mathcal{B}(\mathbb{R})$. 
Let $\varepsilon^+_b = h^+(t_2)-b$. Then, it follows from \eqref{pinned_inv_Girs_step1} that
\begin{align}
&E[F(a+W_{[t_1,t_2]})\ ;\ a+W_{[t_1,t_2]}(t_2) \in db, 
a+W_{[t_1,t_2]} \in K_{[t_1,t_2]}(h^-, h^+)] \nonumber \\
&\quad =E\Big[F\big(b - \overset{\leftarrow}{W}_{[t_1,t_2]}\big)
\ ;\ a+W_{[t_1,t_2]}(t_2) \in db, 
W_{[t_1,t_2]}\in K_{[t_1,t_2]}\big(b-\overset{\leftarrow}{h}^+, b-\overset{\leftarrow}{h}^-\big)\Big]
\nonumber \\
&\quad =E\Big[F\big(b - \overset{\leftarrow}{W}_{[t_1,t_2]}\big)
\ ;\ a+W_{[t_1,t_2]}(t_2) \in db, \nonumber \\
&\qquad \qquad \qquad \qquad \qquad
\overset{\leftarrow}{h}^+-h^+(t_2)+W_{[t_1,t_2]}
\in K_{[t_1,t_2]}\big(-\varepsilon^+_b, b-\overset{\leftarrow}{h}^-+\overset{\leftarrow}{h}^+-h^+(t_2)\big)\Big]
\nonumber \\
&\quad =E\Big[F\big(b - \overset{\leftarrow}{W}_{[t_1,t_2]}\big)
\ ;\ c^+_a+\overset{\leftarrow}{h}^+(t_2)-\overset{\leftarrow}{h}^+(t_1)+W_{[t_1,t_2]}(t_2) \in db, 
\nonumber \\
&\qquad \qquad \qquad \qquad \qquad
\overset{\leftarrow}{h}^+-\overset{\leftarrow}{h}^+(t_1)+W_{[t_1,t_2]}
\in K_{[t_1,t_2]}\big(-\varepsilon^+_b, \overset{\leftarrow}{h}^+-\overset{\leftarrow}{h}^-+b-h^+(t_2)\big)\Big], 
\label{pinned_inv_Girs_step2}
\end{align}
where $c^+_a=a+h^+(t_2)-h^+(t_1)$. 
Combining \eqref{pinned_inv_Girs_step2}, Girsanov's theorem, and 
$-c^+_a+B^{c^+_a\to b}_{[t_1,t_2]}(\cdot) 
\overset{\mathcal{D}}{=} B^{0\to b-c^+_a}_{[t_1,t_2]}(\cdot) $, we obtain
\begin{align}
&E[F(a+W_{[t_1,t_2]})\ ;\ a+W_{[t_1,t_2]}(t_2) \in db, 
a+W_{[t_1,t_2]} \in K_{[t_1,t_2]}(h^-, h^+)] \nonumber \\
&\quad =E\Big[F\big(h^++b -h^+(t_2)- \overset{\leftarrow}{W}_{[t_1,t_2]}\big)
\widetilde{Z}_{[t_1,t_2]}^{h^+(t_2)-\overset{\leftarrow}{h}^+} \big(W_{[t_1,t_2]}\big)^{-1} 
\ ;\ \nonumber \\ 
&\qquad \qquad 
c^+_a+W_{[t_1,t_2]}(t_2) \in db, \ 
W_{[t_1,t_2]}\in K_{[t_1,t_2]}\big(-\varepsilon^+_b, \overset{\leftarrow}{h}^+-\overset{\leftarrow}{h}^-+b-h^+(t_2)\big)\Big]\nonumber \\
&\quad =E\Big[F\big(h^++b -h^+(t_2)- \overset{\leftarrow}{W}_{[t_1,t_2]}\big)
\widetilde{Z}_{[t_1,t_2]}^{h^+(t_2)-\overset{\leftarrow}{h}^+} \big(-c^+_a+c^+_a+W_{[t_1,t_2]}\big)^{-1} 
\ ;\ \nonumber \\ 
&\qquad \qquad 
c^+_a+W_{[t_1,t_2]}(t_2) \in db, \ 
-c^+_a+c^+_a+W_{[t_1,t_2]}\in K_{[t_1,t_2]}\big(-\varepsilon^+_b, \overset{\leftarrow}{h}^+-\overset{\leftarrow}{h}^-+b-h^+(t_2)\big)\Big]\nonumber \\
&\quad =E\Big[F\big(h^++b -h^+(t_2)- \overset{\leftarrow}{B}^{0\to b-c^+_a}_{[t_1,t_2]}\big)
\widetilde{Z}_{[t_1,t_2]}^{h^+(t_2)-\overset{\leftarrow}{h}^+} \big(B^{0\to b-c^+_a}_{[t_1,t_2]}\big)^{-1} 
\ ;\ \nonumber \\ 
&\qquad \qquad 
B^{0\to b-c^+_a}_{[t_1,t_2]}\in K_{[t_1,t_2]}\big(-\varepsilon^+_b, \overset{\leftarrow}{h}^+-\overset{\leftarrow}{h}^-+b-h^+(t_2)\big)\Big]
P(c^+_a+W_{[t_1,t_2]}(t_2) \in db).
\label{Girsanov_pinned_inv_modified_formula}
\end{align}
Therefore, we obtain \eqref{Girsanov_pinned_inv_formula} by \eqref{Girsanov_pinned_inv_modified_formula}.
\end{proof}

\subsection{Preparation for 
Sections~\ref{section_proof_main_Lemma}, 
\ref{section_proof_Moving_between_2curves}, 
\ref{section_proof_Meander_between_2curves}, and
\ref{section_proof_3dBesselbridge_between_2curves}}

The following lemma is obtained by Skorohod's theorem.
\begin{lem}\label{Ap_Lem_conv_restricted_Expectation}
Let $S$ be a Polish space. 
Let $X_n$ and $X$ be random variables defined on 
$(\Omega_n, \mathcal{F}_n, P_n)$ 
and $(\Omega, \mathcal{F}, P)$ that take values in $S$. 
Assume that $X_n \xrightarrow{\mathcal{D}} X$ holds and 
$A \in \mathcal{B}(S)$ satisfies $P(X \in \partial A)=0$. 
Then, for every $\mathbb{R}$-valued bounded continuous function $G$ on $S$, we have
\begin{align*}
\lim_{n \to \infty} E_n[G(X_n)\ ;\ X_n \in A] = E[G(X)\ ;\ X \in A], 
\end{align*}
where $E_n$ denotes the expectation operator under $P_n$.
\end{lem}

Proposition~\ref{Prop_Boundary_purturb_WeakConv_for_Meander_BES3bridge} is used in 
Sections~\ref{section_proof_main_Lemma}, 
\ref{section_proof_Moving_between_2curves}, 
\ref{section_proof_Meander_between_2curves}, and
\ref{section_proof_3dBesselbridge_between_2curves}.

\begin{prop}\label{Prop_Boundary_purturb_WeakConv_for_Meander_BES3bridge}
Let $g$ be an $\mathbb{R}$-valued $C^1$-function defined on $[0,1]$ that satisfies $\min_{0 \leq t \leq 1} g(t) > 0$.  
Assume that $\{ \eta(\varepsilon) \}_{\varepsilon>0}$ satisfies
\begin{align*}
\eta (\varepsilon)\geq 0\quad (\varepsilon>0)
\quad \mbox{and} \quad \eta(\varepsilon) \downarrow 0\quad (\varepsilon \downarrow 0).
\end{align*}
Then, we have
\begin{align}
\lim_{\varepsilon \downarrow 0}P(W |_{K^+(-\varepsilon)}\in K^-(g+ \eta(\varepsilon)))&=P(W^+ \in K^-(g)), 
\label{Prop_Boundary_purturb_ProbConv_for_Meander}\\
\lim_{\varepsilon \downarrow 0}P(B^{0\to b} |_{K^+(-\varepsilon)}\in K^-(g+ \eta(\varepsilon)))&=P(r^{0\to b} \in K^-(g)) \quad (0\leq b<g(1)).
\label{Prop_Boundary_purturb_ProbConv_for_BES3bridge}
\end{align}
Further, for every $\mathbb{R}$-valued bounded continuous function $F$ on $C([0,1],\mathbb{R})$, we have
\begin{align}
\lim_{\varepsilon \downarrow 0}E[F(W |_{K(- \varepsilon, g + \eta(\varepsilon))})]&=E[F(W^+ |_{K^-(g)})], 
\label{Prop_Boundary_purturb_WeakConv_for_Meander}\\
\lim_{\varepsilon \downarrow 0}E[F(B^{0\to b} |_{K(- \varepsilon, g + \eta(\varepsilon))})]&=E[F(r^{0\to b} |_{K^-(g)})] \quad (0\leq b<g(1)).
\label{Prop_Boundary_purturb_WeakConv_for_BES3bridge}
\end{align}
\end{prop}
\begin{proof}
Combining 
Lemma~\ref{Lem_Boundary_Prob_for_Meander_BES3bridge} 
and the fact that $W |_{K^+(-\varepsilon)}$ converges weakly to $W^+$ (\cite{bib_Durret_1977}), we obtain 
\begin{align}
&P(W^+ \in K^-(g))
=\lim_{\varepsilon \downarrow 0}P(W |_{K^+(-\varepsilon)}\in K^-(g))
\leq \varliminf_{\varepsilon \downarrow 0}P(W |_{K^+(-\varepsilon)}\in K^-(g+ \eta(\varepsilon))), 
\label{ineq1_prob_liminf_meander}\\
&\varlimsup_{\varepsilon \downarrow 0}P(W |_{K^+(-\varepsilon)}\in K^-(g+ \eta(\varepsilon)))
\leq \lim_{\varepsilon \downarrow 0}P(W |_{K^+(-\varepsilon)}\in K^-(g+ \delta))
= P(W^+ \in K^-(g+\delta))\quad (\delta>0).
\label{ineq2_prob_liminf_meander}
\end{align}
Then, it follows from 
\eqref{ineq1_prob_liminf_meander}, \eqref{ineq2_prob_liminf_meander} and $\bigcap_{\delta>0}K^-(g+ \delta )=K^-(g)$ that 
\eqref{Prop_Boundary_purturb_ProbConv_for_Meander} holds. 
Similarly, combining Lemma~\ref{Lem_Boundary_Prob_for_Meander_BES3bridge} and 
Theorem~\ref{Th_conv_Brownianbridge_to_3dBesselbridge}, 
we can also deduce that \eqref{Prop_Boundary_purturb_ProbConv_for_BES3bridge} holds. 

Because $W |_{K^+(-\varepsilon)}$ converges weakly to $W^+$ (\cite{bib_Durret_1977}), 
Lemmas~\ref{Ap_Lem_conv_restricted_Expectation} and \ref{Lem_Boundary_Prob_for_Meander_BES3bridge}
imply that 
\begin{align}
E[F(W |_{K(- \varepsilon, g)})]
&=\frac{E[F(W |_{K^+(-\varepsilon)}) \ ;\ W |_{K^+(-\varepsilon)} \in K^-(g)]}{P(W |_{K^+(-\varepsilon)} \in K^-(g))} 
\nonumber \\
&\to \frac{E[F(W^+) \ ;\ W^+\in K^-(g)]}{P(W^+ \in K^-(g))}=
E[F(W^+ |_{K^-(g)})]\qquad (\varepsilon \downarrow 0). 
\label{EQ_Pre_Result}
\end{align}
On the other hand, because we have
\begin{align*}
\varDelta(\varepsilon)
\coloneqq &
\vert E[F(W |_{K(- \varepsilon, g+ \eta(\varepsilon))})] - E[F(W |_{K(- \varepsilon, g)})] \vert \\
\leq &\left\vert 
\frac{E[F(W |_{K^+(-\varepsilon)}) \ ;\ W |_{K^+(-\varepsilon)} \in K^-(g+ \eta(\varepsilon))\setminus K^-(g)]}
{P(W |_{K^+(-\varepsilon)} \in K^-(g+ \eta(\varepsilon)))} 
\right\vert \\ 
&+ \left\vert \dfrac{E[F(W |_{K^+(-\varepsilon)}) \ ;\ W |_{K^+(-\varepsilon)} \in K^-(g)]}{P(W |_{K^+(-\varepsilon)} \in K^-(g))}
\left( \frac{P(W |_{K^+(-\varepsilon)} \in K^-(g))}{P(W |_{K^+(-\varepsilon)} \in K^-(g+ \eta(\varepsilon)))} - 1 \right) \right\vert \\
\leq & 2\|F\|_{\infty}
\frac{P(W |_{K^+(-\varepsilon)}\in K^-(g+ \eta(\varepsilon)))-P(W |_{K^+(-\varepsilon)}\in K^-(g))}{P(W |_{K^+(-\varepsilon)} \in K^-(g))} 
\qquad (\varepsilon>0)
\end{align*}
for $\| F \|_{\infty} := \sup_{w\in C([0,1],\mathbb{R})} \vert F(w)\vert$, 
we can deduce that
\begin{align}
\varlimsup_{\varepsilon \downarrow 0}
\left\vert \varDelta(\varepsilon) \right\vert 
&\leq 2\|F\|_{\infty}
\frac{P(W^+\in K^-(g))-P(W^+ \in K^-(g))}{P(W^+ \in K^-(g))} =0
\label{varlimsupvarDelta_eq_zero}
\end{align}
by \eqref{Prop_Boundary_purturb_ProbConv_for_Meander} and Corollary~\ref{Cor_Prob_BES3bridge_under_single_curve}. 
Therefore, \eqref{EQ_Pre_Result} and \eqref{varlimsupvarDelta_eq_zero} imply 
\eqref{Prop_Boundary_purturb_WeakConv_for_Meander}. 
Similarly, combining 
Lemmas~\ref{Ap_Lem_conv_restricted_Expectation}, \ref{Lem_Boundary_Prob_for_Meander_BES3bridge}
and Theorem~\ref{Th_conv_Brownianbridge_to_3dBesselbridge}, 
we can deduce that \eqref{Prop_Boundary_purturb_WeakConv_for_BES3bridge} holds.
\end{proof}

\subsection{Preparation for the proof of Theorem~\ref{Thm_abs_conti}}
\label{subsection_proof_Preparation_for_absconti}

\begin{lem}\label{Lem_abs_conti_B_motion_and_B_bridge}
Let $a, c \in \mathbb{R}$. For $t\in(0, 1)$, we have
\begin{align*}
\frac{d\left(P\circ (\pi_{[0,t]}\circ B^{a\to c})^{-1}\right)}
{d\left(P\circ (\pi_{[0,t]}\circ (a+W))^{-1}\right)}(w)
=\frac{n_{1-t}(w(t)-c)}{n_1(a-c)}, \quad w\in C([0, t], \mathbb{R}).
\end{align*}
\end{lem}
\begin{proof}
In this proof, let $P^X$ denote 
the measure induced by a continuous process $X=\{X(t)\}_{t\in[0, 1]}$. 
In addition, for a continuous process $X=\{X(t)\}_{t\in[0, 1]}$, 
we write the expectation with respect to the probability $P^X$ as $E^X$.
Let $A\in \mathcal{B}(C([0, t],\mathbb{R}))$ be fixed. 
By the Markov property of $a+W$, we obtain the assertion as follows: 
\begin{align*}
P\left(\pi_{[0,t]}\circ B^{a\to c} \in A\right)
&=\frac{P^{a+W}\left(\pi_{[0, t]}^{-1}(A), w(1)\in dc\right)}{P^{a+W}\left(w(1)\in dc\right)}\\
&=\frac{E^{a+W}\left[1_{\pi_{[0, t]}^{-1}(A)}(w)\cdot P^{a+W}\left(w(1)\in dc~|~w(t)\right)\right]}{P^{a+W}\left(w(1)\in dc\right)}\\
&=\int_{\pi_{[0, t]}^{-1}(A)}\frac{P^{a+W}\left(w(1)\in dc~|~w(t)\right)}{P^{a+W}\left(w(1)\in dc\right)}P^{a+W}\left(dw\right)\\
&=\int_{A}\frac{n_{1-t}(w(t)-c)}{n_1(a-c)}P\left( \pi_{[0,t]}\circ (a+W)\in dw\right).
\end{align*}
\end{proof}

\section*{Acknowledgments}

The authors would like to thank
Prof.\ Kumiko Hattori (Tokyo Metropolitan University), 
Prof.\ Ryozo Miura (Hitotsubashi University), 
Prof.\ Toshihiro Yamada (Hitotsubashi University),  
Prof.\ Masaaki Fukasawa (Osaka University), and 
Prof.\ Tomonori Nakatsu (Shibaura Institute of Technology) 
for their helpful comments and discussions on the subject matter. 
We also thank Editage (www.editage.com) for English language editing. 
This study was supported by a JSPS KAKENHI grant (JP22K01556).
Finally, we would like to express our sincere gratitude to the anonymous reviewers 
for their valuable comments and suggestions, which have significantly contributed to improving the quality of this manuscript.


\end{document}